\documentclass[10pt, reqno,final]{amsart}
\usepackage[notcite,notref]{showkeys}
\usepackage{url,hyperref,multirow}
\usepackage{color}
\usepackage{epsfig,subfigure,amssymb,amsmath,version}
\usepackage{amssymb,version,graphicx,fancybox,mathrsfs,pifont,booktabs}
\usepackage{epstopdf}
\usepackage{cases}

\catcode`\@=11 \theoremstyle{plain}

\newtheorem{lemma}{Lemma}[section]

\newtheorem{proposition}{Proposition}[section]

\newtheorem{remark}{Remark}[section]

\@addtoreset{figure}{section}
\renewcommand\thefigure{\thesection.\@arabic\c@figure}
\@addtoreset{table}{section}
\renewcommand\thetable{\thesection.\@arabic\c@table}

\newcommand{\bs}[1]{\boldsymbol{#1}}
\def \ri {{\rm i}}

\DeclareSymbolFont{ugmL}{OMX}{mdugm}{m}{n}
\SetSymbolFont{ugmL}{bold}{OMX}{mdugm}{b}{n}

\DeclareMathAccent{\wideparen}{\mathord}{ugmL}{"F3}

\title[Calculation of SPH and VSH Expansions]{accurate calculation of spherical and vector spherical harmonic expansions via spectral element grids}
\author[B. Wang, \; L. Wang \;  $\&$ \; Z. Xie] {Bo Wang$^{1,2}$, \; Li-Lian Wang$^{2}$ \;       and \;  Ziqing  Xie$^{1}$}
\thanks{\noindent $^{1}$Key Laboratory of High Performance Computing and Stochastic Information Processing (HPCSIP) (Ministry of Education of China), College of Mathematics and Computer Science, Hunan Normal University, 410081, P. R. China. The research of the first author is partially supported by NSFC Grants (11341002 and 11401206), the Construct Program of the Key Discipline in Hunan Province and a Scientific Research Fund of Hunan Provincial Education Department (No. 16B154). The third author is partially supported by NSFC (91430107 and 11171104) and the Construct Program of the Key Discipline in Hunan Province.\\
\indent $^{2}$Division of Mathematical Sciences, School of Physical
and Mathematical Sciences,  Nanyang Technological University,
637371, Singapore. The research of the first two authors is  supported by a Singapore  MOE AcRF Tier 2 Grant (MOE 2013-T2-1-095, ARC 44/13),  and  Singapore MOE AcRF Tier 1 Grant (RG 27/15).
}

\begin{document}
\bibliographystyle{plain}
\maketitle

\vspace*{-12pt}
\begin{abstract}
We present in this paper a spectrally accurate numerical method  for  computing the spherical/vector spherical harmonic expansion of a  function/vector field with given (elemental) nodal values on a spherical surface. Built upon suitable analytic formulas for dealing with the involved highly oscillatory integrands, the method is robust for  high mode expansions. We apply the numerical method to the simulation of three-dimensional acoustic and electromagnetic multiple scattering problems.
Various numerical evidences show that the high  accuracy can be achieved within reasonable computational time.
This also paves the way for spectral-element discretization of  3D scattering problems reduced by spherical  transparent boundary conditions based on the Dirichlet-to-Neumann map.   
\end{abstract}

\section{Introduction}
Many applications in e.g.,
geophysics, weather and climate modelling, and optical and electromagnetic materials, involve computations over spherical surfaces, where  various scenarios require to calculate  the expansions of scalars or vector fields in spherical harmonics (SPH) or vector spherical harmonics (VSH).  Conventionally,  the associated   surface integrals on $\mathbb S^2=\{(\theta,\varphi) :  \theta \in [0,\pi],\, \varphi\in [0,2\pi)\},$ are evaluated by numerical quadrature on global ``tensor product" of equispaced grids (\cite{rokhlin2006fast,Swa.S00}):
\begin{equation}
\label{intnodephi}
\theta_j=\frac{\pi(j+1/2)}{J+1},\quad 1\leq j\leq J;\;\;\; \varphi_k=\frac{2\pi(k+1/2)}{2K+1},\quad 1\leq k\leq K.
\end{equation}
The number of points in two directions is generally determined by the so-called alias-free condition: $J\geq (3L+1)/2$ and $ K\geq 3L+1,$ where $L$ is the cut-off frequency or the highest expansion mode.
 The computational complexity of direct computation of the discrete spherical harmonics transform is $O(L^3).$ 
 Nevertheless, fast algorithms were developed based on the FFT and new fast associated Legendre transform \cite{Swa.S00,rokhlin2006fast, suda2002fast}.  The related spectral transform method has been widely used
in simulations of partial differential equations (PDEs) on the sphere (see \cite{Boyd01} and the references therein), after all the spherical harmonics are eigenfunctions of the Laplace-Beltrami operator. However,   it is noteworthy that the VSH  spectral transform method has been less studied.



Various partitions  of the spherical surfaces  and local element methods  have been proposed to  reduce the cost of nonlocal communications between global transforms and facilitate the parallel  implementation.   These  include the icosahedral, hexahedral, and cubed-sphere grids   (cf. \cite{giraldo2004scalable}), among which the   cubed-sphere
 partition of the sphere introduced by Sadourny
\cite{Sadourny} meets the desirable features for a ``good" partition  in the sense of
Phillips \cite{phillips1957coordinate}. In practice, spectral element methods with such cartesian coordinate-based partitions  provide powerful tools in weather prediction and climate modelling and simulations.

In this  paper, we are  concerned with high-mode SPH and VSH expansions, where  the scalar and vector fields of interest  are given nodal values resulted from a
(conforming or non-conforming) spectral element discretization inside or exterior to the spherical surface.  One important application is related to the acoustic and electromagnetic wave scattering with high wavenumber using spherical transparent (or nonreflecting) boundary condition (TBC)  (or NRBC) for domain reduction.
To fix the idea,  let $\mathbb{S}_h^2:=\{S^e\}_{e=1}^E$ be a generic  non-overlapping   partition (cf. Figure \ref{sphericalcoordinates}) of the sphere ${\mathbb S^2}$,  
and  consider the spectral element approximations of unknown function $u$ and vector field $\bs v$ projected upon (or restricted to)  ${\mathbb S}^2$ given by
\begin{equation}
\label{spectralelementapprox}
u_N^E(\theta,\varphi)=\sum\limits_{i=0}^N\sum\limits_{j=0}^Nu_{ij}^{e}\psi_{ij}(\mathcal{F}_e^{-1}(\theta,\varphi)),\quad
\bs v_N^E(\theta,\varphi)=\sum\limits_{i=0}^N\sum\limits_{j=0}^N
\bs v_{ij}^{e}\psi_{ij}(\mathcal{F}_e^{-1}(\theta,\varphi))\;\;\hbox{on}\;\;S^e,
\end{equation}
where  $\{u_{ij}^e\}_{i,j=0}^N$ and $\{\bs v_{ij}^e\}_{i,j=0}^N$ are nodal values on the spectral-element grid on $S^e$, $\mathcal{F}_e$ is
the elemental mapping from the reference square $Q:=[-1, 1]\times [-1, 1]$ to  $S^e,$ and $\{\psi_{ij}\}$ are the spectral-element basis functions. We intend to evaluate
\begin{equation}\label{expan1}
\begin{split}
&\tilde a_l^m=\int_{\mathbb{S}^2}u_N^E(\theta,\varphi)\overline{Y_l^m(\theta,\varphi)}dS,\quad {\rm and}\\
\tilde v_{lm}^r=\int_{\mathbb{S}^2}{\bs v}_N^E\cdot \overline{\bs Y_l^{m}}dS,\;\;
&\tilde v_{lm}^{(1)}=\varpi_l^{-1}\int_{\mathbb{S}^2}{\bs v}_N^E\cdot \overline{\bs\Psi_l^{m}}dS,
\;\;  \tilde v_{lm}^{(2)}=\varpi_l^{-1} \int_{\mathbb{S}^2}{\bs v}_N^E\cdot \overline{\bs\Phi_l^{m}}dS,
\end{split}
\end{equation}
where $\varpi_l=l(l+1)$ and $\{\bs Y_l^m, \bs\Psi_l^m, \bs\Phi_l^m\}_{0\le |m|\le l}$ are VSH defined in \eqref{vecylm}-\eqref{vecphilm} (cf. \cite{Swa.S00}).
%
The significant challenges reside in the following aspects.
\begin{enumerate}
\item[(i)] The spectral element approximation $u_N^E$ and $\bs v_N^E$ are piecewise smooth (i.e., only in $C^0(\mathbb{S}^2)$), 
    so the interplay between global quadrature points (cf. \eqref{intnodephi})   and spectral element grids via interpolation and  fast transformation algorithm has only  a first-order convergence.
\item[(ii)] When the SPH/VSH mode is high (i.e., $l$ is large),  the SPH and VSH (so the integrands) become  highly oscillatory.
On the other hand,  the elemental mapping (e.g., for the cubed-sphere grids) significantly complicates the integrand when one computes the integrals in the reference element $Q.$ Thus, the naive use of a composite rule with partition coherent to the spectral element partition $\mathbb{S}_h^2$ can be applied to produce high accuracy, but one has to use a large number  of points.
\end{enumerate}

To overcome these difficulties, we employ  a suitable partition of the spherical surface that
allows for using some analytic formulas to evaluate the above integrals exactly element by element.
 Moreover, the double integrals can be calculated by iterated integrals.  These are  essential for the efficiency of the algorithm, and effectiveness for high-mode expansions.   This approach  does not induce any additional error that exceeds accumulated computer round-off error by using these analytical formulas.
It is noteworthy that Fournier \cite{fournier2005exact} proposed a method for calculating  global Fourier coefficients
 for given nodal values on  non-conforming two-dimensional spectral elements.  Yang et al.  \cite{Yang2015}  introduced a  new elemental mapping for conforming (curvilinear) spectral elements in simulations of polygonal invisibility cloaks involving time-harmonic Maxwell equations with non-local circular Dirichlet-to-Neumann (DtN) boundary conditions.
A very efficient semi-analytic approach was proposed to compute Fourier coefficients via two-dimensional spectral element grids  therein.  The algorithm herein is  actually spawned by \cite{Yang2015}, and aims at paving the way for simulating three-dimensional acoustic and electromagnetic wave scattering with bounded scatterers using spherical DtN TBCs.
Accurate and rapid evaluation of integrals  in \eqref{expan1} on spectral element grids becomes  an exceedingly important ingredient for the whole algorithm (which we shall report  in a future work).

With the accurate tool at our disposal, we consider three-dimensional acoustic and electromagnetic wave scattering in several scenarios where the incident waves are given by high-order spectral element approximations. By using the separation variable method together with our high mode SPH and VSH expansion algorithms for spectral element approximations, we obtain very accurate approximations for high frequency scattering waves from a single spherical scatterer. Besides, we  provide a very stable means  to compute the ratio of spherical Hankel functions with a wide range of arguments. An algorithm consists of the proposed SPH expansion formulation and transformed field expansion (TFE) method (cf.  \cite{nicholls2006a,fang2007a}) is discussed for solving scattering problem with irregular scatterer. This idea can also be extended to electromagnetic scattering problem by using VSH expansion. The SPH and VSH expansions play an  important part in solving multiple scattering problems as well. The simulation becomes much  more challenging compared with the single scattering problem,  due to the unavoidable interactions between the scattering  waves from  different scatterers.  Much effort has been devoted to 2D problems (see, e.g.,  \cite{acosta2010coupling,young1975multiple,grote2004dirichlet,amirkulova2015acoustic} and the references therein). However,  there has been limited  study for the 3D case \cite{demkowicz2006few, jiang2012adaptive,huang2010two}. Here, we consider the case involving  multiple spherical scatterers where the incident wave is given by spectral element approximation. Unlike the 2D problem, the separation matrices in 3D case take much more complicated  form (cf. \cite{martin2006multiple}). Based on the recurrence formula presented in \cite{chew1992recurrence,gumerov2004recursions}, we derive a stable recurrence formula for the computation of the expansion coefficients of the scaled out-going wave functions.  Then the expansion coefficients for purely outgoing waves can be obtained by solving a linear system resulted by matching boundary data on the scatterers. The proposed algorithm actually provides  an accurate way to compute the far-field scattering waves from local element based approximations \cite{garate2012solution,song2013multi}. Besides,  they can be combined with some appropriate solvers \cite{acosta2010coupling} to tackle practical 3D problems involving scatterers with complex shapes.


The rest of the paper is organized as follows. In Section \ref{sect2:pre}, we collect some relevant properties of   spherical harmonics and vector spherical harmonics. We then present the algorithms for computing  the spherical harmonic coefficients and vector spherical harmonic coefficients in Section \ref{sect3:expan}. Some interesting applications in 3D multiple scattering problems are discussed in Section \ref{sect4:appl}. Ample numerical results are provided to validate the accuracy of the proposed algorithms.

\section{Preliminaries}\label{sect2:pre}

\subsection{Spherical harmonics}
The spherical coordinate of a given point $\bs x=(x, y, z)$ in $\mathbb{R}^3$ is represented by 
\begin{equation}\label{sphco}
x=r\sin \theta \cos \varphi,\quad y=r\sin\theta \sin
\varphi,\quad z=r\cos\theta,
\end{equation}
where $r\ge 0$, $\theta \in [0,\pi]$ and $\varphi \in [0,2\pi)$ stand for the radial distance, polar angle and azimuthal angle (see Figure \ref{sphericalcoordinates}).
The unit coordinate vectors in spherical coordinate system are denoted by $\{\bs e_r,\bs e_{\theta}, \bs e_{\varphi}\}$ which form a right hand coordinate system, i.e.,
\begin{equation}\label{eretheta}
\bs e_r\times \bs e_{\theta}=\bs e_{\varphi},\quad \bs e_{\theta}\times \bs e_{\varphi}=\bs e_r,\quad
\bs e_{\varphi}\times \bs e_r= \bs e_{\theta}.
\end{equation}
\begin{figure}[htbp]
\begin{center}
 \subfigure[Spherical coordinates]{ \includegraphics[scale=.19]{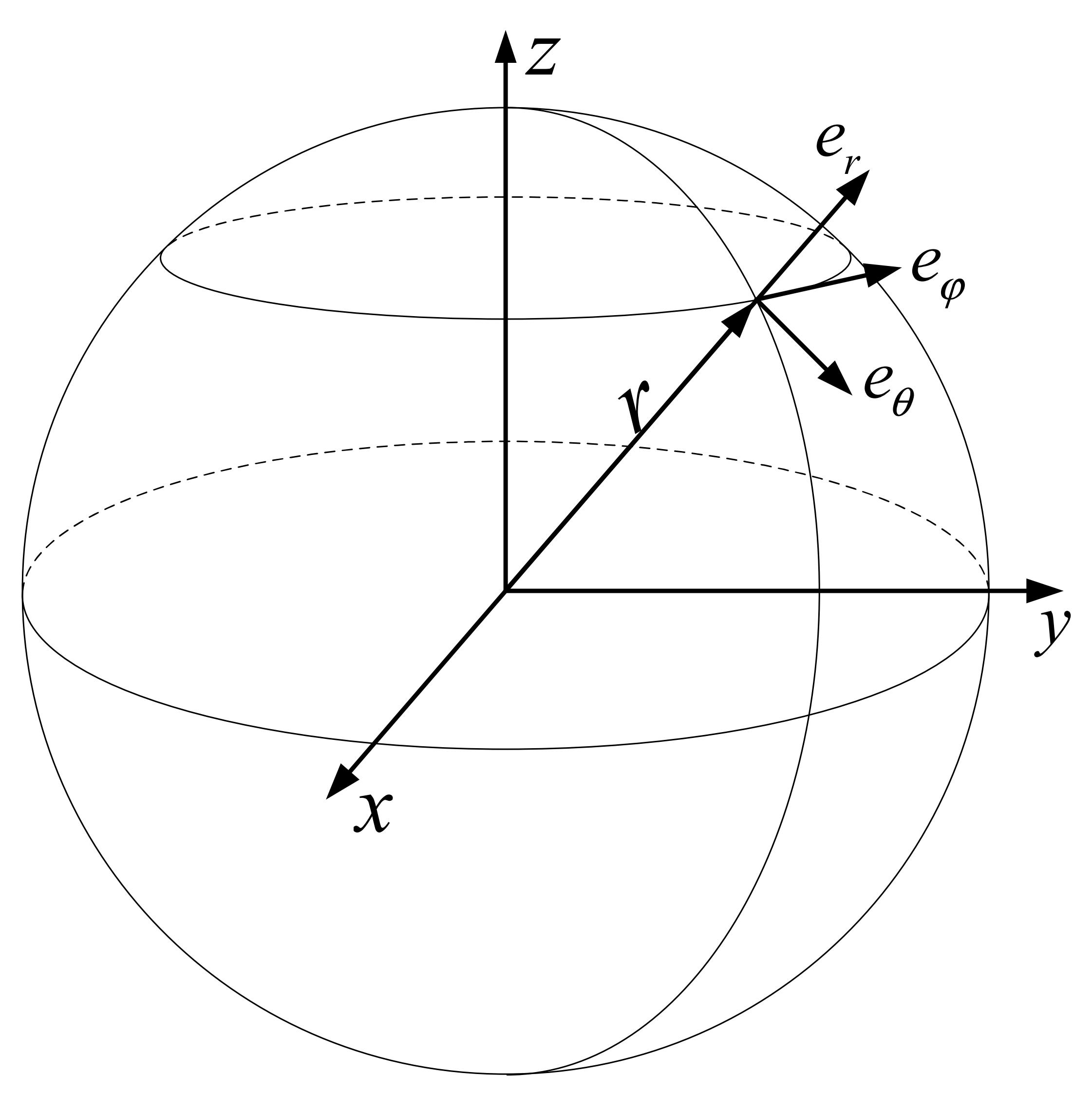}}\hspace*{16pt}
 \subfigure[Uniform partition]{\includegraphics[scale=.6]{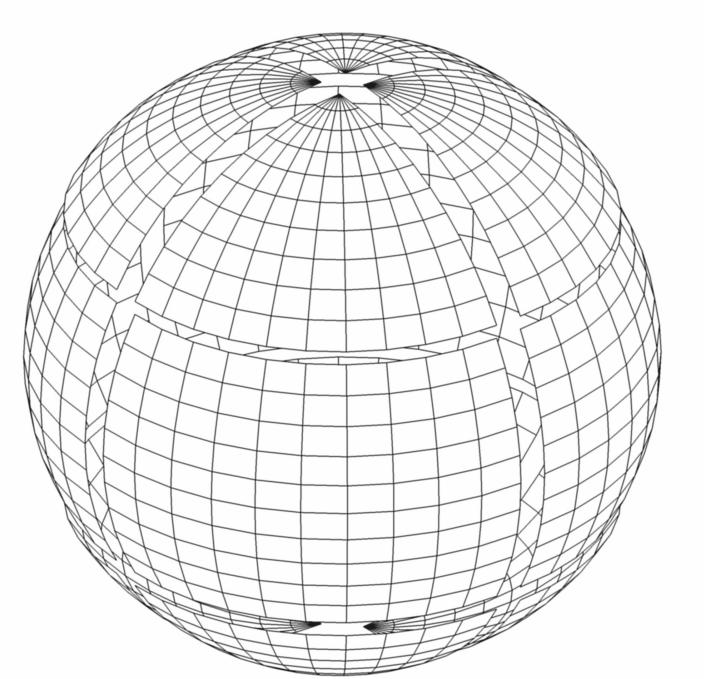}}\hspace*{16pt}
 \subfigure[Adaptive partition]{ \includegraphics[scale=.6]{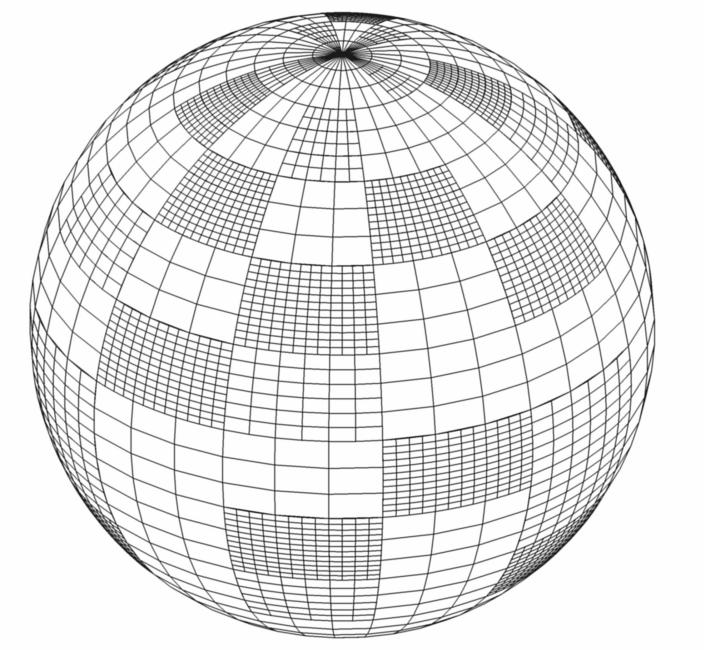}}
  \caption{\small Spherical coordinates system and partitions of unit sphere.}
\label{sphericalcoordinates}
\end{center}
\end{figure}
The spherical harmonics (see, e.g., \cite{Abr.S84})  are defined as
\begin{equation}
\label{sphharmonics}
Y_l^m(\theta, \varphi)=(-1)^m c_l^m
P_l^m(\cos\theta)e^{\ri m\varphi}=\widehat P_l^m(\cos\theta)e^{\ri m\varphi},\quad c_l^m:=\sqrt{\frac{2l+1}{4\pi}\frac{(l-m)!}{(l+m)!}}\,,
\end{equation}
for $0\le |m|\le l,$  $l=0,1\cdots, $
where $P_l^{m}$ (resp. $\widehat P_l^{m}$)  is the associated (resp. normalized) Legendre function of degree $l$ and order $m$.
Recall that  for integer $0\leq m\leq l,$
\begin{equation}
\label{neglegendre}
P_l^m(x)=(-1)^m(1-x^2)^{{m}/{2}} \frac{d^m}{dx^m}P_l(x),\quad x\in (-1,1),
\end{equation}
where $P_l(x)$ is the Legendre polynomial of degree $l$. For negative order, it is defined by 
\begin{equation}\label{negativem}
P_l^{-m}(x)=(-1)^m\frac{(l-m)!}{(l+m)!}P_l^m(x), \;\; {\rm so} \;\; \widehat  P_l^{-m}(x)=(-1)^m \widehat  P_l^{m}(x),\;\;\; 0<m\leq l.
\end{equation}
The so-defined spherical harmonics  $\big\{Y_l^{m}(\theta,\varphi)\big\}_{0\le |m|\le l}$ constitute a complete, orthogonal basis of $L^2(\mathbb {S}^2)$ (where ${\mathbb  S}^2 $ is a unit spherical surface), and
\begin{equation}\label{orthrela}
\langle Y_l^m, Y_l^m\rangle_{\,\mathbb S^2}=\int_0^{\pi}\int_0^{2\pi}Y_l^m(\theta,\varphi)\overline{Y_{l'}^{m'}}(\theta,\varphi)\sin\theta\, d\varphi d\theta=\delta_{ll'}\delta_{mm'},
\end{equation}
where $\bar u$ is the conjugate of $u.$
Note that we have
\begin{equation}\label{reflectrela}
Y_l^{-m}(\theta,\varphi)=(-1)^m\overline{Y_l^{m}}(\theta,\varphi).
\end{equation}


The  algorithms  to be proposed heavily  rely on the  trigonometric forms of the associated normalized Legendre functions $\{\widehat P_l^m(\cos\theta)\}$. For this purpose, we start with the second-order equation:
\begin{equation}\label{2ndode}
\frac{1}{\sin\theta}\frac{d}{d\theta}\Big(\sin\theta\frac{d}{d\theta} \widehat P_l^m(\cos\theta)\Big)
+\Big\{l(l+1)-\frac{m^2}{\sin^2\theta}\Big\}\widehat  P_l^m(\cos\theta)=0,
\end{equation}
which can be solved by the Fourier method.
Based  on the parity of $l$ and $m$, we have
\begin{eqnarray}
\widehat P_l^m(\cos\theta)=\sum\limits_{k=0}^{{l}/{2}}A_{lm}^k\cos(2k\theta),\quad  l\;\;\hbox{even},\;  m\;\;\hbox{even};\label{triform1}\\
\widehat P_l^m(\cos\theta)=\sum\limits_{k=1}^{{l}/{2}}A_{lm}^k\sin(2k\theta),\quad  l\;\;\hbox{even};\; m\;\;\hbox{odd}; \label{triform2}\\
\widehat P_l^m(\cos\theta)=\sum\limits_{k=1}^{{(l+1)}/{2}}A_{lm}^k\cos(2k-1)\theta,\quad l\;\;\hbox{odd};\;  m\;\;\hbox{even};
\label{triform3}\\
\widehat  P_l^m(\cos\theta)=\sum\limits_{k=1}^{{(l+1)}/{2}}A_{lm}^k\sin(2k-1)\theta,\quad l\;\; \hbox{odd};\; m\;\; \hbox{odd};\label{triform4}
\end{eqnarray}
where  $\{A_{lm}^k\}$ can be computed by  (backward) three-term recurrence relations obtained by
inserting the expansions into \eqref{2ndode}.   More precisely,  for a given pair $(l,m),$
\begin{itemize}
 \item  if $l$ is even, then we compute
 \begin{equation}\label{Aevenl0}
 \begin{split}
 &A_{lm}^{l/2}=(-1)^{\lceil m/2\rceil} d_{lm}; \quad    A_{lm}^{l/2-1}=a_{l/2+1} A_{lm}^{l/2}; \quad\\
 &   A_{lm}^{k-2}=a_k  A_{lm}^{k-1} +b_k A_{lm}^{k},\;\;\; k=\frac l 2,\cdots, 3; \\
 & A_{lm}^{0}= \begin{cases}
(a_2 A_{lm}^{1}+b_2 A_{lm}^{2})/2,\quad  &{\rm if}\;\; l\ge 4,\\[4pt]
 \dfrac{l(l+1)-2}{2l(l+1)-4m^2} A_{lm}^{1} ,\quad & {\rm if}\;\; l=2,
 \end{cases}
 \end{split}
 \end{equation}
 where $\lceil a\rceil$ is the smallest integer $\ge a,$ and
\begin{equation}\label{akbkdlm}
\begin{split}
&d_{lm}= \frac{\Gamma(l+1/2)}{\pi}\sqrt{\frac{2l+1}{(l-m)!(l+m)!}}\,,\\
& a_k=\frac{2(2m^2-l(l+1)+4(k-1)^2)}{2(k-2)(2k-3)-l(l+1)},\quad b_k=\frac{l(l+1)-2k(2k-1)}{2(k-2)(2k-3)-l(l+1)};
\end{split}
\end{equation}
\item if $l$ is odd, then we compute
 \begin{equation}\label{Aevenl02}
 \begin{split}
 &A_{lm}^{(l+1)/2}=(-1)^{\lceil m/2\rceil}d_{lm}; \quad A_{lm}^{(l-1)/2}=a_{(l+3)/2} A_{lm}^{(l+1)/2}; \\
 &  A_{lm}^{k-2}=a_k  A_{lm}^{k-1} +b_k A_{lm}^{k},\;\;\; k=\frac{l+1}{2},\cdots, 3;
 \end{split}
 \end{equation}
where
$$a_k=\frac{2(2m^2-l(l+1)+(2k-3)^2)}{(2k-5)(2k-4)-l(l+1)},\quad b_k=\frac{l(l+1)-2(k-1)(2k-1)}{(2k-5)(2k-4)-l(l+1)}.$$
\end{itemize}
\begin{remark}
The formulas for  $A_{lm}^{l/2}$   in \eqref{Aevenl0} and  $A_{lm}^{(l+1)/2}$ in  \eqref{Aevenl02} can be derived from
the Rodrigue's formula {\rm(}cf.  \cite{Szeg75}{\rm):}
\begin{equation}\label{rodrigues}
\widehat{P}_l^m(\cos\theta)=c_l^m\frac{(-1)^m}{2^ll!} \sin^m \theta \frac{d^{l+m}}{dx^{l+m}}(x^2-1)^l,\;\;\; x=\cos \theta,
\end{equation}
whose  leading trigonometric terms are respectively
\begin{equation}\label{leadPlm}
\widehat P_l^m(\cos\theta)=\begin{cases}
(-1)^{{m}/{2}} \dfrac{c_l^m (2l)!}{2^{2l-1}l!(l-m)!}\cos(l\theta)+\cdots, \;\; &  m\;\;\hbox{is even};\\[6pt]
(-1)^{{(m+1)}/{2}} \dfrac{ c_l^m (2l)!}{2^{2l-1}l!(l-m)!}\sin(l\theta)+\cdots, \;\; & m\;\;\hbox{is odd}.
\end{cases}
\end{equation}
Using \eqref{sphharmonics}  and  the identity {\rm(}cf. \cite{Abr.S84}{\rm):}
$$\sqrt{\pi}\Gamma(2l+1)=2^{2l}\Gamma(l+1/2)\Gamma(l+1),$$
we can obtain the desired formulas  in \eqref{Aevenl0} and \eqref{Aevenl02}. \qed
\end{remark}

%

\subsection{Vector spherical harmonics} Several versions of VSH with different notation and  properties  have been used in practice (see,  e.g., \cite{Morse53,Hill54,Nede01,HagStep07,freeden2009spherical,Swa.S00}).
Here, we adopt  the vector spherical harmonics (VSH) $\{\bs Y_l^m, \bs\Psi_l^m, \bs \Phi_l^m\}$ in e.g.,  \cite{Swa.S00}, defined by
\begin{align}
&\bs Y_l^m=Y_l^m\bs e_r,\label{vecylm}\\
&\bs\Psi_l^m=\nabla_SY_l^m=\frac{\partial Y_l^m}{\partial \theta}\bs e_{\theta}+\frac{1}{\sin\theta}\frac{\partial Y_l^m}{\partial \varphi}\bs e_{\varphi},\label{vecpsilm}\\
&\bs\Phi_l^m=\nabla_SY_l^m\times\bs e_r=\frac{1}{\sin\theta}\frac{\partial Y_l^m}{\partial \varphi}\bs e_{\theta}-\frac{\partial Y_l^m}{\partial \theta}\bs e_{\varphi},\label{vecphilm}
\end{align}
for $0\le |m|\le l,$ and note that $\bs \Psi_0^0=\bs \Phi_0^0=\bs 0$.
The VSH form a complete orthogonal basis (right hand) of $(L^2({\mathbb S^2}))^3$. More precisely, they are mutually orthogonal and
\begin{equation}\label{multualorth}
\int_{\mathbb S^2}\bs Y_l^m\cdot \overline{\bs Y_{l'}^{m'}}dS=\delta_{ll'}\delta_{mm'},\quad
\int_{\mathbb S^2} \bs\Psi_l^m\cdot \overline{\bs\Psi_{l'}^{m'}}dS=\int_{\mathbb S^2} \bs\Phi_l^m\cdot \overline{\bs\Phi_{l'}^{m'}}dS=l(l+1)\delta_{ll'}\delta_{mm'}.
\end{equation}

Notice from  \eqref{sphharmonics} that
  $$\frac{\partial Y_l^m}{\partial \theta}=e^{\ri m\varphi}\frac{d}{d\theta}\widehat P_l^m(\cos\theta),\quad \frac{1}{\sin\theta}\frac{\partial Y_l^m}{\partial \varphi}=\ri me^{\ri m\varphi}\frac{\widehat P_l^m(\cos\theta)}{\sin\theta},$$
 where  ${\widehat  P_l^m(\cos\theta)}/{\sin\theta}$ is singular at the poles.
  In order to compute the basis functions $\{\bs\Psi_l^m, \bs\Phi_l^m\}$ accurately and efficiently,  it is necessary to  compute   ${\widehat  P_l^m(\cos\theta)}/{\sin\theta}$ and $\frac{d}{d\theta}\widehat P_l^{m}(\cos\theta)$ by some compact combination of $\{\widehat P_l^m(\cos\theta)\}.$
For this purpose, we  recall  the recurrence formula (cf. \cite{Abr.S84})
\begin{equation*}\label{PlmA}
\sqrt{1-x^2}\frac{d}{dx}P_l^m(x)=\frac{1}{2}\big((l+m)(l-m+1)P_l^{m-1}(x)-P_l^{m+1}(x)\big),
\end{equation*}
so  by  \eqref{sphharmonics},
%
%
%
\begin{equation}\label{hatPrec}
\sqrt{1-x^2}\frac{d}{dx}\widehat P_l^m(x)=\frac{1}{2}\Big(\frac{c_l^m}{c_l^{m+1}}\widehat  P_l^{m+1}(x)-(l+m)(l-m+1)\frac{c_l^m}{c_l^{m-1}}\widehat P_l^{m-1}(x)\Big).
\end{equation}
Thus, we have
\begin{equation}
\label{deriassleg}
\left\{
\begin{split}
&\frac{d}{d\theta}\widehat P_l^m(\cos\theta)=\tilde c_l^{m,1}\widehat  P_l^{m-1}(\cos\theta)-\tilde c_l^{m,2}\widehat  P_l^{m+1}(\cos\theta),\;\;\; {\rm where}\\[4pt]
& \tilde c_l^{m,1}=\frac{1}{2}\sqrt{(l+m)(l-m+1)},\quad \tilde c_l^{m,2}=\frac{1}{2}\sqrt{(l+m+1)(l-m)}.
\end{split}\right.
\end{equation}
It is noteworthy that for $m=0,l,$ the above formulas should be understood as follows
\begin{equation}
\label{deriasslegspecial}
\hspace{-15pt}
\begin{split}
&\frac{d}{d\theta}\widehat  P_l^0(\cos\theta)=\frac{1}{2} \sqrt{l(l+1)} \Big(\widehat P_l^{-1}(\cos\theta)-\widehat  P_l^{1}(\cos\theta)\Big)=-\sqrt{l(l+1)}\widehat P_l^{1}(\cos\theta),\\
&\frac{d}{d\theta}\bar P_l^l(\cos\theta)=\sqrt{\frac{l}{2}}\widehat P_l^{l-1}(\cos\theta),
\end{split}
\end{equation}
where we used  \eqref{negativem} to derived the first identity.

To deal with ${\widehat  P_l^m(\cos\theta)}/{\sin\theta},$ we recall a second  recurrence formula  (cf. \cite{Abr.S84})
\begin{equation*}
\frac{1}{\sqrt{1-x^2}}P_l^m(x)=-\frac{1}{2m}\big((l-m+2)(l-m+1)P_{l+1}^{m-1}(x)+P_{l+1}^{m+1}(x)\big),\;\;\; m>0,
\end{equation*}
which, together with  \eqref{sphharmonics},   leads to
\begin{equation}\label{Plmsin}
\left\{
\begin{split}
& \frac{\widehat P_l^m(\cos\theta)}{\sin\theta}=\frac{1}{2m}\Big(\hat c_l^{m,1}\widehat P_{l+1}^{m-1}(\cos\theta)+\hat c_l^{m,2}\widehat P_{l+1}^{m+1}(\cos\theta)\Big),\\[4pt]
&\hat c_l^{m,1}=\sqrt{\frac{(2l+1)(l-m+1)(l-m+2)}{2l+3}},\quad \hat c_l^{m,2}=\sqrt{\frac{(2l+1)(l+m+1)(l+m+2)}{2l+3}}.
\end{split}\right.
\end{equation}

\vskip 10pt
\section{Algorithms for SPH and VSH expansions}\label{sect3:expan}
\setcounter{equation}{0}

In this section, we introduce the partition  and   analytic formulas for computing the integrals in \eqref{expan1}, and demonstrate the high accuracy of  high-mode expansions.
\subsection{Spectral-element grids on the sphere}\label{Sph31}
As already mentioned, the way of partitioning the spherical surface and the form of the associated elemental mapping
are essential for computing the SPH and VSH expansion coefficients.   Here, we adopt a partition so that we can  resort to  analytical formulas  to evaluate the integrals of interest as one-dimensional iterated integrals.
Basically,  let   $\mathbb{S}_h^2$ be a non-overlapping  partition of ${\mathbb S}^2$ such that each element $S^e\in\mathbb{S}_h^2$  is  a rectangular domain in  $(\theta,\varphi)$-coordinates, that is,  $S^e=[\theta_{s-1},\theta_s]\times[\varphi_{t-1}, \varphi_t]$ (see Figure \ref{sphericalcoordinates}).
%
Then we can choose the elemental mapping $\mathcal{F}_{e}: Q=[-1,1]^2\rightarrow S^e$  in the spectral-element discretization to be
\begin{equation}
\label{anotransform}
(\eta, \xi)\rightarrow (\theta^s(\eta),\varphi^t(\xi)):\;\; \theta^s(\eta)=\hat{\theta}_s\eta+\alpha_s,\;\;
\varphi^t(\xi)=\hat{\varphi}_t\xi+\beta_t,\;\;\; \xi,\eta\in [-1,1],
\end{equation}
where
\begin{equation}\label{geocons}
\hat{\theta}_s=\frac{\theta_s-\theta_{s-1}}{2}
\quad \alpha_s=\frac{\theta_{s-1}+\theta_s}{2},\;\;\;
\hat{\varphi}_t=\frac{\varphi_{t}-\varphi_{t-1}}{2},\;\;\;  \beta_t=\frac{\varphi_{t-1}+\varphi_t}{2},
\end{equation}
are constants determined by the ``vertices" of the element $S^e$. It is noteworthy that  the mapping  $\mathcal{F}_e$ is linear and smooth  from the reference coordinates $(\eta,\xi)$ to the spherical coordinates $(\theta,\varphi)$.

Let $\{(\eta_j,\xi_i)\}_{i, j=0}^N$ be the tensorial Legendre-Gauss-Lobatto (LGL) points in the reference square $Q$. The distributions of mapped LGL points in two typical kinds of physical elements are sketched in Figure \ref{LGLpoints}. The spectral-element approximations of a scalar and a vector field in \eqref{spectralelementapprox} take  the form:
\begin{align}
&u_N^E|_{S^e}=u_N^E(\mathcal F_e(\eta, \xi))=\sum\limits_{i=0}^N\sum\limits_{j=0}^Nu_{ij}^{e}\psi_{ij}(\mathcal F_e^{-1}(\theta, \varphi)), \label{scafieldseapp}
\\&\bs v_N^E|_{S^e}=\bs v_N^E(\mathcal F_e(\eta, \xi))=\sum\limits_{i=0}^N\sum\limits_{j=0}^N\bs v_{ij}^{e}\psi_{ij}(\mathcal F_e^{-1}(\theta, \varphi)), \label{vecfieldseapp}
\end{align}
where $\{u_{ij}^e\}$ and $\{\bs v_{ij}^e\}$ are nodal values on the mapped LGL points on $S^e,$ and $\psi_{ij}(\eta,\xi)=l_i(\xi)l_j(\eta)$ are the corresponding basis functions with $\{l_i(\xi), l_j(\eta)\}$ being the  Lagrange interpolating  basis polynomials  with respect to the LGL points (see, e.g., \cite{ShenTangWang2011}).
\begin{figure}[htbp]
\begin{center}
 \subfigure[LGL points on  $Q$]{ \includegraphics[scale=.23]{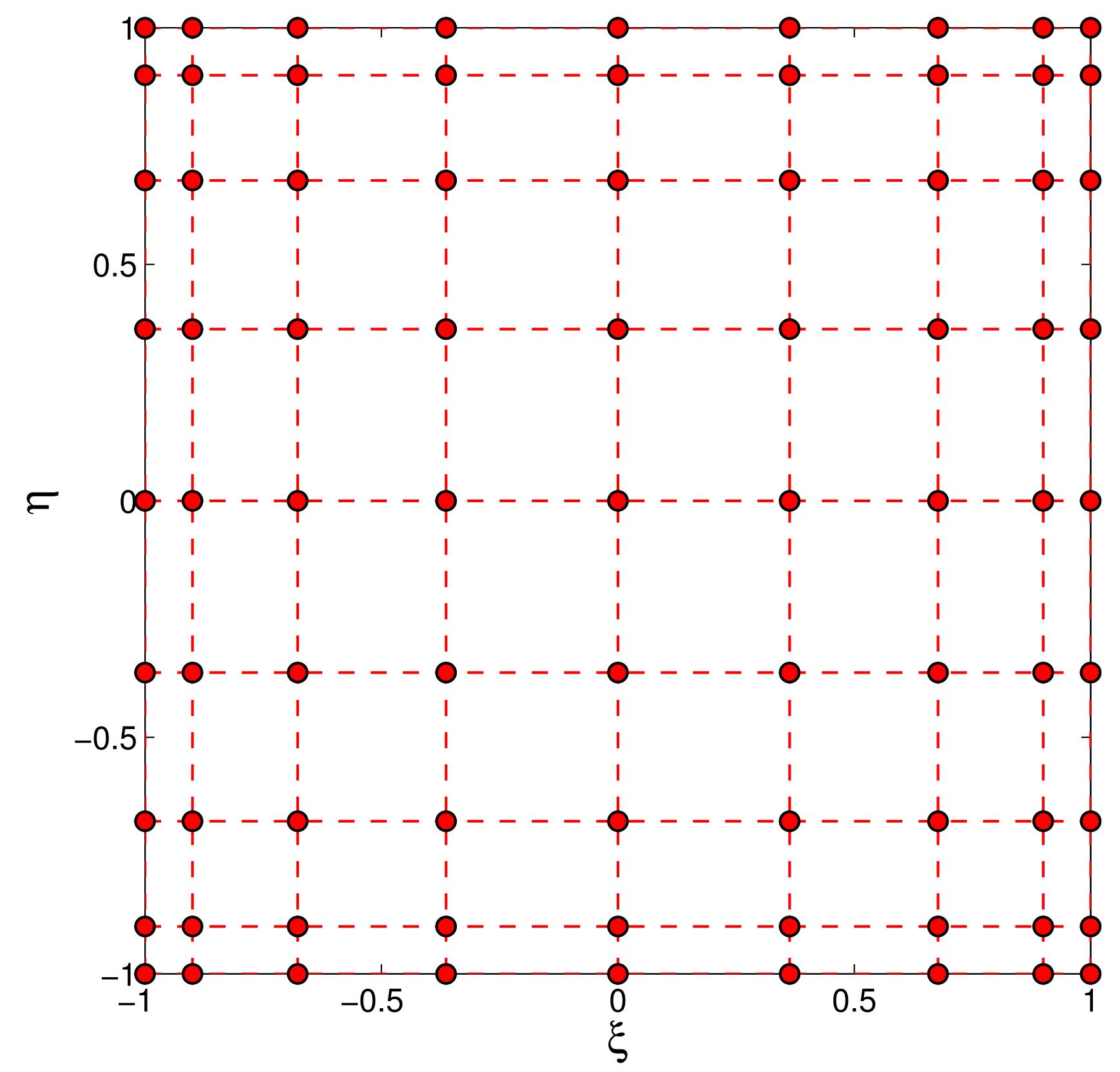}}\hspace*{20pt}
 \subfigure[Mapped LGL points on $S^e$]{\includegraphics[scale=.28]{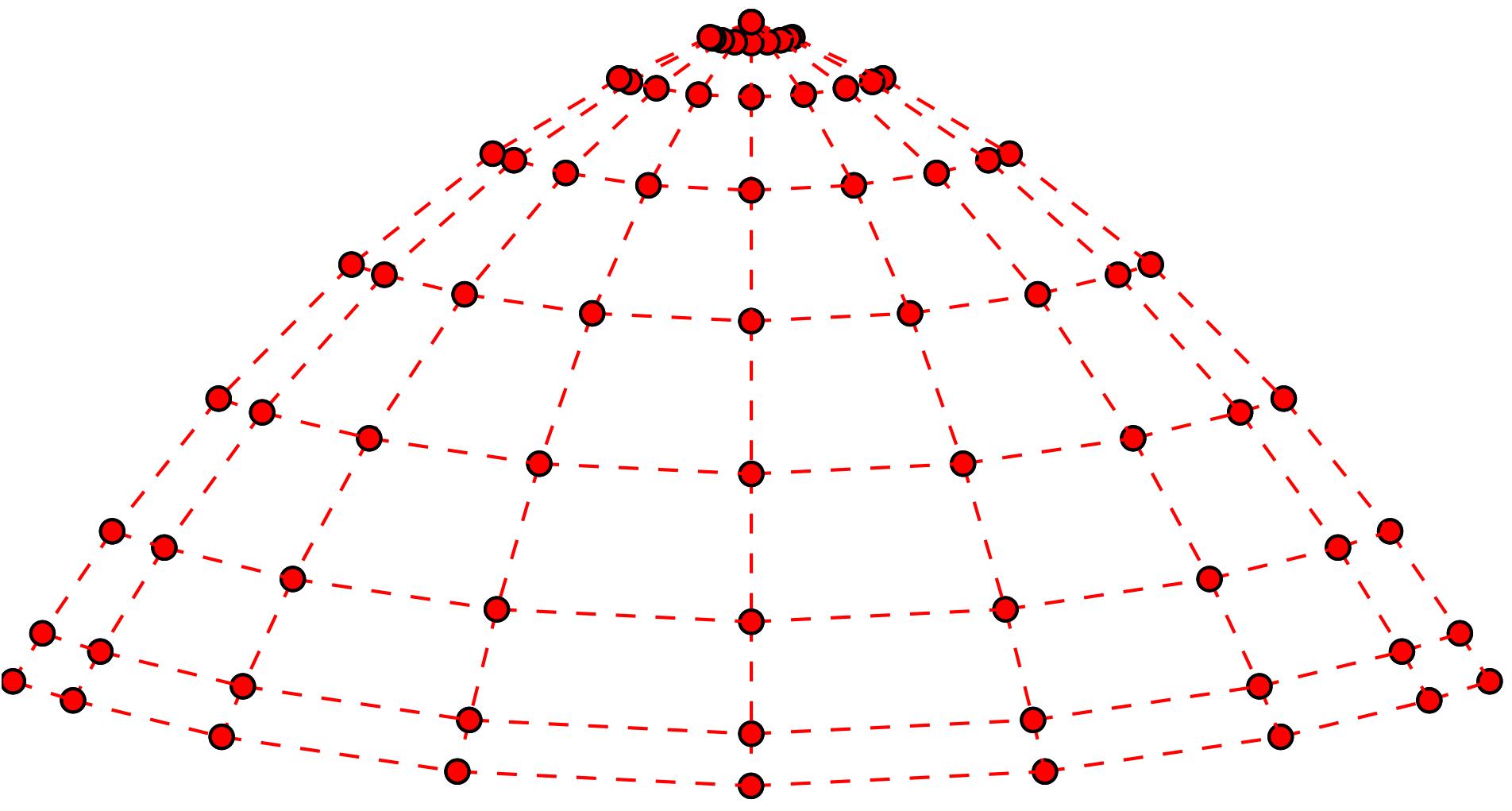}}\hspace*{20pt}
 \subfigure[Mapped LGL points on $S^e$]{ \includegraphics[scale=.28]{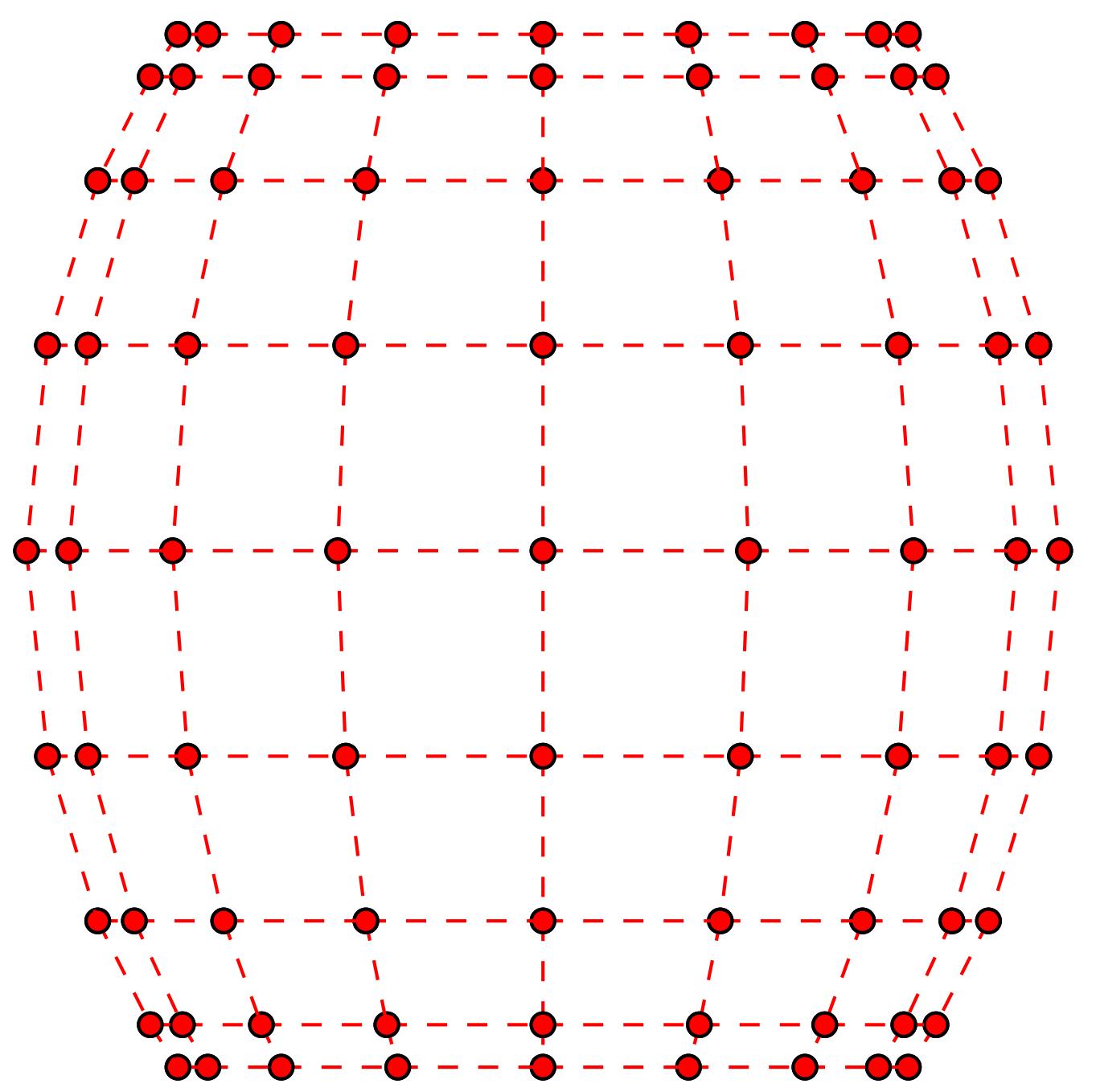}}
  \caption{\small Tensorial LGL points on the reference square and two typical kinds of physical elements with the mapping.}
\label{LGLpoints}
\end{center}
\end{figure}



\subsection{Accurate formulas for SPH expansions}
Observe from  \eqref{expan1} that  the computation of  spherical harmonic coefficients  is actually to calculate a bunch of integrals involving piecewise smooth and highly oscillatory  integrands on $\mathbb{S}^2$.
The naive use of usual quadrature rules does not lead to  very accurate results.
 In order to maintain the spectral accuracy obtained by the spectral element approximation within reasonable computational cost, we propose the following semi-analytical means.
\begin{proposition}\label{Prop31}
Given a spectral element approximation $u_N^E$ defined in \eqref{scafieldseapp}, its spherical harmonic expansion coefficients
can be computed by the formula
\begin{equation}
\label{scasphcoefformula}
\tilde a_l^m=\int_{\mathbb{S}^2}u_N^E(\theta,\varphi)\overline{Y_l^m(\theta,\varphi)}dS=\sum\limits_{e=0}^E\sum\limits_{i=0}^N\sum\limits_{j=0}^Nu_{ij}^e\mathfrak{a}_{m,i}^t\mathfrak{b}_{l, m,j}^s,
\end{equation}
where
\begin{equation}
\mathfrak{a}_{m,i}^t=\begin{cases}
2v_{i0}\,\hat{\varphi}_t\,e^{-\ri m\beta_t},  & m=0,\\
\hat{\varphi}_te^{-\ri m\beta_t}\sum\limits_{n=0}^N2 \ri^{-n} v_{in}\, j_n(m\hat{\varphi}_t), & m\neq 0,
\end{cases}\quad\mathfrak{b}_{l,m,j}^s=\hat{\theta}_s\sum\limits_{n=0}^Nv_{jn}\mathscr{Q}^n_{lm}(\hat{\theta}_s,\alpha_s),
\label{ablmijkanal}
\end{equation}
$\{v_{in}\}$ are the Legendre polynomial representation coefficients of Lagrange nodal basis $\{l_i\}$  in \eqref{representlag}, $j_n(z)$ is the spherical Bessel function of the first kind,  and
\begin{equation}
\label{blmjint2}
\mathscr{Q}^n_{lm}(\lambda,\rho):=\int_{-1}^1 P_n(x)\widehat P_l^m(\cos(\lambda x+\rho))\sin(\lambda x+\rho)\,dx,
\end{equation}
{\rm(}$P_n(x)$ is the Legendre polynomial of degree $n${\rm),} which can be computed by
  \eqref{analyticintformula4}.
\end{proposition}
\begin{proof}
According to the partition $\mathbb{S}_h^2$, we decompose the integral into
\begin{equation}
\label{coefscheme}
\tilde a_l^m=\int_{\mathbb{S}^2}u_N^E(\theta,\varphi)\overline{Y_l^m(\theta,\varphi)}dS
=\sum\limits_{e=1}^{E}\bs I_{lm}^{e},
\end{equation}
where
\begin{equation}
\label{ilme}
\bs I_{lm}^{e}:=\int_{S^e}
u_N^E(\theta,\varphi)\overline{Y_l^m(\theta,\varphi)}\sin\theta d\theta d\varphi=\int_{\varphi_{t-1}}^{\varphi_t}\int_{\theta_{s-1}}^{\theta_s}
u_N^E(\theta,\varphi)\overline{Y_l^m(\theta,\varphi)}\sin\theta d\theta d\varphi.
\end{equation}
Applying the elemental mapping  $\mathcal{F}_{e}$ yields
\begin{equation}
\label{sphcoefformula}
\bs I_{lm}^{e}=\int_{-1}^1\int_{-1}^1u_N^E(\theta^s(\eta),\varphi^t(\xi))\overline{Y_l^m(\theta^s(\eta),\varphi^t(\xi))}
\sin(\theta^s(\eta))\Big|\frac{D(\theta^s,\varphi^t)}{D(\xi,\eta)}\Big| d\xi d\eta.
\end{equation}
Here $\frac{D(\theta^s,\varphi^t)}{D(\xi,\eta)}$ is the Jacobian of the transformation $\mathcal{F}_e$ on patch $S^e$. A very important feature in the mapping $\mathcal{F}_{e}$ is that the angular variables $\theta$ and $\varphi$ are linearly dependent on the reference variables $\eta$ and $\xi$ independently  (cf. \eqref{anotransform}). Consequently, the surface integral \eqref{sphcoefformula} can be formulated into the product of two one-dimensional integrals as follows
\begin{equation}
\label{anopatchapprox}
\hspace{-20pt}\bs I_{lm}^{e}=\sum\limits_{i,j=0}^Nu_{ij}^{e}\Big(\hat{\varphi}_t\int_{-1}^{1}
l_i(\xi)e^{-\ri m\varphi^t(\xi)}d\xi\Big)\Big(\hat{\theta}_s\int_{-1}^{1}l_j(\eta)\widehat P_l^m(\cos\theta^s(\eta)) \sin(\theta^s(\eta))d\eta\Big).
\end{equation}
Next, we  deal with the above two integrals
\begin{equation}
\label{independentint}
\hspace{-10pt}\mathfrak{a}_{m,i}^t=\hat{\varphi}_t\int_{-1}^{1}l_i(\xi)e^{-\ri m\varphi^t(\xi)}d\xi,\quad
\mathfrak{b}_{l,m,j}^s=\hat{\theta}_s\int_{-1}^{1}l_j(\eta)\widehat P_l^m(\cos\theta^s(\eta)) \sin(\theta^s(\eta))d\eta,
\end{equation}
separately, so the formula \eqref{scasphcoefformula} is obtained by simply using \eqref{anopatchapprox} and \eqref{independentint}  in \eqref{coefscheme}.
More precisely, by the definition \eqref{anotransform},
\begin{align}
&\mathfrak{a}_{m,i}^t=\hat{\varphi}_te^{-\ri m\beta_t}\int_{-1}^1l_i(\xi)e^{-\ri m\hat{\varphi}_t\xi}d\xi,\label{seperateint1}\\
&\mathfrak{b}_{l,m,j}^s=\hat{\theta}_s\int_{-1}^{1}l_j(\eta)\widehat P_l^m\big(\cos(\hat{\theta}_s\eta+\alpha_s)\big) \sin(\hat{\theta}_s\eta+\alpha_s)d\eta.\label{seperateint2}
\end{align}
By \eqref{negativem}, we have
$$\mathfrak{a}_{-m,i}^t=\overline{\mathfrak{a}_{m,i}^t},\quad \mathfrak{b}_{l,-m,j}^s=(-1)^m\mathfrak{b}_{l,m,j}^s,\quad \hbox{for}\;\;m>0.$$
Therefore, we only need to consider the integrals \eqref{seperateint1}-\eqref{seperateint2} for $0\leq m\leq l$.

Note that the Lagrange nodal basis $\{l_i(\xi)\}_{i=0}^N$ can be represented in terms of Legendre polynomials $\{P_n(\xi)\}_{n=0}^N$ as follows (cf. \cite{ShenTangWang2011}): 
\begin{equation}
\label{representlag}
l_i(\xi)=\sum\limits_{n=0}^Nv_{in}P_n(\xi),\quad  2v_{in}=\begin{cases}
(2n+1)\omega_iP_n(\xi_i), & i=0, \cdots, n-1,\\
n\omega_iP_n(\xi_i), & i=n,
\end{cases}
\end{equation}
where $\{\xi_i, \omega_i\}_{i=0}^N$ are LGL quadrature nodes and weights.
Therefore, it suffices to derive analytic formulas for the integrals \eqref{seperateint1} and \eqref{seperateint2} with the Lagrange nodal basis functions replaced by Legendre polynomials. For this purpose, let us define
\begin{equation}\label{blmjint}
\begin{split}
&\mathscr{C}_n(\lambda,\rho):=\int_{-1}^1 P_n(x)\cos(\lambda x+\rho)dx,\quad \mathscr{S}_n(\lambda,\rho):=\int_{-1}^1 P_n(x)\sin(\lambda x+\rho)dx,
\\&\mathscr{P}^n_{lm}(\lambda,\rho):=\int_{-1}^1 P_n(x)\widehat P_l^m\big(\cos(\lambda x+\rho)\big)dx.
\end{split}
\end{equation}

We proceed with  the following important formula:
\begin{equation}\label{analyticintformula1}
\int_{-1}^1 P_n(x)e^{-\ri (\lambda x+\rho)}dx
=\begin{cases}
2e^{-\ri \rho}\delta_{n0}, & \lambda=0,\\
2\ri^{-n}e^{-\ri \rho}j_n(\lambda), & \lambda\neq 0,
\end{cases}
\end{equation}
which can be  derived from
\begin{equation}\label{newformula}
\int_{-1}^1 P_n(x)e^{-\ri \lambda x}dx
=\begin{cases}
\displaystyle 2\delta_{n0}, & \lambda=0,\\
\displaystyle 2\ri^{-n}j_n(\lambda), & \lambda\in\mathbb{C}\backslash\{0\},
\end{cases}
\end{equation}
  (cf. \cite{bateman1954tables}) straightforwardly.   Consequently, we obtain the explicit formulas for the first two integrals in
  \eqref{blmjint}:
  \begin{equation}
\label{analyticintformula2}
\hspace{-10pt}
\begin{split}
\mathscr{C}_n(\lambda,\rho)
=\begin{cases}
2\delta_{n0}\cos\rho, & \lambda=0,\\
2j_n(\lambda)(\mathfrak{Re}\{\ri^{-n}\}\cos\rho+\mathfrak{Im}\{\ri^{-n}\}\sin\rho), & \lambda\neq 0;
\end{cases}\\
\mathscr{S}_n(\lambda,\rho)
=\begin{cases}
2\delta_{n0}\sin\rho, & \lambda=0,\\
2j_n(\lambda)(\mathfrak{Re}\{\ri^{-n}\}\sin\rho-\mathfrak{Im}\{\ri^{-n}\}\cos\rho), & \lambda\neq 0.
\end{cases}
\end{split}
\end{equation}

With \eqref{analyticintformula1},  we can then compute the integral \eqref{blmjint2} and the last  integral in \eqref{blmjint} by using the trigonometric formulas \eqref{triform1}-\eqref{triform4}.  For example, if $l$ and $m$ are both even, then
\begin{equation*}
\hspace{-25pt}
\begin{split}
P_n(x)\widehat P_l^m(\cos(\lambda x+\rho))&=\sum\limits_{k=0}^{l/2} A_{lm}^kP_n(x)\cos(q_{k}(\lambda x+\rho)),\\
P_n(x)\widehat P_l^m(\cos(\lambda x+\rho))\sin(\lambda x+\rho)&=\frac{1}{2}\sum\limits_{k=0}^{l/2} A_{lm}^kP_n(x)[\sin(p_k(\lambda x+\rho))-\sin(p_{k-1}(\lambda x+\rho))],
\end{split}
\end{equation*}
where $p_k:=2k+1$, $q_k:=2k$ and $A_{lm}^k$ are coefficients given in \eqref{triform1}-\eqref{triform4}. Thus the exact formulas \eqref{analyticintformula2} can be used.
In summary,  we have 
\begin{equation}
\label{analyticintformula3}
\mathscr{P}^n_{lm}(\lambda,\rho)=\begin{cases}
\sum\limits_{k=0}^{l/2} A_{lm}^k\mathscr{C}_n(q_{k}\lambda,q_{k}\rho), & l\;\;\hbox{even};\;\; m\;\;\hbox{even},\\
\sum\limits_{k=1}^{l/2} A_{lm}^k\mathscr{S}_n(q_{k}\lambda,q_{k}\rho), & l\;\;\hbox{even};\;\; m\;\;\hbox{odd},\\
\sum\limits_{k=1}^{(l+1)/2} A_{lm}^k\mathscr{C}_n(p_{k-1}\lambda,p_{k-1}\rho), & l\;\;\hbox{odd};\;\; m\;\;\hbox{even},\\
\sum\limits_{k=1}^{(l+1)/2} A_{lm}^k\mathscr{S}_n(p_{k-1}\lambda,p_{k-1}\rho), & l\;\;\hbox{odd};\;\; m\;\;\hbox{odd},
\end{cases}
\end{equation}
and
\begin{equation}
\label{analyticintformula4}
\mathscr{Q}^n_{lm}(\lambda,\rho)=\begin{cases}
\dfrac{1}{2}\sum\limits_{k=0}^{l/2} A_{lm}^k[\mathscr{S}_n(p_{k}\lambda,p_{k}\rho)-\mathscr{S}_n(p_{k-1}\lambda,p_{k-1}\rho)], & l\;\;\hbox{even};\;\; m\;\;\hbox{even},\\
\dfrac{1}{2}\sum\limits_{k=1}^{l/2} A_{lm}^k[\mathscr{C}_n(p_{k-1}\lambda,p_{k-1}\rho)-\mathscr{C}_n(p_{k}\lambda,p_{k}\rho)],& l\;\;\hbox{even};\;\; m\;\;\hbox{odd},\\
\dfrac{1}{2}\sum\limits_{k=1}^{(l+1)/2} A_{lm}^k[\mathscr{S}_n(q_{k}\lambda,q_{k}\rho)-\mathscr{S}_n(q_{k-1}\lambda,q_{k-1}\rho)],& l\;\;\hbox{odd};\;\; m\;\;\hbox{even},\\
\dfrac{1}{2}\sum\limits_{k=1}^{(l+1)/2} A_{lm}^k[\mathscr{C}_n(q_{k-1}\lambda,q_{k-1}\rho)-\mathscr{C}_n(q_{k}\lambda,q_{k}\rho)],& l\;\;\hbox{odd};\;\; m\;\;\hbox{odd},
\end{cases}
\end{equation}
where $p_k:=2k+1$, $q_k:=2k$ and $A_{lm}^k$ are coefficients given in \eqref{triform1}-\eqref{triform4}.

Substituting the Legendre polynomial representation formula \eqref{representlag} of $\{l_i(\xi), l_j(\eta)\}$ and analytic formulas \eqref{analyticintformula1} and \eqref{analyticintformula4} into \eqref{seperateint1} and \eqref{seperateint2}, we obtain formulas \eqref{ablmijkanal}.
\end{proof}
\begin{remark}\label{oneremark}
In \eqref{analyticintformula4}, we have $\mathscr{S}_n(-\lambda,-\rho)$ when $l$ is even and $k=0$. That needs the computation of $j_n(-\lambda)$ with $\lambda>0$. The spherical Bessel function with negative arguments is usually not available in some popular library, e.g. GSL library. In fact, the calculation of $\mathscr{S}_n(-\lambda,-\rho)$ can be avoided. When $l$ is even and $k=0$, we have
$$\mathscr{S}_n(\lambda,\rho)-\mathscr{S}_n(-\lambda,-\rho)=2\int_{-1}^1P_n(x)\sin(\lambda x+\rho)dx=2\mathscr{S}_n(\lambda,\rho).$$
\end{remark}

\begin{remark}\label{byproduct}
It is noteworthy that starting from the spectral element approximation of a function on the sphere, the proposed algorithm for computing the spherical harmonic expansion coefficients does not induce additional errors.   Here,  we summarise the approach for computing the coefficients of a given smooth function $u$ {\color{blue}on the sphere $\mathbb{S}^2$} as follows: 
\begin{enumerate}
  \item[(i)] Construct a Cartesian partition $\mathbb{S}_h^2=\{S^e\}_{e=1}^E$ of $\mathbb{S}^2$ in the $\theta$-$\varphi$ plane;
  \item[(ii)] Approximate $u$ by its {\color{blue}LGL} interpolation $\mathcal{I}_N^Eu$ associated with the mesh $\mathbb{S}_h^2$;
  \item[(iii)] Compute the spherical harmonic coefficients of $\mathcal{I}_N^Eu.$    \qed
\end{enumerate}
\end{remark}


To have some insights into  the complexity of this approach, we arrange the nodal values $\{u_{ij}^e\}$ into  the matrix
$$
\bs U^e=\begin{bmatrix}
\bs U_0^e & \bs U_1^e & \cdots & \bs U_N^e
\end{bmatrix},
$$
where the column vectors are given by $\bs U^e_j=(u_{ij}^e)_{(N+1)\times 1}$.
Define the vectors $\bs A^t_m=(\mathfrak{a}_{m,i}^t)'$ and $\bs B^s_{lm}=(\mathfrak{b}_{l,m,j}^s)'$ where the prime denotes the transpose. Then, the matrix-vector form of \eqref{scasphcoefformula} is
\begin{equation}\label{matform}
\sum\limits_{i=0}^N\sum\limits_{j=0}^Nu_{ij}^e\mathfrak{a}_{m,i}^t\mathfrak{b}_{l, m,j}^s=((\bs A^t_m)'\bs U^e)\bs B^s_{lm}.
\end{equation}
It is important to remark that the vector $(\bs A^t_m)'\bs U^e$ is independent of $l$, so it can be used repeatedly  for different $l$. Assuming that we need to compute the coefficients $\{\tilde a_l^m\}_{l=0}^L$, we  have $L+1$ vectors $\{(\bs A^t_m)'\bs U^e\}_{m=0}^L$ to be formed, which requires $(L+1)\times (N+1)^2$ multiplications. With these vectors ready, $N+1$ multiplications are needed for computing every coefficient. From \eqref{representlag}, it is easy to see that the computational cost for $\{v_{in}\}$ is $(N+1)^2$.  Thus, once the values $\{v_{in}\}$ and  $\{\mathscr{Q}^n_{lm}(\hat{\theta}_s,\alpha_s)\}$ are computed off stage, the total computational cost is $((L+1)\times (N+1)^2+(L+1)^2\times (N+1))\times E$.

\subsection{Accurate formulas for VSH expansions}
In this section, we explore accurate formulas for computing VSH expansions of a vector field  $\bs v_N^E$ in \eqref{vecfieldseapp} given by  spectral-element nodal values. 
For convenience, we express the vector field $\bs v_N^E$ in the spherical coordinate basis $\{\bs e_r,\bs e_{\theta},\bs e_{\varphi}\}$, namely,
$$\bs v_N^E(\theta,\varphi)=u_N^E(\theta,\varphi)\bs e_r+v_N^E(\theta,\varphi)\bs e_{\theta}+w_N^E(\theta,\varphi)\bs e_{\varphi}.$$
Noting that  for any vector field $\bs v$ given  in the Cartesian coordinate basis:
$${\bs v}(\theta, \varphi)=v_x(\theta,\varphi)\bs e_x+v_y(\theta,\varphi)\bs e_y+v_z(\theta,\varphi)\bs e_z,$$
one can represent it in terms of  $\{\bs e_r, \bs e_{\theta}, \bs e_{\varphi}\}:$
$${\bs v}(\theta, \varphi)=
v_r(\theta,\varphi)\bs e_r+ v_{\theta}(\theta,\varphi)\bs e_{\theta}+ v_{\varphi}(\theta,\varphi)\bs e_{\varphi},$$
where the components are connected via
\begin{equation}
\begin{bmatrix}
v_r(\theta,\varphi)& v_{\theta}(\theta,\varphi) & v_{\varphi}(\theta,\varphi)
\end{bmatrix}=\begin{bmatrix}
v_x(\theta,\varphi)& v_y(\theta,\varphi) &v_z(\theta,\varphi))
\end{bmatrix}\mathbb T(\theta,\varphi),
\end{equation}
with the transformation matrix $\mathbb T(\theta,\varphi)$ given by
$$
\mathbb T(\theta,\varphi)=\begin{bmatrix}
\sin\theta\cos\varphi & \sin\theta\sin\varphi & \cos\theta\\
\cos\theta\cos\varphi & \cos\theta\sin\varphi & -\sin\theta\\
-\sin\varphi & \cos\varphi & 0
\end{bmatrix}.
$$

Denote by $\{{U}_{ij}^{e}, {V}_{ij}^{e}, {W}_{ij}^{e}\}$ the $\{\bs e_r,\bs e_{\theta},\bs e_{\varphi}\}$ components of the vector nodal values $\bs v_{ij}^e$ in \eqref{vecfieldseapp}.
\begin{proposition}
Given a spectral element approximation $\bs v_N^E$ in \eqref{vecfieldseapp}, we compute $\big\{\tilde v_{lm}^r, \tilde v_{lm}^{(1)}, \tilde v_{lm}^{(2)}\big\}$ in \eqref{expan1}  by
\begin{align}
&\tilde v_{lm}^r=\sum\limits_{e=1}^{E}
\sum\limits_{i=0}^{N}\sum\limits_{j=0}^{N}U_{ij}^{e}\,\mathfrak{a}_{m,i}^t\mathfrak{b}_{l,m,j}^s,\label{vecsphcoefformula1}\\
&\tilde{v}_{lm}^{(1)}=\frac{1}{l(l+1)}\sum\limits_{e=1}^{E}\sum\limits_{i=0}^{N}
\sum\limits_{j=0}^{N}\mathfrak{a}_{m,i}^t\big\{V_{ij}^{e}\mathfrak{d}_{l,m,j}^s-\ri mW_{ij}^{e}\mathfrak{c}_{l,m,j}^s\big\},\label{vecsphcoefformula2}\\
&\tilde{v}_{lm}^{(2)}=\frac{-1}{l(l+1)}\sum\limits_{e=1}^{E}\sum\limits_{i=0}^{N}
\sum\limits_{j=0}^{N}\mathfrak{a}_{m,i}^t\big\{\ri mV_{ij}^{e}\mathfrak{c}_{l,m,j}^s+W_{ij}^{e}\mathfrak{d}_{l,m,j}^s\big\},\label{vecsphcoefformula3}
\end{align}
where $\mathfrak{a}_{m,i}^t$ and $\mathfrak{b}_{l,m,j}^s$ are the same as in  \eqref{ablmijkanal},
\begin{equation}
\label{vecsphcoefcompform}
\begin{split}
&\mathfrak{c}_{l,m,j}^s=\hat{\theta}_s\mathscr{P}_{lm}^n(\hat{\theta}_s,\alpha_s),\quad \mathfrak{d}_{l,0,j}^s=-\sqrt{l(l+1)}\mathfrak{b}_{l,1,j}^s,\quad \mathfrak{d}_{l,l,j}^s=\sqrt{{l}/{2}}\mathfrak{b}_{l,l-1,j}^s,\\
&\mathfrak{d}_{l,m,j}^s=\tilde c_l^{m,1}\mathfrak{b}_{l,m-1,j}^s-\tilde c_l^{m,2}\mathfrak{b}_{l,m+1,j}^s,\quad m=1,\cdots, l-1,
\end{split}
\end{equation}
and $\mathscr{P}_{lm}^n$ is the same as in  \eqref{blmjint} {\rm(}which can be computed by \eqref{analyticintformula3}{\rm).}
\end{proposition}
\begin{proof}
By definition, we have  $\bs\Psi_0^0=\bs\Phi_0^0=\bs 0,$ so $\tilde v_{00}^{(1)}=\tilde v_{00}^{(2)}=0$. Moreover, using the definition of
 $\{\bs Y_l^m, \bs\Psi_l^m, \bs\Phi_l^m\}$ in \eqref{vecylm}-\eqref{vecphilm}, we obtain
\begin{equation}
\label{formula12}
\hspace{-20pt}
\begin{split}
\tilde v_{lm}^{r}&=\int_{\mathbb{S}^2}u_N^E\overline{Y_l^m}dS,\quad\tilde v_{lm}^{(1)}=\frac{1}{l(l+1)}\Big(\int_{\mathbb{S}^2}{v}_N^E\frac{\partial \overline{Y_l^{m}}}{\partial \theta}dS+\int_{\mathbb{S}^2}w_N^E\frac{\partial \overline{Y_l^{m}}}{\partial \varphi}\frac{1}{\sin\theta}dS\Big),\\
\tilde v_{lm}^{(2)}&=\frac{1}{l(l+1)}\Big(\int_{\mathbb{S}^2}{v}_N^E\frac{\partial \overline{Y_l^{m}}}{\partial \varphi}\frac{1}{\sin\theta}dS-\int_{\mathbb{S}^2}w_N^E\frac{\partial \overline{Y_l^{m}}}{\partial \theta}dS\Big).\\
\end{split}
\end{equation}
Therefore, we need to compute the following integrals
\begin{equation}
\label{vseint}
\begin{split}
&\int_{\mathbb{S}^2}u_N^E\overline{Y_l^m}dS,\quad  \int_{\mathbb{S}^2}v_N^E\frac{\partial \overline{Y_l^{m}}}{\partial \theta}dS, \quad \int_{\mathbb{S}^2}v_N^E\frac{\partial \overline{Y_l^{m}}}{\partial \varphi}\frac{1}{\sin\theta}dS,\\
&\int_{\mathbb{S}^2}w_N^E\frac{\partial \overline{Y_l^{m}}}{\partial \theta}dS, \quad \int_{\mathbb{S}^2}w_N^E\frac{\partial \overline{Y_l^{m}}}{\partial \varphi}\frac{1}{\sin\theta}dS,
\end{split}
\end{equation}
to obtain the vector spherical harmonic coefficients.

We have already derived an analytic formula for the first integral in Proposition \ref{Prop31}, i.e.,
\begin{equation*}
\tilde v_{lm}^r=\int_{\mathbb{S}^2}u_N^E\overline{Y_l^m}dS=\sum\limits_{e=1}^{E}
\sum\limits_{i=0}^{N}\sum\limits_{j=0}^{N}U_{ij}^{e}\mathfrak{a}_{m,i}^t\mathfrak{b}_{l,m,j}^s,
\end{equation*}
where $\mathfrak{a}_{m,i}^t$ and $\mathfrak{b}_{l,m,j}^s$ are given in  \eqref{ablmijkanal}.
For other integrals, we decompose them as
\begin{equation}
\label{vintegral}
\begin{split}
&\int_{\mathbb{S}^2}v_N^E\frac{\partial \overline{Y_l^{m}}}{\partial \theta}dS=\sum\limits_{e=1}^{E}\sum\limits_{i=0}^{N}\sum\limits_{j=0}^N
V_{ij}^{e}\int_{\varphi_{t-1}}^{\varphi_t}\int_{\theta_{s-1}}^{\theta_s}l_i(\xi)l_j(\eta)\frac{\partial \overline{Y_l^{m}}}{\partial \theta}\sin\theta d\theta d\varphi,\\
&\int_{\mathbb{S}^2}v_N^E\frac{\partial \overline{Y_l^{m}}}{\partial \varphi}\frac{1}{\sin\theta}dS=-\ri m\sum\limits_{e=1}^{E}\sum\limits_{i=0}^{N}\sum\limits_{j=0}^N
V_{ij}^{e}\int_{\varphi_{t-1}}^{\varphi_t}\int_{\theta_{s-1}}^{\theta_s}l_i(\xi)l_j(\eta)\overline{Y_l^m} d\theta d\varphi,
\end{split}
\end{equation}
and
\begin{equation}
\label{wintegral}
\begin{split}
&\int_{\mathbb{S}^2}w_N^E\frac{\partial \overline{Y_l^{m}}}{\partial \theta}dS=\sum\limits_{e=1}^{E}\sum\limits_{i=0}^{N}\sum\limits_{j=0}^N
W_{ij}^{e}\int_{\varphi_{t-1}}^{\varphi_t}\int_{\theta_{s-1}}^{\theta_s}l_i(\xi)l_j(\eta)\frac{\partial \overline{Y_l^{m}}}{\partial \theta}\sin\theta d\theta d\varphi,\\
&\int_{\mathbb{S}^2}w_N^E\frac{\partial \overline{Y_l^{m}}}{\partial \varphi}\frac{1}{\sin\theta}dS=-\ri m\sum\limits_{e=1}^{E}\sum\limits_{i=0}^{N}\sum\limits_{j=0}^N
W_{ij}^{e}\int_{\varphi_{t-1}}^{\varphi_t}\int_{\theta_{s-1}}^{\theta_s}l_i(\xi)l_j(\eta)\overline{Y_l^m} d\theta d\varphi.
\end{split}
\end{equation}
We see that they actually involve  two integrals
\begin{equation}
\label{vecseprateint}
\int_{\varphi_{t-1}}^{\varphi_t}\int_{\theta_{s-1}}^{\theta_s}l_i(\xi)l_j(\eta)\frac{\partial \overline{Y_l^{m}}}{\partial \theta}\sin\theta d\theta d\varphi,
\quad\int_{\varphi_{t-1}}^{\varphi_t}\int_{\theta_{s-1}}^{\theta_s}l_i(\xi)l_j(\eta)\overline{Y_l^m} d\theta d\varphi.
\end{equation}
Applying the mapping $\mathcal{F}_{e}$ and following the proof of Proposition \ref{Prop31}, we derive
\begin{equation*}
\begin{split}
&\int_{\varphi_{t-1}}^{\varphi_t}\int_{\theta_{s-1}}^{\theta_s}l_i(\xi)l_j(\eta)\overline{Y_l^m} d\theta d\varphi
=\mathfrak{a}_{m,i}^t\mathfrak{c}_{l,m,j}^s,\\
&\int_{\varphi_{t-1}}^{\varphi_t}\int_{\theta_{s-1}}^{\theta_s}l_i(\xi)l_j(\eta)\frac{\partial \overline{Y_l^{m}}}{\partial \theta}\sin\theta d\theta d\varphi=\mathfrak{a}_{m,i}^t\mathfrak{d}_{l,m,j}^s,
\end{split}
\end{equation*}
where $\mathfrak{a}_{m,i}^t$ is defined in \eqref{independentint} and
\begin{equation}
\label{clmjdlmj}
\begin{split}
&\mathfrak{c}_{l,m,j}^s:=\hat{\theta}_s\int_{-1}^{1}l_j(\eta)\widehat P_l^m(\cos(\theta(\eta)))d\eta,\\
&\mathfrak{d}_{l,m,j}^s:=\hat{\theta}_s\int_{-1}^{1}l_j(\eta)\frac{d}{d\theta}\widehat P_l^m(\cos(\theta(\eta)))\sin(\theta(\eta))d\eta.
\end{split}
\end{equation}
Then, using \eqref{vintegral}, \eqref{wintegral} and \eqref{clmjdlmj} in \eqref{formula12} we obtain \eqref{vecsphcoefformula2} and \eqref{vecsphcoefformula3}.

As in the calculation of $\mathfrak{b}_{l,m,j}^s$ in \eqref{independentint}, we derive analytic formulas for $\mathfrak{c}_{l,m,j}^s$ and $\mathfrak{d}_{l,m,j}^s$ by using very similar techniques. Recalling the definition \eqref{blmjint}, we have
\begin{equation}
\label{clmj}
\mathfrak{c}_{l,m,j}^s=\hat{\theta}_s\int_{-1}^{1}l_j(\eta)\widehat P_l^m(\cos(\hat{\theta}_s\eta+\alpha_s))d\eta=\hat{\theta}_s\mathscr{P}_{lm}^n(\hat{\theta}_s,\alpha_s).
\end{equation}
The integration $\mathfrak{d}_{l,m,j}^s$ can be reformulated into $\mathfrak{b}_{l,m,j}^s$ by using recurrence formulas \eqref{deriassleg} and \eqref{deriasslegspecial}. We first consider the special cases with $m=0, l$. From  \eqref{deriasslegspecial}, we have
\begin{equation}
\label{dlmj1}
\begin{split}
&\mathfrak{d}_{l,0,j}^s=-\sqrt{l(l+1)}\hat{\theta}_s\int_{-1}^{1}l_j(\eta)\widehat P_l^{1}(\cos(\theta^s(\eta)))\sin(\theta^s(\eta))d\eta=-\sqrt{l(l+1)}\mathfrak{b}_{l,1,j}^s,\\
&\mathfrak{d}_{l,l,j}^s=\sqrt{\frac{l}{2}}\hat{\theta}_s\int_{-1}^{1}l_j(\eta)\widehat P_l^{l-1}(\cos(\theta^s(\eta)))\sin(\theta^s(\eta))d\eta=\sqrt{\frac{l}{2}}\mathfrak{b}_{l,l-1,j}^s.
\end{split}
\end{equation}
Then for $0<m<l,$  inserting \eqref{deriassleg} into \eqref{clmjdlmj} yields
\begin{equation}
\label{dlmj2}
\begin{split}
\mathfrak{d}_{l,m,j}^s=&\hat{\theta}_s\int_{-1}^{1}l_j(\eta)\big[\tilde c_l^{m,1}\widehat P_l^{m-1}(\cos(\theta^s(\eta)))-\tilde c_l^{m,2}\widehat P_l^{m+1}(\cos(\theta^s(\eta)))\big]\sin(\theta^s(\eta))d\eta\\
=&\tilde c_l^{m,1}\mathfrak{b}_{l,m-1,j}^s-\tilde c_l^{m,2}\mathfrak{b}_{l,m+1,j}^s.
\end{split}
\end{equation}
Assembling the above formulas leads to the desired results.
\end{proof}

\begin{remark}\label{byproduct2} As with Remark  {\rm \ref{byproduct}}, the above algorithm applies to  the VSH expansion  of a given vector field $\bs v$ on sphere $\mathbb{S}^2$ with a spectral element approximation on the same partition $\mathbb{S}_h^2$.  \qed
\end{remark}

Observe from the above that the VSH expansion can be performed with a constant multiple of the cost for the SPH expansion.
To have more insights into this, we define
$$\bs V^e=\begin{bmatrix}
\bs V_0^e & \bs V_1^e & \cdots &  \bs V_N^e
\end{bmatrix}, \quad\bs V^e_j=\begin{bmatrix}
v_{0j}^e & v_{1j}^e & \cdots & v_{Nj}^e
\end{bmatrix}',$$
and likewise $\bs W^e=(\bs W^e_j)'$, $\bs W^e_j=(w_{ij}^e)'$, $\bs C^s_{lm}=(\mathfrak{c}_{l,m,j}^s)'$ and $\bs D^s_{lm}=(\mathfrak{d}_{l,m,j}^s)'.$
Then, we can formulate    \eqref{vecsphcoefformula1}-\eqref{vecsphcoefformula3} into the following matrix form:
\begin{equation*}
\begin{split}
\tilde v_{lm}^r&=\sum\limits_{e=1}^E((\bs A^t_m)'\bs U^e)\bs B^s_{lm},\quad
\tilde v_{lm}^{(1)}=\frac{1}{l(l+1)}\sum\limits_{e=1}^E\big\{((\bs A^t_m)'\bs V^e)\bs D^s_{lm}-\ri m((\bs A^t_m)'\bs W^e)\bs C^s_{lm}\big\},\\
\tilde v_{lm}^{(2)}&=\frac{-1}{l(l+1)}\sum\limits_{e=1}^E\big\{((\bs A^t_m)'\bs W^e)\bs D^s_{lm}+\ri m((\bs A^t_m)'\bs V^e)\bs C^s_{lm}\big\}.
\end{split}
\end{equation*}

\subsection{Accuracy test}
We now provide some illustrative numerical examples to show the high accuracy of the aforementioned method. For simplicity,
we consider the partition  ${\mathbb S}_h^2$ by equi-spaced latitude lines $\theta=s\pi/\mathcal{N},\; s=1, 2, \cdots, \mathcal{N}-1,$ and longitude lines $\varphi=2\pi t/\mathcal{M},\; t=0, 1, \cdots, \mathcal{M}-1$. We test the expansions of  the plane wave $u(\bs x)=e^{\ri k\hat{\bs k}\cdot\bs x}$ in SPH and the vector field
\begin{equation}\label{vfield}
\bs v(\bs x)=\nabla u+\nabla u\times \bs e_r=\ri k\Big(\hat{\bs k}+\frac{\hat{\bs k}\times\bs x}{|\bs x|}\Big) e^{\ri k\hat{\bs k}\cdot\bs x}\end{equation}
in VSH, where $k\hat{\bs k}$ is the propagation vector. Given a spherical surface ${\mathbb S}_{R}^2:=\{\bs x: |\bs x|=R\}$, the plane wave $u(\bs x)=e^{\ri k\hat{\bs k}\cdot\bs x}$ confined on ${\mathbb S}_{R}^2$ has the SPH expansion: 
\begin{equation*}\label{ubxtest}
u(\bs x)=\sum\limits_{l=0}^{\infty}\sum\limits_{m=-l}^{l}a_l^mY_l^m(\theta,\varphi)\;\;\hbox{with}\;\;a_l^m=4\pi\ri^lj_l(kR)\overline{Y_l^m}(\hat{\bs k}),
\end{equation*}
derived from the Funk-Hecke formula \cite[P. 72]{martin2006multiple}. Here, $(R,\theta,\varphi)$ is the spherical coordinate of $\bs x\in {\mathbb S}_{R}^2$. Thus, we find the VSH expansion of $\nabla u$ on ${\mathbb S}_{R}^2$ is
$$\nabla u(\bs x)=kj_0'(kR)\bs e_r+\sum\limits_{l=1}^{\infty}\sum\limits_{m=-l}^{l}4\pi\ri^l\overline{Y_l^m}(\hat{\bs k})\Big(kj_l'(kR)\bs Y_l^m(\theta,\varphi)+\frac{j_l(kR)}{R}\bs\Psi_l^m(\theta,\varphi)\Big).$$
Consequently,  the exact VSH expansion coefficients $\{v_{lm}^r, v_{lm}^{(1)}, v_{lm}^{(2)}\}$ of $\bs v(\bs x)$ on ${\mathbb S}_{R}^2$ are
\begin{equation*}
\{v_{lm}^r, v_{lm}^{(1)}, v_{lm}^{(2)}\}=4\pi\ri^l\overline{Y_l^m}(\hat{\bs k})\Big\{kj_l'(kR), \frac{j_l(kR)}{R}, \frac{j_l(kR)}{R}\Big\},\;\;\;\; l\geq 1,
\end{equation*}
and
$\big\{v_{00}^r, v_{00}^{(1)}, v_{00}^{(2)}\big\}=\big\{kj_0'(kR), 0, 0\big\}.$

Without loss of generality, we fix $\hat{\bs k}=(1, 1, 1)$,   and test several examples with various wavenumber  $k$. For the SPH expansion of the scalar $u(\bs x)$, we examine the error:
$$E_L(u):=\max_{|l|\leq L}\max_{|m|\leq l}|a_l^m-\tilde a_l^m|,$$
while the error:
$$E_L(\bs v):=\max_{|l|\leq L}\max_{|m|\leq l}\Big\{|v_{lm}^r-\tilde{v}_{lm}^r|, |v_{lm}^{(1)}-\tilde{v}_{lm}^{(1)}|,|v_{lm}^{(2)}-\tilde{v}_{lm}^{(2)}|\Big\},$$
is examined for the VSH expansion  of $\bs v(\bs x)$ in \eqref{vfield}.
We  depict the errors against $N$ (with fixed partition $\mathcal{N}=3$, $\mathcal{M}=4$), against the number of elements $N_S$ with
$\mathcal{N}=\mathcal{M}=N_S$ (and fixed $N=10$), and against the degree of freedom  (dof) of the spectral element approximations with $\mathcal{N}=\mathcal{M}=N$ in Figures \ref{sphexp01}-\ref{sphexp02}, respectively, for $R=1$, $L=20$ and $k=10, 20, 30$. We observe from the above figures an exponential decay of the errors as the number of points increases even for high mode expansions.  
\begin{figure}[!ht]
	\subfigure[$\mathcal{N}=3$, $\mathcal{M}=4$ fixed]{\includegraphics[scale=.25]{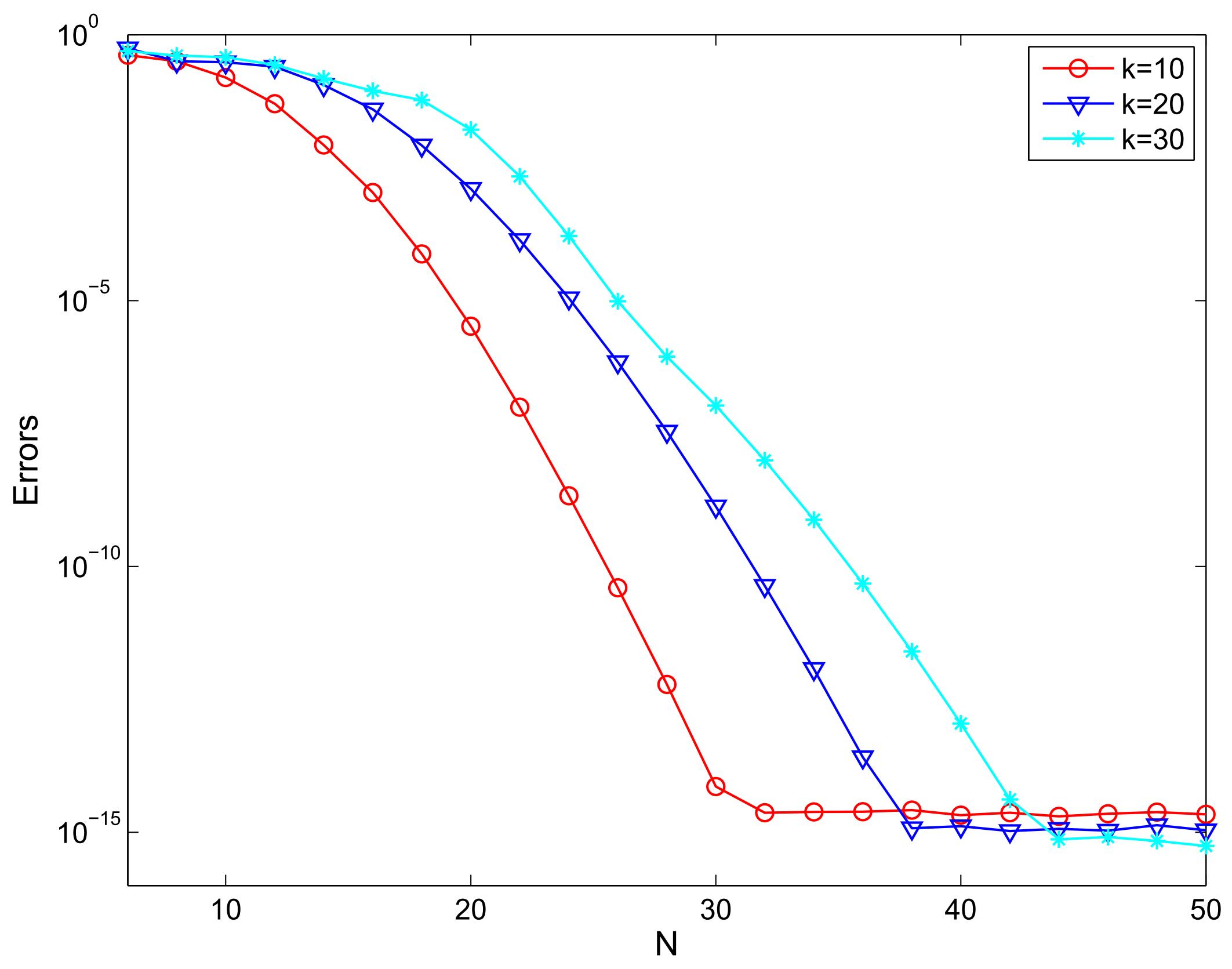}}\quad
	\subfigure[$N=10$ fixed, $\mathcal{N}=\mathcal{M}=N_S$]{\includegraphics[scale=.25]{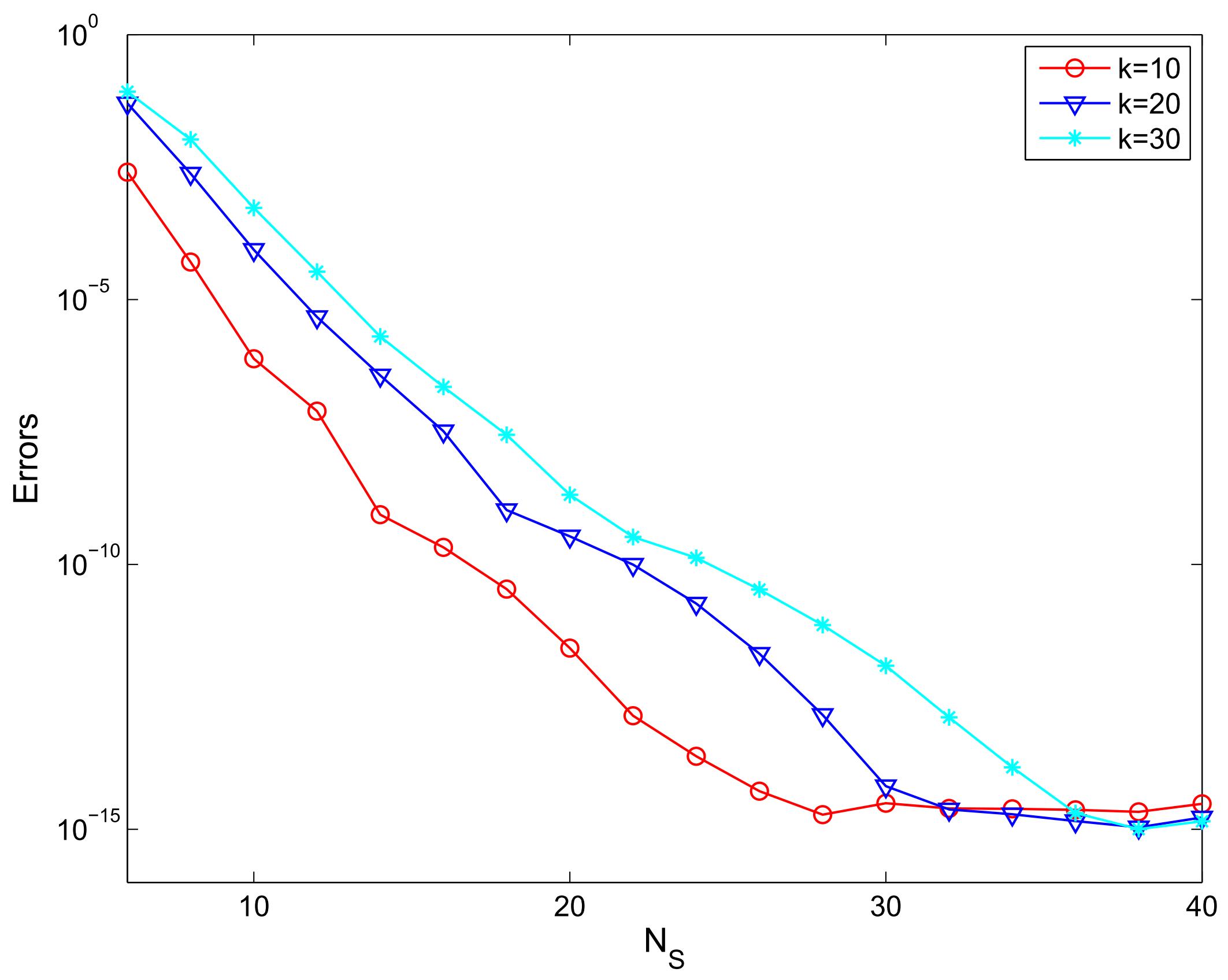}}\quad
	\subfigure[$\mathcal{N}=\mathcal{M}=N$]{\includegraphics[scale=.25]{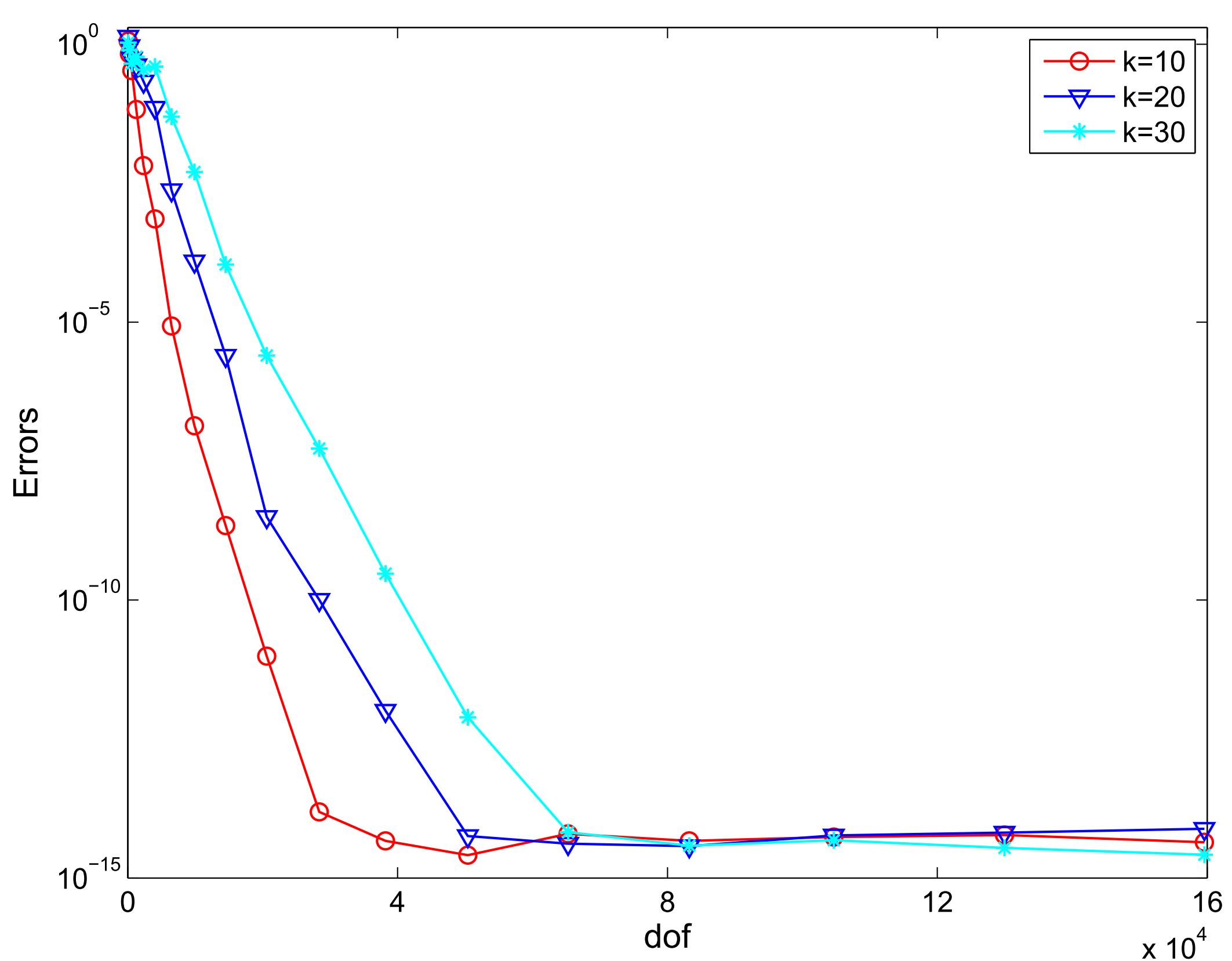}}
	\caption{\small SPH expansion errors $E_{20}(u)$ against $N$, $N_S$ and $dof$ respectively.} \label{sphexp01}
\end{figure}
\begin{figure}[!ht]	
	\subfigure[$\mathcal{N}=3$, $\mathcal{M}=4$ fixed]{\includegraphics[scale=.25]{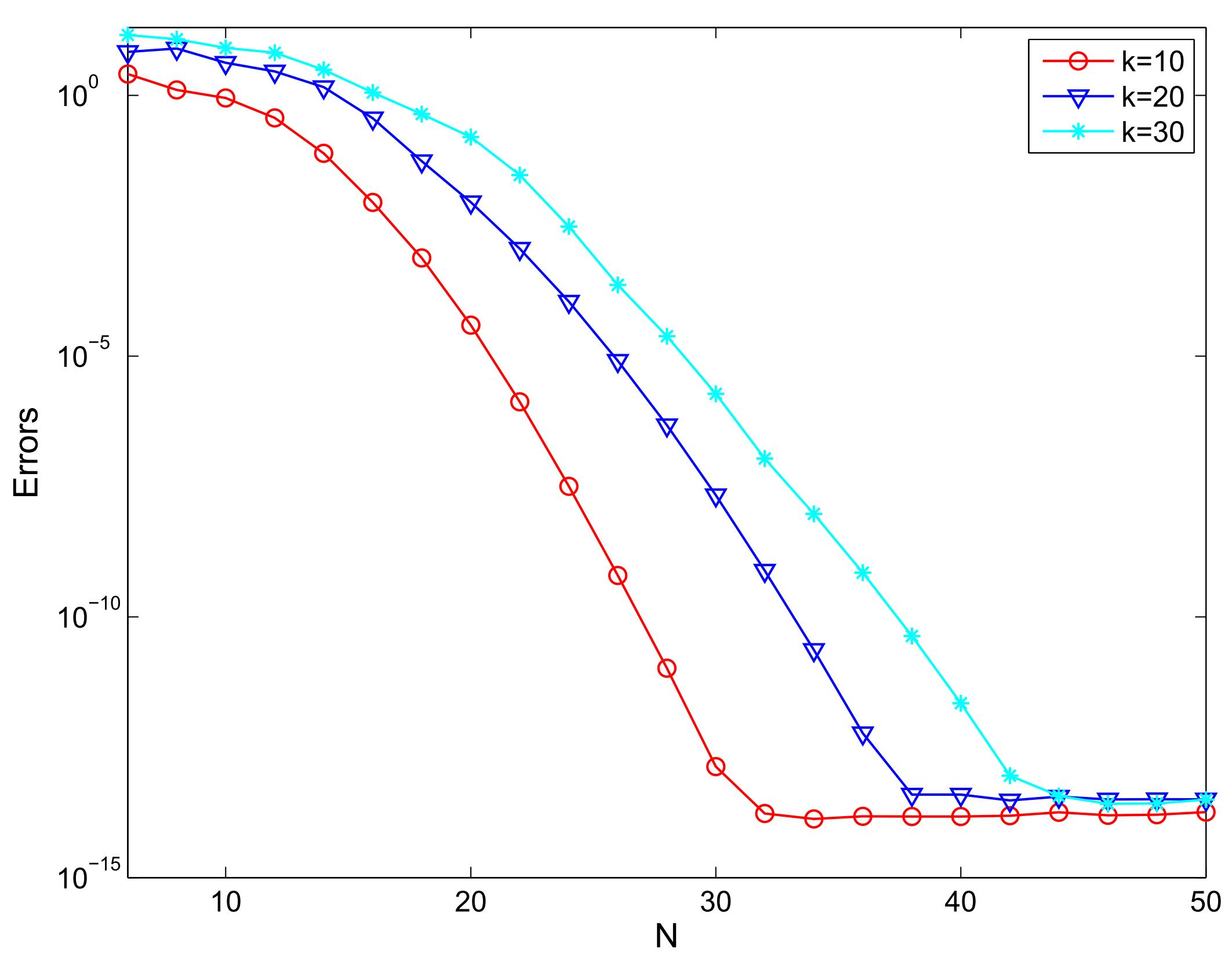}}\quad
	\subfigure[$N=10$ fixed, $\mathcal{N}=\mathcal{M}=N_S$]{\includegraphics[scale=.25]{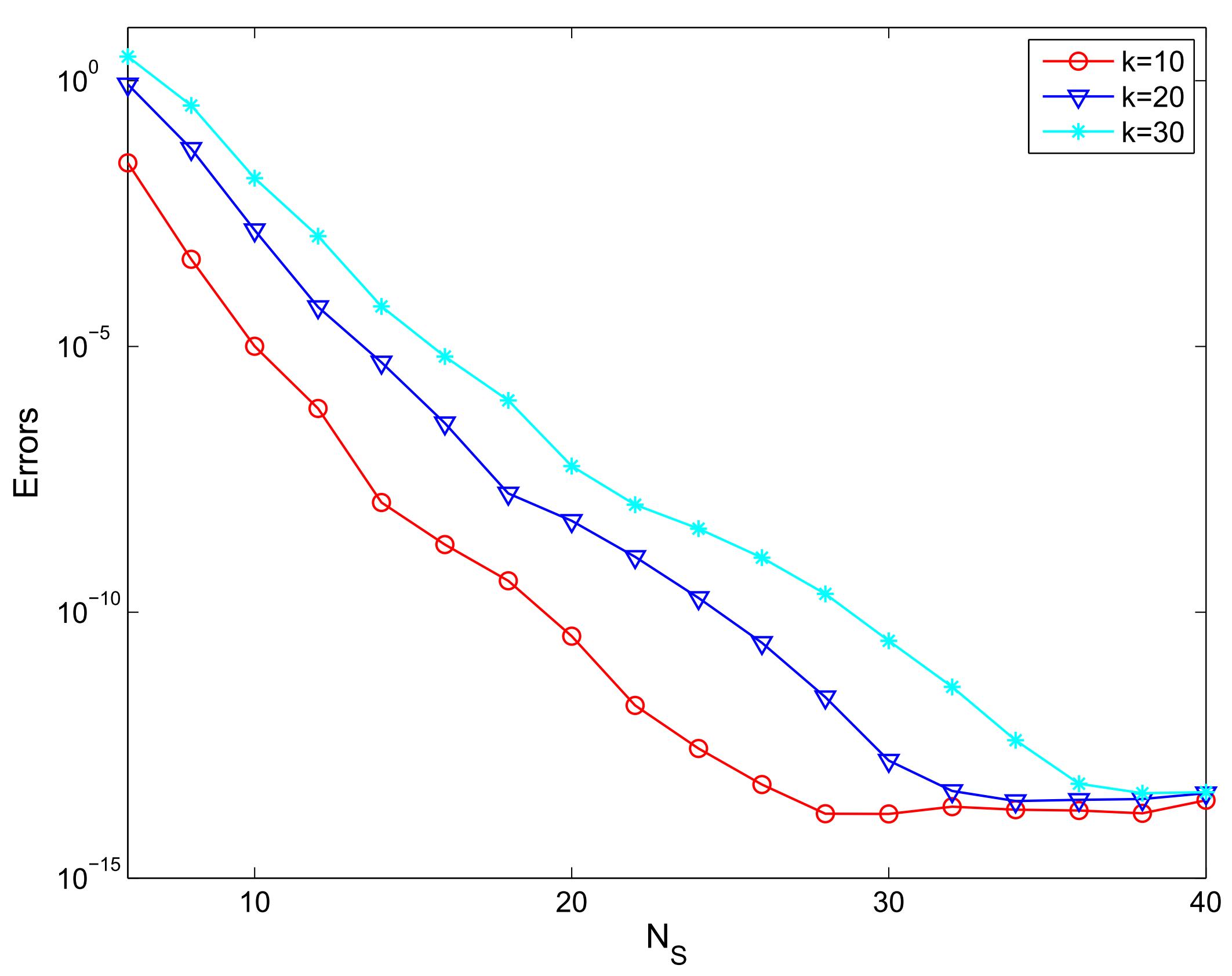}}\quad
	\subfigure[$\mathcal{N}=\mathcal{M}=N$]{\includegraphics[scale=.25]{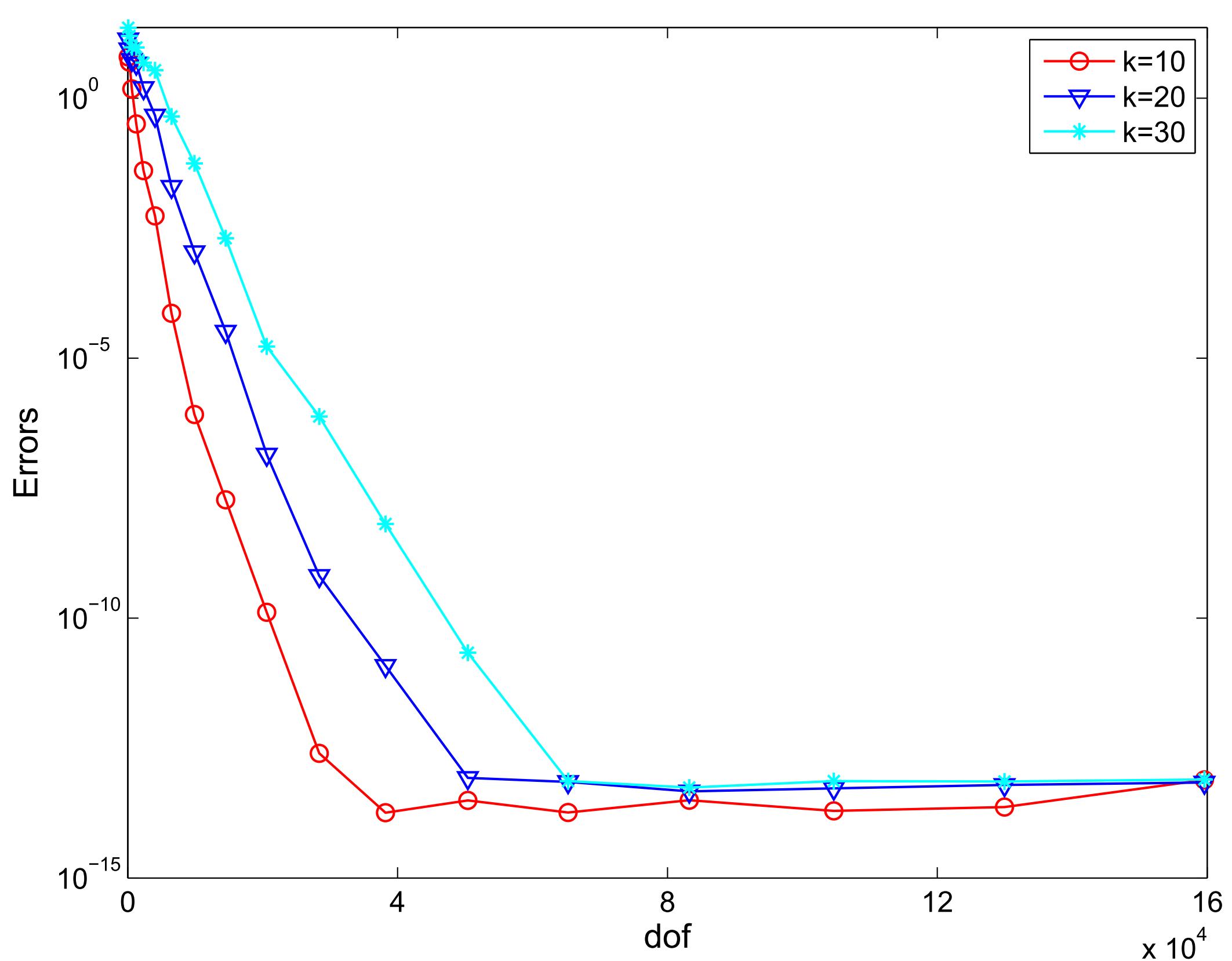}}
	\caption{\small VSH expansion errors $E_{20}(\bs v)$ against $N$, $N_S$ and $dof$ respectively. } \label{sphexp02}
\end{figure}

One important feature of the proposed algorithm is that it can produce highly accurate expansion for large  wave number.  If the given spectral element approximation is accurate enough, the algorithm can calculate arbitrary high mode coefficients without inducing additional errors. We now provide a numerical example to show the capability of the presented algorithm in the computation of high modes. Consider again the spherical harmonic expansion of the plane wave $u(\bs x)=e^{\ri k\hat{\bs k}\cdot\bs x}$. Here, we take $\hat{\bs k}=(1, 0, 0)$ and $k=100, 200,$ respectively. We fix the partition $\mathcal{N}=3, \mathcal{M}=4$ and use polynomials of degree $N=120$ and $N=220$ for spectral element approximation. According to our calculation, the spectral element approximation errors for $e^{\ri k\hat{\bs k}\cdot\bs x}$ are $7.4847e$-$14$ and $1.55806e$-$13$, respectively.  Define $\tilde{a}_{l}=\max\limits_{0\leq |m|\leq l}|\tilde{a}_{l}^m|,$
and list the results in Table \ref{highwavenumbertab}.  We see that it is sufficient to set $L=145, 260$ for $k=100, 200$ to achieve machine accuracy as shown in Table \ref{highwavenumbertab}, which also indicates the high efficiency of the proposed algorithm.
\begin{table}[h]
	\caption{Algorithm performance for high wave number.}
	\begin{center}
		\begin{tabular}{|c|c|c|c|c|c|c|c|}
			\hline
			\multicolumn{4}{|c|}{$k=100$} & \multicolumn{4}{|c|}{$k=200$}\\
			\hline
			$l$& $\tilde{a}_l$ & $E_{145}(u)$ & CPU time (s)  & $l$ & $\tilde{a}_l$& $E_{260}(u)$ & CPU time (s)\\
			\hline
			141& 1.1973e-12  & \multirow{5}*{1.8241e-14} & \multirow{5}*{1.48} & $256$ & 3.0841e-14& \multirow{5}*{3.8124e-14} & \multirow{5}*{8.24}\\
			\cline{1-2}\cline{5-6}
			142& 4.9071e-13  & &  & $257$ & 1.5661e-14& &\\
			\cline{1-2}\cline{5-6}
			143& 1.9962e-13  & &  & $258$ & 6.8950e-15& &\\
			\cline{1-2}\cline{5-6}
			144& 8.0668e-14  & &  & $259$ & 4.7365e-15& &\\
			\cline{1-2}\cline{5-6}
			145& 3.1744e-14  & &  & $260$ & 2.0108e-15& &\\
			\hline
		\end{tabular}
	\end{center}
	\label{highwavenumbertab}
\end{table}

\section{Applications to wave scattering simulation}\label{sect4:appl}
\setcounter{equation}{0}

In this section, we   solve  time-harmonic wave scattering problems numerically in several scenarios,  where the proposed algorithm for SPH and VSH expansions play an important part.  We start with a relatively simple situation of computing (far-field) acoustic and electromagnetic waves scattered by a spherical scatterer, where much care must be devoted to properly dealing with the involved ratio of the Hankel functions.  
We then   consider an irregular scatterer by   integrating the above far-field computation  with 
the  transformed field expansion (TFE) method \cite{nicholls2006a,fang2007a}, and the spherical Dirichlet-to-Neumann (DtN) transparent boundary condition. Finally, we extend the approach to multiple scatterers.   
 

%
%
%


\subsection{Computing scattering waves by a single spherical scatterer}   
\label{singlesphscat}
Consider the time-harmonic wave scattering  governed by 
\begin{equation}\label{scatterprob}
\left\{
\begin{split}
& \Delta u+k^2 u=0,\quad  r>b, \\
& u|_{r=b}=u_N,\quad \frac{\partial u}{\partial r}-{\rm i}ku=O\Big(\frac{1}{r}\Big),\quad \text{as}\;\; r\rightarrow\infty,
\end{split}\right.
\end{equation}
and
\begin{subequations}\label{extthmaxwellsecondorder}
	\begin{numcases}{}
	{\nabla\times\nabla\times} {\bs E}-k^2 {\bs E}=\bs 0,  \qquad\qquad \;\;\; \text{in } \;\;  r>b \\
	{\bs E}_S=\hat{\bs x}\times ({\bs E}_N\times \hat{\bs x}),\qquad\qquad\quad\,\, \hbox{on}\;\;r=b,\label{scatterbc}\\
	\lim_{r\to \infty} r\big(\,\nabla\times\bs E\times \bs{\hat x}-\ri k \bs E\big)=\bs 0,      \label{silvermullerbc}
	\end{numcases}
\end{subequations}
with a spherical scatterer $B(b)$ (note: $B(b)$ is a ball of radius $b$), where $\hat{\bs x}:=\bs x/|\bs x|$, $k$ is the wave number,  $\bs E_S=\hat{\bs x}\times ({\bs E}_N\times \hat{\bs x})$ (note: $\hat{\bs x}$ unit outward normal of $B(b)$) is the tangential component of $\bs E$, and the boundary data $u_N$ and ${\bs E}_N$ are assumed to be spectral element approximations resulted from given incident waves or provided by other numerical solvers. For acoustic scattering problem \eqref{scatterprob}, sound soft boundary condition on the scatter $B(b)$ and Sommerfeld radiation condition at the infinity are used. On the other hand, the perfect conduct boundary condition on $B(b)$ and the Silver-Muller radiation condition at infinity are imposed  for the electromagnetic scattering problem \eqref{extthmaxwellsecondorder}. The exterior solver presented below can be combined with appropriate interior solver (cf \cite{acosta2010coupling}) to solve multiple scattering problems with irregular scatterers.

%

Using  the separation variable method together with SPH and VSH expansion (cf. \cite{Nede01}) leads to approximations:
\begin{equation}
\label{approxsolution}
u_L=\sum\limits_{l=0}^{L}\sum\limits_{m=-l}^{l}\frac{\widehat{U}_{N,l}^m}{h_l^{(1)}(kb)}\psi_{lm}(\bs x),\;\;\; \psi_{lm}(\bs x):=h_l^{(1)}(kr)Y_l^m(\theta,\varphi),
\end{equation}
and
\begin{equation}
\label{EHsoluA1}
\begin{split}
{\bs E}_L= &\sum_{l=1}^L\sum_{|m|=0}^l\bigg\{ -\frac{l(l+1)B_l^m}{{\rm i}k}\frac{h_l^{(1)}(kr)}{r} \bs Y_l^m-\frac{B_l^mZ_l(kr)}{{\rm i}k r} \bs\Psi_l^m+A_l^mh_l^{(1)}(kr) \bs\Phi_l^m \bigg\}, \\
{\bs H}_L = &\sum_{l=1}^L\sum_{|m|=0}^l\bigg\{ \frac{l(l+1)A_l^m}{{\rm i}k}\frac{h_l^{(1)}(kr)}{r} \bs Y_l^m
+\frac{A_l^mZ_l(kr)}{{\rm i}k r} \bs\Psi_l^m +B_l^mh_l^{(1)}(kr) \bs\Phi_l^m\bigg\},
\end{split}
\end{equation}
where $L$ is the cut-off number, $\{\widehat{U}_{N,l}^m\}$ are the SPH coefficients of $u_N$ on the spherical surface $r=b$, $\{\psi_{lm}(\bs x)\}$ is the outgoing spherical wave functions (cf. \cite{martin2006multiple}), $\{h_l^{(1)}(z)\}$ are the spherical Hankel functions of the first kind, 
\begin{equation}\label{Zlzque}
Z_l(z):=h_l^{(1)}(z)+zh_l^{(1)'}(z),
\end{equation} 
and the coefficients $\{A_l^m, B_l^m\}$ need to be determined by matching the boundary data on $r=b$.
Now, we use the algorithm proposed in Section \ref{sect3:expan}  to compute the SPH coefficients $\{\widehat{U}_{N,l}^m\}$ and the approximate VSH expansion
\begin{equation}
\label{tanbcdata}
\hat{\bs x}\times ({\bs E}^s_N\times \hat{\bs x})=\sum_{l=1}^L\sum_{|m|=0}^l\big\{
V_{N,l}^{m} \bs\Psi_l^m+W_{N, l}^{m} \bs\Phi_l^m \big\}.
\end{equation}
Then matching the boundary condition \eqref{scatterbc} leads to
\begin{equation*}
A_l^m=\frac{W_{N, l}^{m}}{h_l^{(1)}(kb)},\quad B_l^m=-\frac{\ri k b\, V_{N, l}^{m}}{Z_l(kb)}.
\end{equation*}
Substituting it into \eqref{EHsoluA1}, we obtain 
\begin{align}
{\bs E} _L= &\sum_{l=1}^L\sum_{|m|=0}^l\bigg\{ b\varpi_l V_{N, l}^{m}\frac{h_l^{(1)}(kr)}{rZ_l(kb)} \bs Y_l^m
+b V_{N, l}^{m}\frac{Z_l(kr)}{rZ_l(kb)} \bs\Psi_l^m+W_{N, l}^{m}\frac{h_l^{(1)}(kr)}{h_l^{(1)}(kb)} \bs\Phi_l^m \bigg\}, \label{EsoluA1}\\
{\bs H} _L= &\sum_{l=1}^L\sum_{|m|=0}^l\bigg\{ \frac{\varpi_lW_{N, l}^{m}}{{\rm i}k }\frac{h_l^{(1)}(kr)}{rh_l^{(1)}(kb)} \bs Y_l^m
+\frac{W_{N, l}^{m}}{{\rm i}k }\frac{Z_l(kr)}{rh_l^{(1)}(kb)} \bs\Psi_l^m -\ri k b V_{N, l}^{m}\frac{h_l(kr)}{Z_l(kb)} \bs\Phi_l^m\bigg\}. \label{HsoluA1}
\end{align}
We reiterate  that the error in the approximations \eqref{approxsolution} and \eqref{EsoluA1}-\eqref{HsoluA1} only comes from the mode of truncation because the computation of the SPH coefficients $\widehat{U}_{N,l}^m$ and the VSH coefficients $\{V_{N,l}^m, W_{N,l}^m\}$ does not induce any additional error that exceeds accumulated computer round-off error.

One challenge for computing the far-field scattering waves lies in that the naive calculation of the ratio ${\widehat{U}_{N,l}^m}/{h_l^{(1)}(kb)}$ is not stable in performing the superposition
\eqref{approxsolution}.  This is due to the ``bad"  asymptotic behaviour of $j_{l}(z)$ and $y_{l}(z)$ (cf. \cite{Abr.S84}):
\begin{equation}
\label{besselasymptotic}
j_{l}(z)\sim\frac{1}{2l+1}\sqrt{\frac{1}{2\pi z}}\Big(\frac{ez}{2l+1}\Big)^{l+\frac{1}{2}},\quad y_{l}(z)\sim-\frac{1}{2l+1}\sqrt{\frac{8}{\pi z}}\Big(\frac{ez}{2l+1}\Big)^{-l-\frac{1}{2}},\quad l\gg 1,
\end{equation}
which induces numerical underflow/overflow  in computing  the ratio ${\widehat{U}_{N,l}^m}/{h_l^{(1)}(ka)}$ and $\psi_{lm}(\bs x)$   for large $l$.  Therefore, the usual approach for the direct computation of the scattering field works  only for
lower mode $L$ and small $r$.  In order to overcome this obstacle, we rearrange the terms in \eqref{approxsolution} as
\begin{equation}
\label{approxsolution1}
u_L=\sum\limits_{l=0}^{L}\sum\limits_{|m|=0}^{l}\widehat{U}_{N,l}^mR_l(r)Y_l^m(\theta,\varphi),\;\; \hbox{where}\;\;R_l(r):=\frac{h_l^{(1)}(kr)}{h_l^{(1)}(kb)}.
\end{equation}
It is important to point out that  the  ratio $R_l(r)$ is well behaved for all $l$ and $r\geq b$, which can be evaluated efficiently as follows.
For this purpose, we reformulate the ratio as
\begin{equation}\label{ratioode2}
R_l'(r)-k\rho_l(kr)R_l(r)=0, \;\;\;  r>b;  \quad R_l(b)=1,
\end{equation}
where
\begin{equation}
\label{logarithmic}
\rho_l(z):=\frac{h_{l}^{(1)'}(z)}{h_{l}^{(1)}(z)}=\frac{d}{dz}\log h_{l}^{(1)}(z),\quad z>0.
\end{equation}
Equivalently, we have
\begin{equation}
\label{ratiohankel}
R_l(r)=\exp\Big(k \int_b^r\rho_l(k\xi)d\xi\Big),\quad l\ge 0,\;\; r>b.
\end{equation}
\begin{remark}\label{rholk} Recall the properties of $\rho_l(z)$ {\rm(}cf.  \cite{Nede01,She.W07}{\rm)}:
\begin{equation*}
\rho_0(z)=-\frac 1 z+\ri; \;\;\; -\frac {l+1} z\le{\rm Re}(\rho_l(z))\le -\frac 1 {z}, \quad  0<{\rm Im}(\rho_l(z))\le 1,
\end{equation*}
for all $l\ge 1$ and $z>0,$ and  ${\rm Im}(\rho_l(z))$ becomes exponentially small for large $l.$ This implies  $|R_l(r)|\le e^{b-r}$ for
$r>b$ and all $l\ge 1.$ \qed
\end{remark}

Note that $\rho_l$ can be evaluated  recursively and stably  by (see, e.g., \cite{alpert2000rapid})
\begin{equation}\label{logarithmicderi}
\rho_l(z)=\frac{z}{l-1-z\rho_{l-1}(z)}-\frac{l+1}{z}, \quad l\ge 1;   \;\;\; \rho_0(z)=-\frac 1 z+\ri.
\end{equation}
With this, we can use a  suitable (composite) quadrature rule  and evaluate the ratio \eqref{ratiohankel} in a very accurate manner.

The above algorithm also applies to the computation of \eqref{EsoluA1}-\eqref{HsoluA1}. Indeed, except for  $R_l(r)$  in \eqref{approxsolution}, we also need to deal with
\begin{align}
\tilde R_l(r):=\frac{Z_l(kr)}{Z_l(kb)}=\frac{h_l^{(1)}(kr)+krh_l^{(1)'}(kr)}{h_l^{(1)}(kb)+kah_l^{(1)'}(kb)}
=\frac{R_l(r)+rR'_l(r)}{1+kb\rho_l(b)},\label{Rla}\\
\breve R_l(r):=\frac{h^{(1)}_l(kr)}{Z_l(kb)}=\frac{h_l^{(1)}(kr)}{h_l^{(1)}(kb)+kbh_l^{(1)'}(kb)}
=\frac{R_l(r)}{1+kb\rho_l(b)},\label{Rlb}
\end{align}
where we used \eqref{Zlzque},  and $\rho_l(r)$ is defined in \eqref{logarithmic}. Therefore, the ratios $\tilde R_l$ and $\breve R_l$ can be computed by using the ratios $R_l(r)$ and $\rho_l(r)$ accurately.

\vskip 2pt

We provide some numerical illustrations of the above algorithm. 
 For the acoustic scattering problem, we take the incident wave $u_N$ to be the spectral-element interpolation of the plane wave $e^{\ri k\hat{\bs k}\cdot\bs x}$, and the spherical wave $e^{\ri k|\bs x-\bs x_0|}$, where $\bs x_0$ is the center of the spherical wave and $\hat{\bs k}$ is the propagation vector. For electromagnetic scattering test, we take $\bs E_N$ to be the spectral element interpolation of the plane wave $(e^{\ri kz}, e^{\ri kz}, 0)^{\rm T}$. In our example, we set $\hat{\bs k}=(1,0,0)$, $\bs x_0=(1, 0, 0)$, $k=40$, $B(b)$ be the ball centered at origin with radius $b=0.25$.  For the spectral element interpolation, we use a uniform $3\times 4$ mesh in the $\theta$-$\varphi$ plane for the spherical surface $r=b$ and set $N=50$ to obtain machine accuracy. According to our computation, the interpolation errors for $e^{\ri k\hat{\bs k}\cdot\bs x}$, $e^{\ri k|\bs x-\bs x_0|}$ and vector field $(e^{\ri kz}, e^{\ri kz}, 0)^{\rm T}$ in a discrete maximum norm are $1.1233e$-$14$, $2.90127e$-$14$ and $1.3110e$-$14$, respectively. Define
$$\widehat{U}_{N,l}=\max\limits_{0\leq |m|\leq l}|\widehat{U}_{N,l}^m|\quad{V}_{N,l}=\max\limits_{0\leq |m|\leq l}|{V}_{N,l}^m| \quad {W}_{N,l}=\max\limits_{0\leq |m|\leq l}|{W}_{N,l}^m|,$$
and list the computed values from $l=30$ to $l=35$ in Table \ref{tabcubic1}. These results show that it is enough to set the truncation mode $L=35$ for $k=40$.
We depict the  scattering waves in Figure \ref{singlescatter01-01} and Figure \ref{emspherescatter01-03}.

\begin{table}[h]
	\caption{Magnitude of spherical harmonic coefficients (single spherical scatterer).}
	\begin{center}
		\begin{tabular}{|c|c|c|c|c|}
			\hline
            \multirow{2}{*}{$l$} & \multicolumn{2}{|c|}{$\widehat{U}_{N,l}$} & \multirow{2}{*}{${V}_{N,l}$} & \multirow{2}{*}{${W}_{N,l}$}\\
			\cline{2-3}
			& plane incident wave   & spherical incident wave & &\\
			\hline
			30 &  2.2271e-12  &  2.6243e-12    & 4.7366e-13 & 1.6125e-13  \\
			\hline
			31 &  3.6545e-13  &  4.3957e-13    & 7.8593e-14  &  2.5827e-14  \\
			\hline
			32 &  5.8040e-14  &  7.1058e-14    & 1.2605e-14  &  3.9934e-15    \\
			\hline
			33 &  8.4522e-15  &  1.1510e-14    & 1.9629e-15  &  6.0101e-16   \\
			\hline
			34 &  1.3907e-15  &  1.8552e-15    & 2.9224e-16  &  8.0510e-17   \\
			\hline
			35 &  2.5900e-16  &  6.2305e-16    & 3.2677e-17  &  7.3894e-18    \\
			\hline		
		\end{tabular}
	\end{center}
	\label{tabcubic1}
\end{table}

\begin{figure}[!ht]
\begin{center}
	\subfigure[Real part]{\includegraphics[scale=0.14]{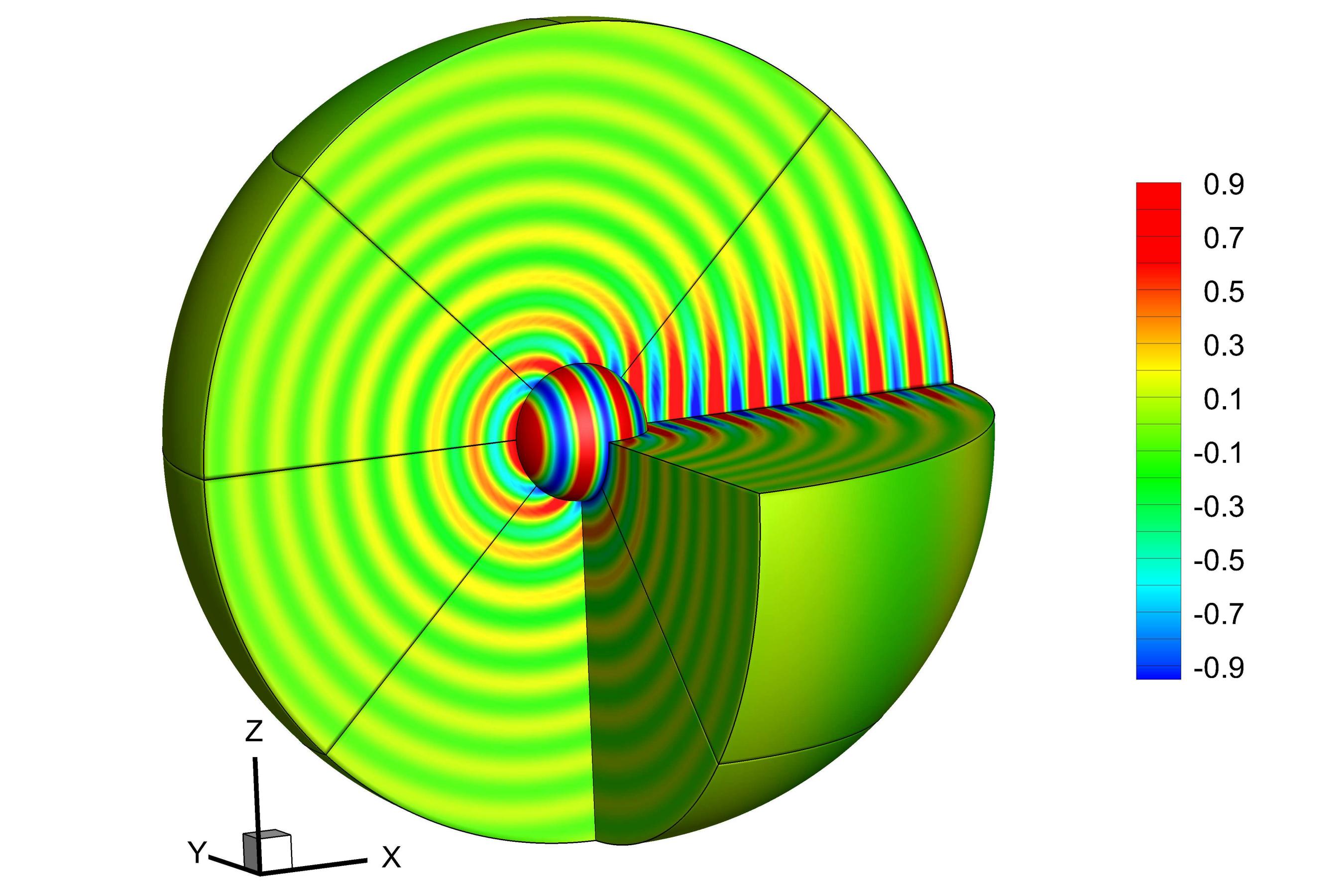}}\quad 
	\subfigure[Imaginary part]{\includegraphics[scale=0.14]{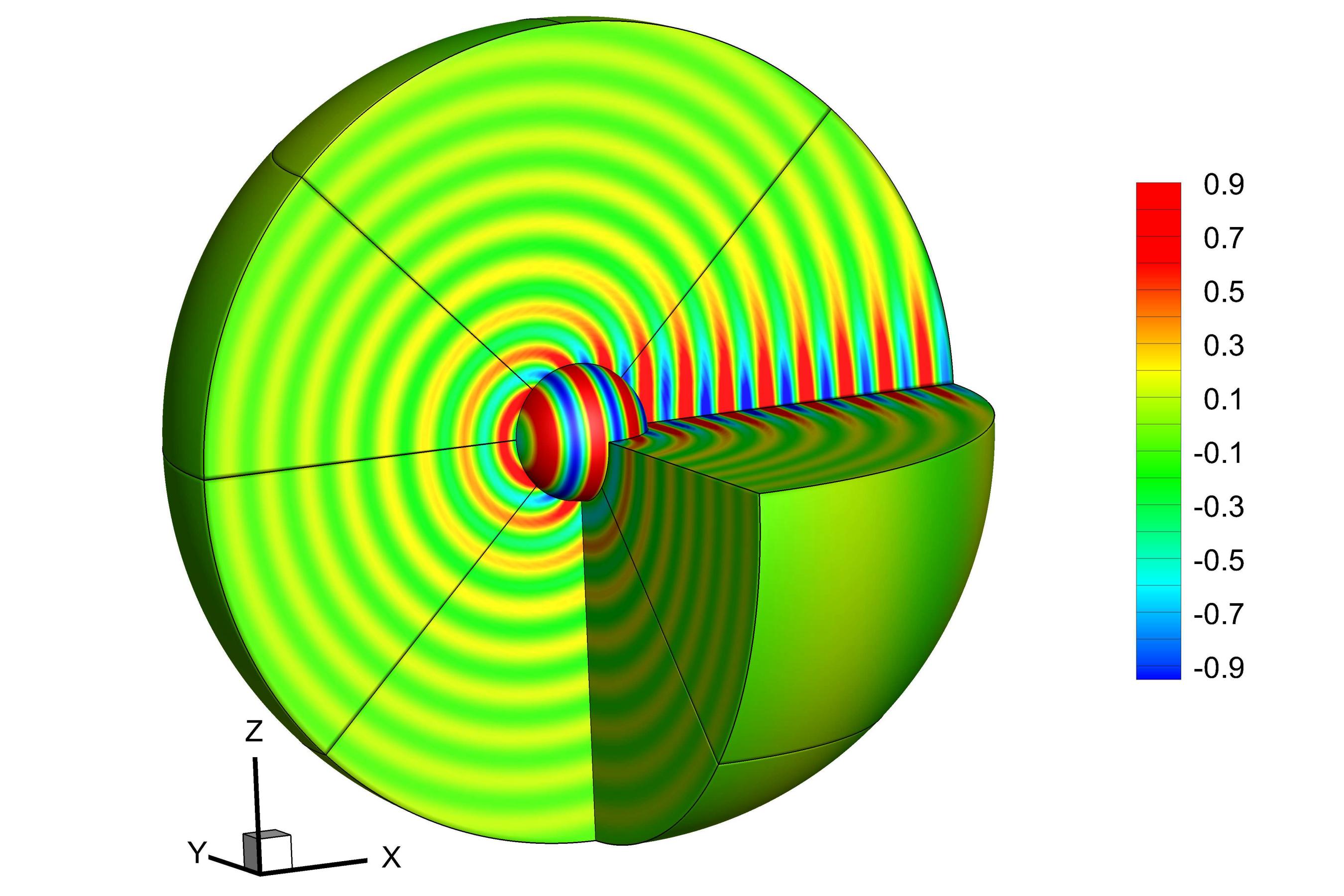}}\quad
	\subfigure[Real part]{\includegraphics[scale=0.14]{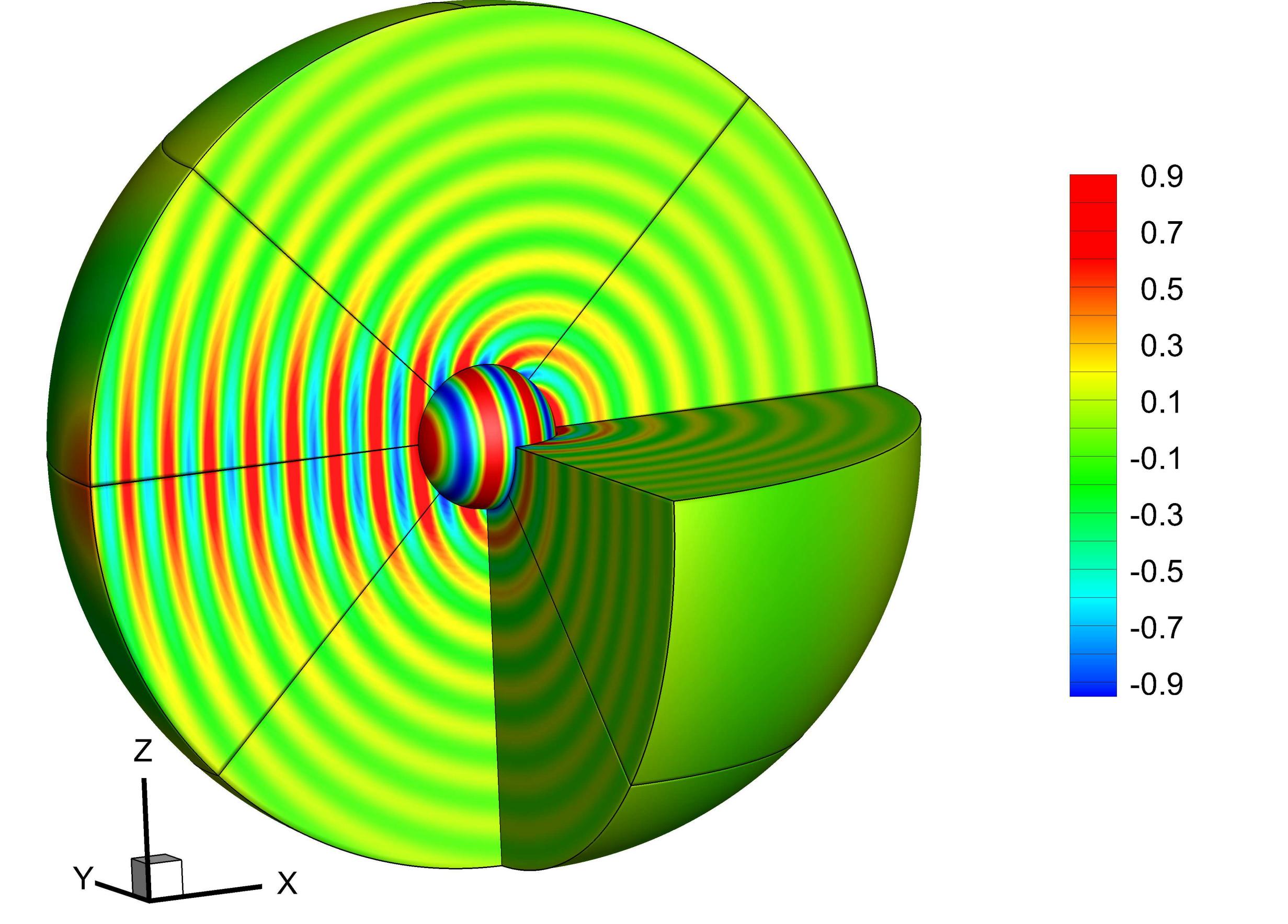}} \quad 
	\subfigure[Imaginary part]{\includegraphics[scale=0.14]{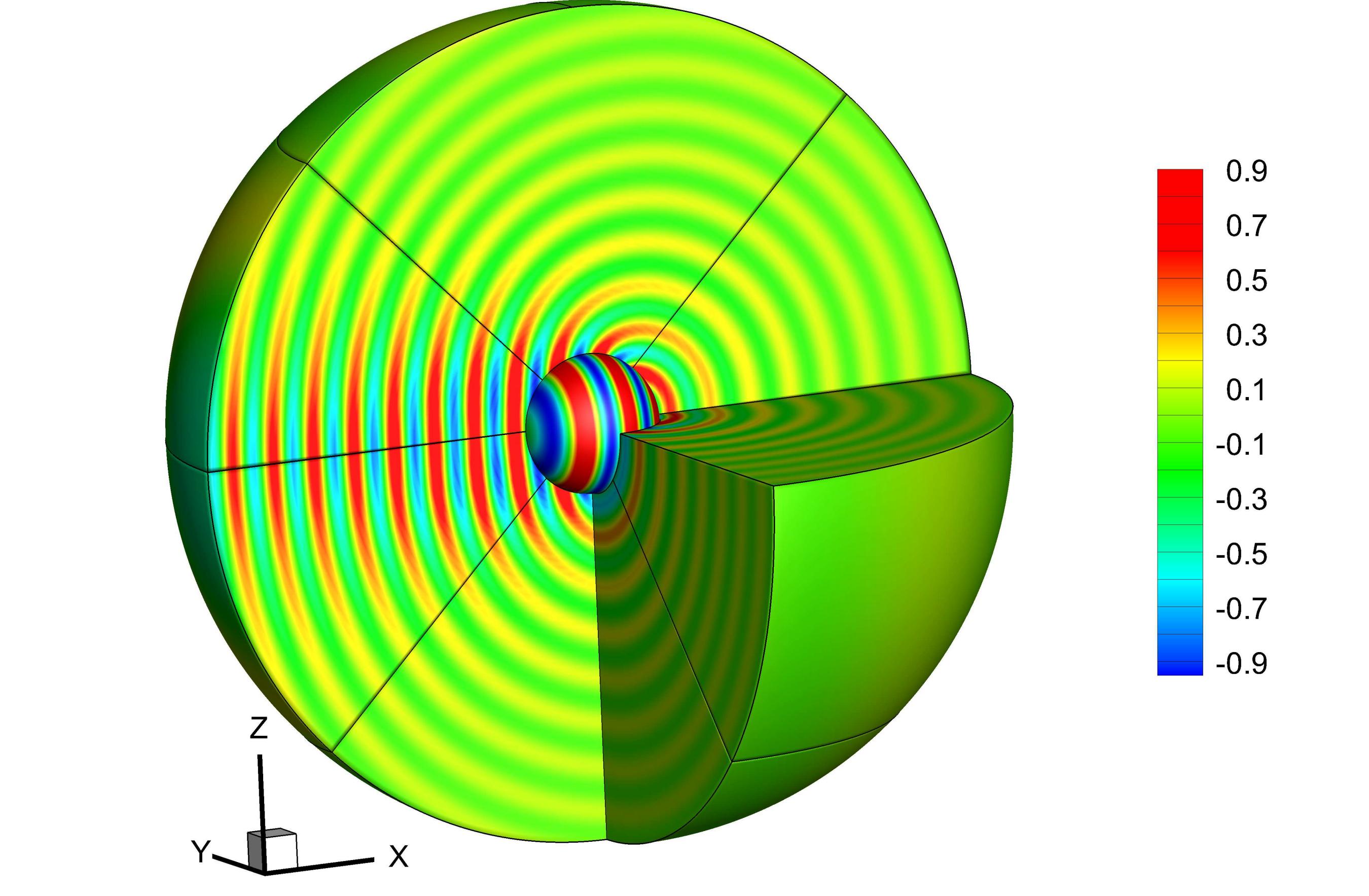}}
	\caption{Acoustic scattering waves from a spherical scatterer with plane incident wave (a)-(b) and spherical incident wave (c)-(d).}
	\label{singlescatter01-01}
\end{center}
\end{figure}

\begin{figure}[!ht]
\begin{center}
	\subfigure[Real part of $E_x$]{\includegraphics[scale=0.18]{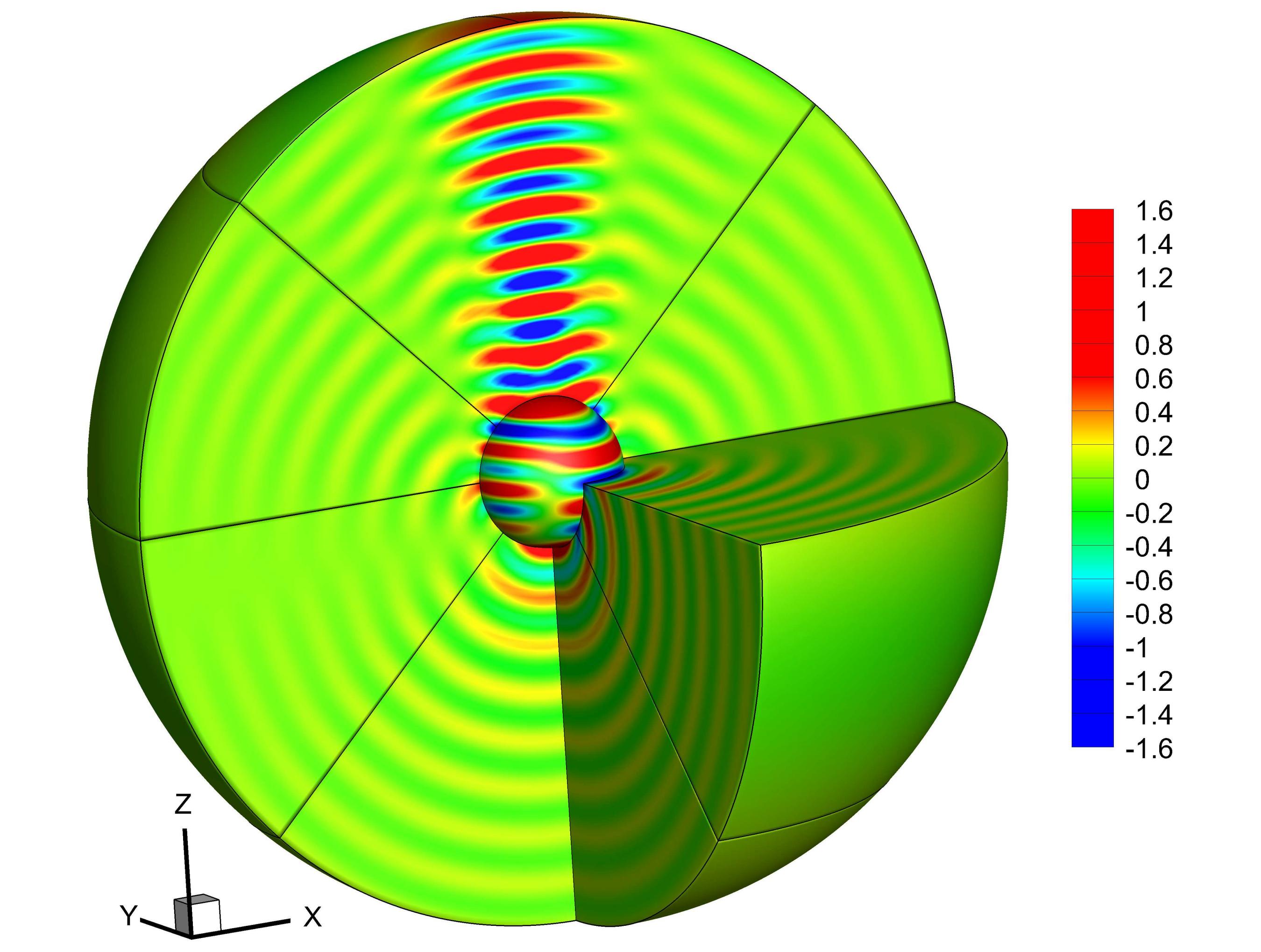}}
	\subfigure[Real part of $E_y$]{\includegraphics[scale=0.18]{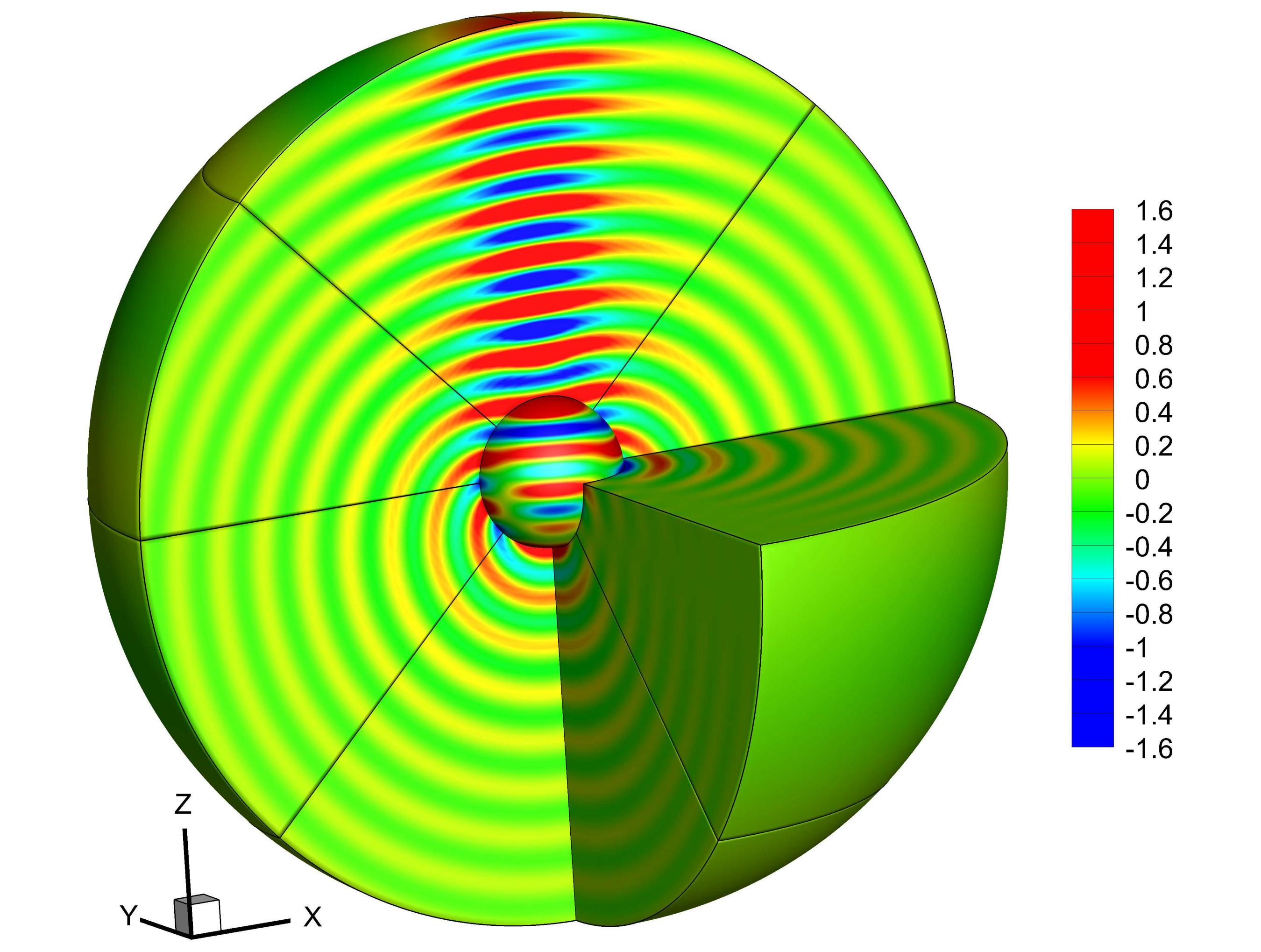}}
	\subfigure[Real part of $E_z$]{\includegraphics[scale=0.18]{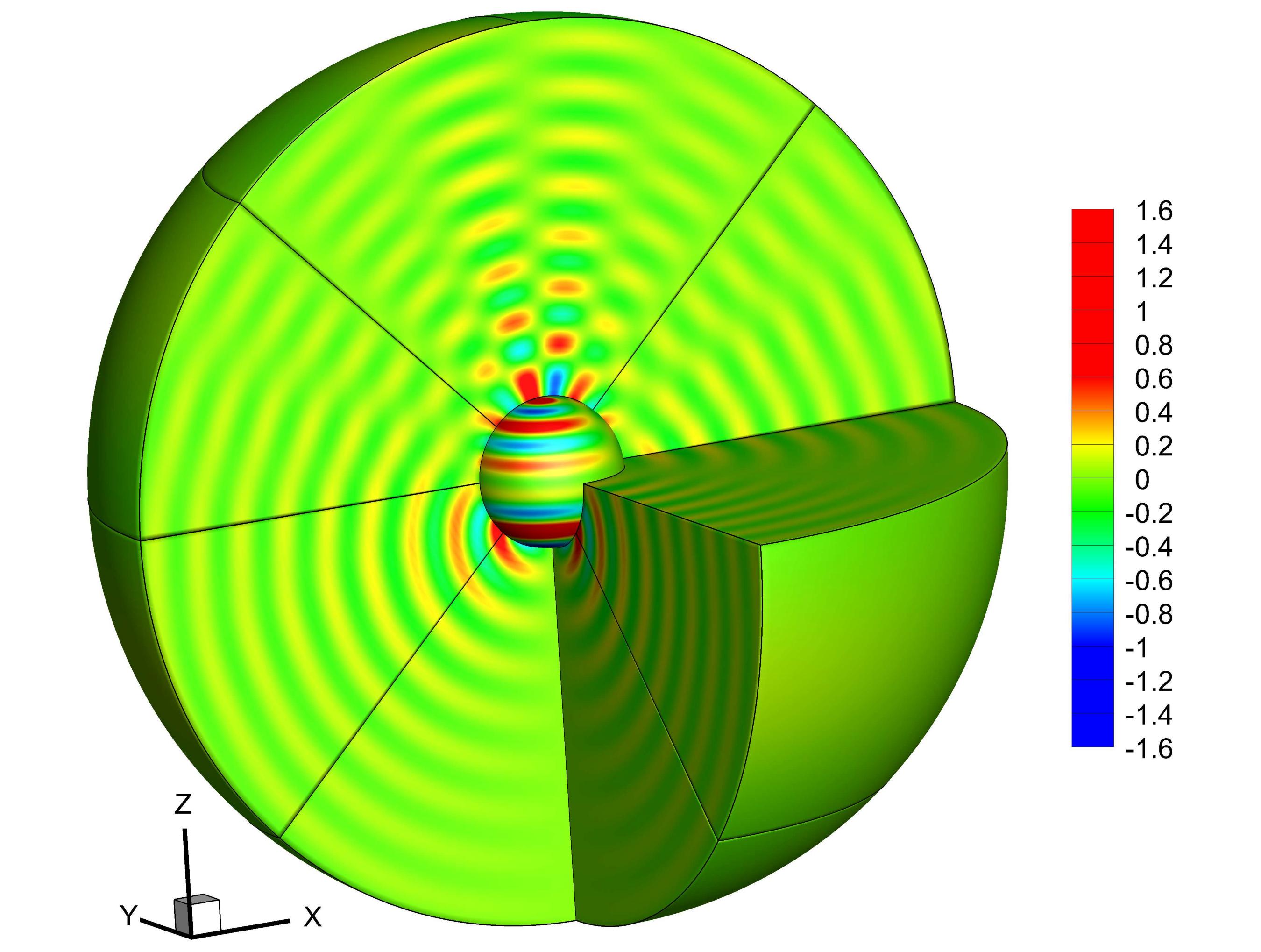}}\\
	\subfigure[Imaginary part of $E_x$]{\includegraphics[scale=0.18]{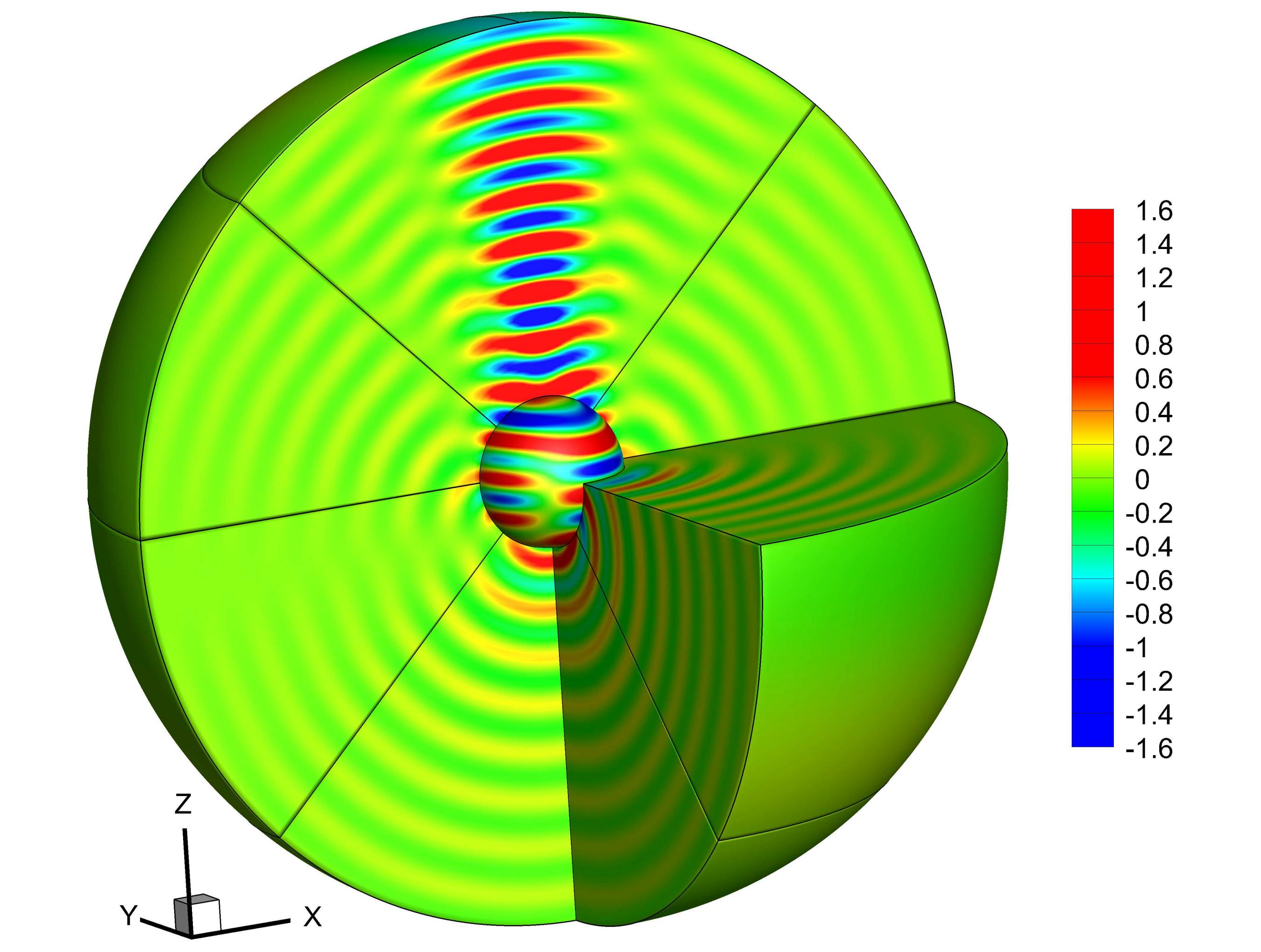}}
	\subfigure[Imaginary part of $E_y$]{\includegraphics[scale=0.18]{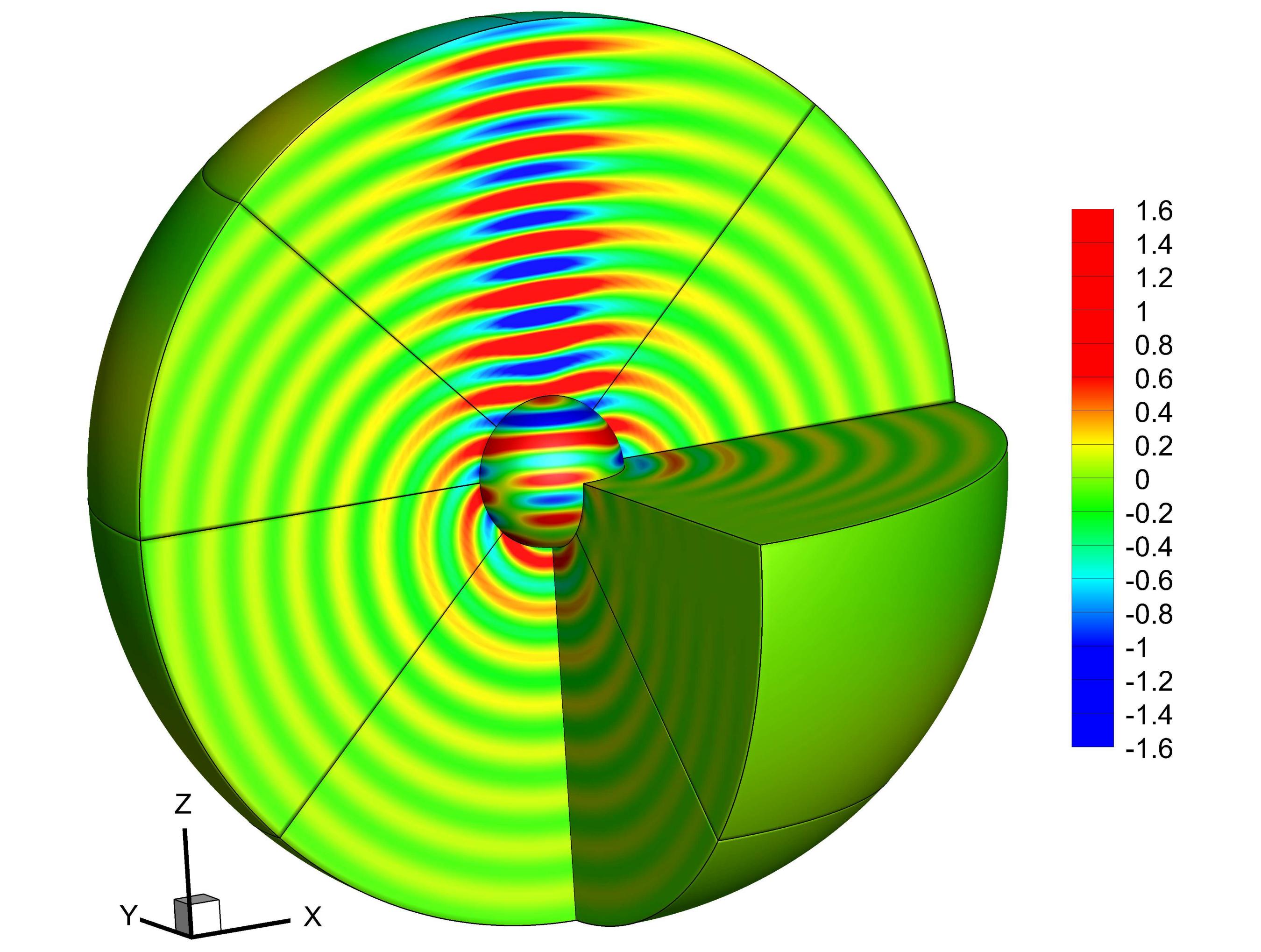}}
	\subfigure[Imaginary part of $E_z$]{\includegraphics[scale=0.18]{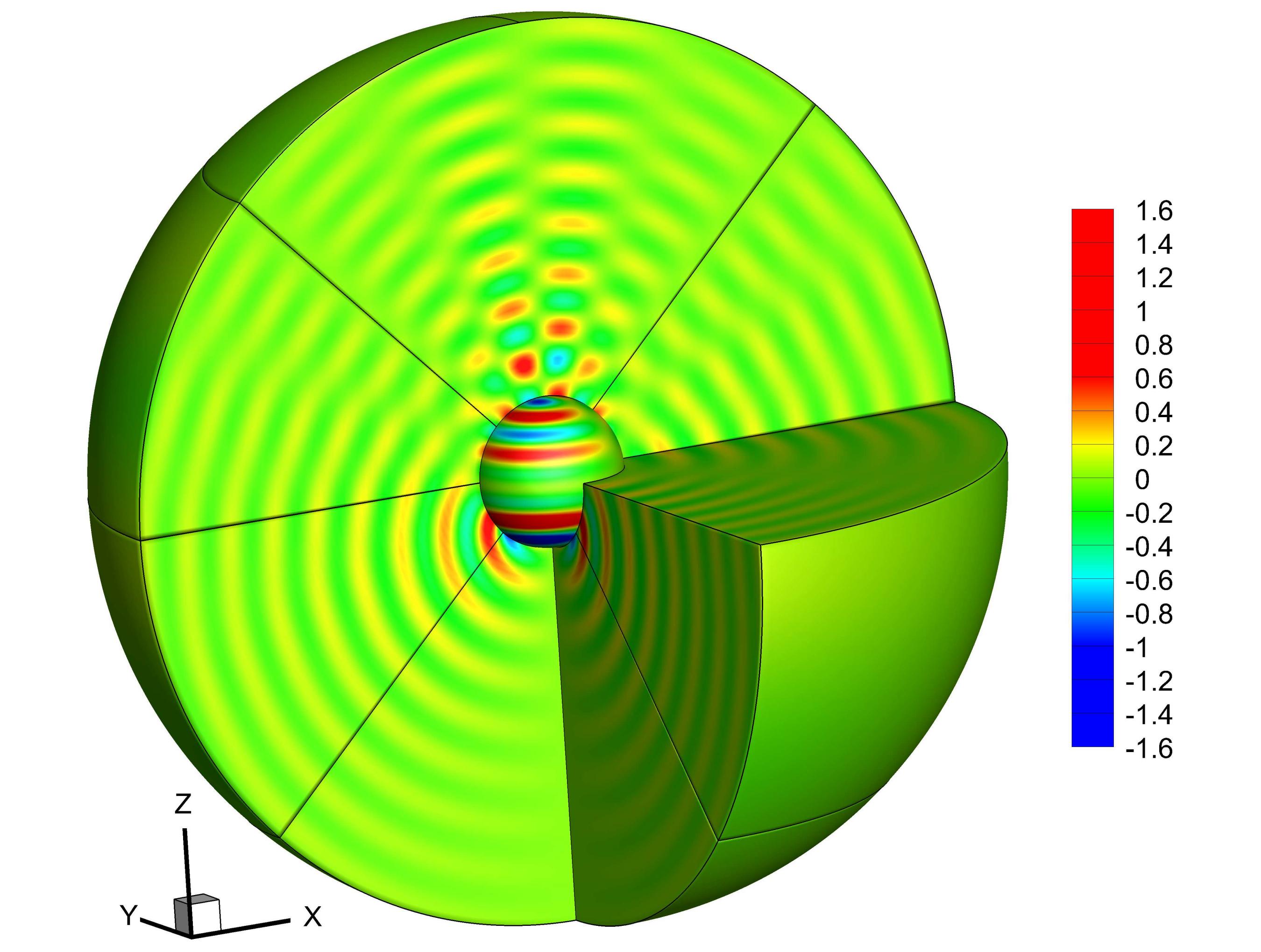}}
	\caption{\small Electromagnetic scattering waves from a spherical scatterer.}
	\label{emspherescatter01-03}
\end{center}
\end{figure}

\subsection{Computing scattering waves by an irregular scatterer}
\label{generalscatterer}

It is seen that with the spectral element approximations (or solutions) on the sphere, we can compute the far-filed scattering waves without loss of accuracy.  Indeed, in many applications (e.g., radar detection and remote sensing), the computation of far-field scattering waves becomes very important.  In what follows, we compute the acoustic scattering wave by an irregular scatterer and note that the approach can be 
extended to Maxwell's equations.  The essential idea is to reduce the unbounded domain by a ball  enclosing the scatterer, and  use the transformed field expansion and spectral-Galerkin method (cf.  \cite{nicholls2006a,fang2007a}) to solve the reduced problem inside the ball. Then we compute the exterior scattering wave by using the aforementioned algorithm.  Again the SPH and VSH expansions become indispensable tools for both the interior and exterior solvers.    


To fix the idea, we consider  
\begin{equation}\label{Helm3D}
\left\{
\begin{split}
&\Delta u+k^2 u=0,\quad {\rm in}\;\; 
\mathbb{R}^3\backslash D, \\
& u|_{\partial D}=g; \quad \frac{\partial u}{\partial r}-{\rm i}ku=O\Big(\frac{1}{r}\Big),\quad {\rm as} \;\; r\rightarrow\infty,
\end{split}\right.
\end{equation}
where  $D$ is a bounded scatterer given by 
\begin{equation}\label{bndscatter}
D=\big\{(r,\theta,\varphi)|0\leq r\leq a+w(\theta,\varphi), 0\leq \theta\leq \pi, 0\leq \varphi<2\pi\big\},
\end{equation}
and $w(\theta,\varphi)=\varepsilon \varrho(\theta, \varphi)$ is  a smooth perturbation of a ball of radius $a.$

%
%
%
%
%
%

We reduce the unbounded domain to a bounded domain by a ball of radius $b$ that  
 encloses the scatterer $D$  (see Figure \ref{singleirregularscatterer} ), and then impose     the exact DtN boundary condition at $r=b.$ 
This leads to an equivalent boundary value problem: 
\begin{subequations}\label{truncatedprob}
	\begin{numcases}{}
	\Delta u+k^2u=0, \quad (r, \theta,\varphi)\in\Omega:=B(b)\backslash D,\label{helmholtz1}\\
	u(a+w(\theta,\varphi), \theta,\varphi)=g(\theta,\varphi), \label{bdrycond1}\\
	\partial_r u+T_b[u]=0,\quad {\rm at}\;\;  r=b,\label{dtn}
	\end{numcases}
\end{subequations}
where
$$T_b[u]:=-\sum\limits_{l=0}^{\infty}k\frac{\partial_zh_l^{(1)}(kb)}{h_l^{(1)}(kb)}\sum\limits_{m=-l}^l\widehat u_{lm}Y_{l}^m(\theta,\varphi),$$
is the DtN operator and $\{\widehat u_{lm}\}$ are the spherical harmonic expansion coefficients of $u$ at $r=b$.

\begin{figure}[!ht]
	\includegraphics[scale=.25]{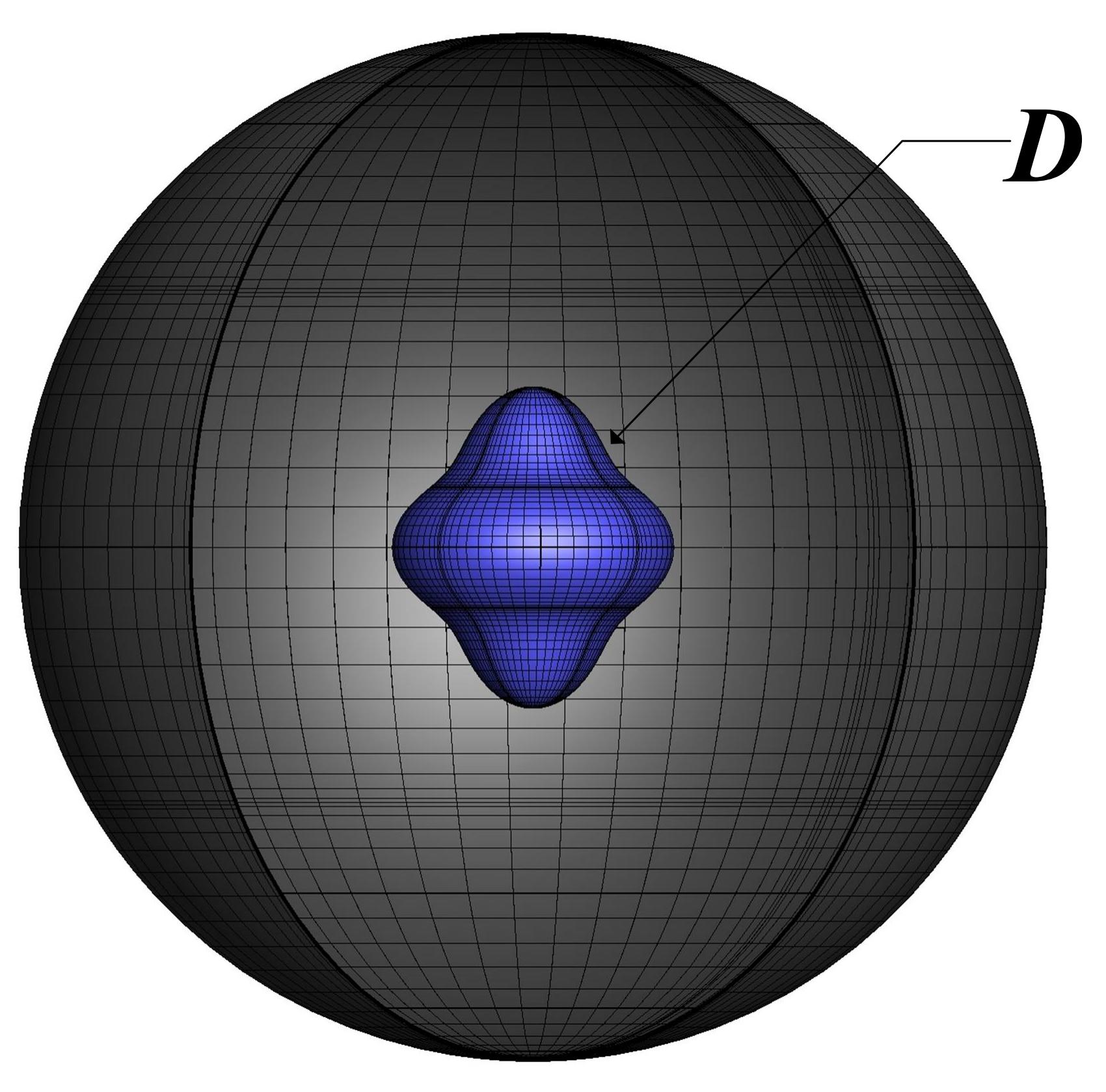}
	\caption{\small An illustration of scatterer $D$ and artificial boundary $\mathbb S_b^2$.} \label{singleirregularscatterer}
\end{figure}

The transformed field expansion (TFE) method (cf.  \cite{nicholls2006a,fang2007a})  has been proven to be efficient in solving  the BVP \eqref{truncatedprob} even with slightly large perturbation.  It relies on an accurate and efficient SPH expansion algorithm for solving the transformed field in the separable spherical annulus $\Omega_{a, b}:=\{(r,\theta, \varphi), a<r<b\}$ transformed from $\Omega$ via the transformation:  
\begin{equation}
\hat r=\frac{(b-a)r-bw(\theta,\varphi)}{(b-a)-w(\theta,\varphi)}=\frac{dr-bw}{d-w},\quad \hat\theta=\theta, \quad \hat\varphi=\varphi,
\end{equation}
where $d=b-a$. Let $\widetilde u(\hat r,\hat\theta,\hat{\varphi})=u\big(\hat r+\frac{(b-\hat r)w(\hat\theta, \hat\varphi)}{d}\big)$ be the transformed field of $u$. Then it satisfies
\begin{subequations}\label{transformedmodel}
	\begin{numcases}{}
	\partial_{\hat r}(\hat r^2\partial_{\hat r}\widetilde u)+\Delta_{\widehat S}\widetilde u+k^2\hat r^2\widetilde u=f(\hat r,\hat\theta,\hat\varphi;\widetilde u,w),\quad (\hat r,\hat \theta,\hat \varphi)\in\Omega_{a,b},\label{transformedhelm}\\
	\widetilde u(a,\hat\theta,\hat\varphi)=g(\hat\theta, \hat\varphi),\quad \partial_{\hat r}\widetilde u(b,\hat\theta,\hat\varphi)+T_b[\widetilde u(b,\hat\theta,\hat\varphi)]=H(\hat\theta,\hat\varphi;\widetilde u,w),\label{transformeddiri}
	\end{numcases}
\end{subequations}
where 
\begin{equation}
\Delta_{\widehat S}\widetilde u:=\frac{1}{\sin\hat\theta}\partial_{\hat\theta}(\sin\hat\theta\partial_{\hat\theta}\widetilde u)+\frac{1}{\sin^2\hat\theta}\partial_{\hat\varphi}^2\widetilde u.
\end{equation}
According to the field expansion theory \cite{fang2007a}, we can approximate the solution of \eqref{transformedmodel} by 
\begin{equation}\label{fieldexp}
\widetilde u_J(\hat r, \hat\theta,\hat{\varphi})=\sum\limits_{j=0}^{J}\widetilde u_j(\hat r, \hat\theta,\hat{\varphi})\varepsilon^j, 
\end{equation}
and the functions $\{\widetilde u_j(\hat r, \hat\theta,\hat{\varphi})\}$ are determined by 
\begin{subequations}\label{TFEmodeeq}
	\begin{numcases}{}
	\partial_{\hat r}(\hat r^2\partial_{\hat r}\widetilde u_j)+\Delta_{\widehat S}\widetilde u_j+k^2\hat r^2\widetilde u_j=f_j(\hat r,\hat\theta,\hat\varphi;\{\widetilde u_{j-s}\}_{s=1}^4),\quad (\hat r,\hat \theta,\hat \varphi)\in\Omega_{a,b},\label{TFEmodeeq1}\\
	\widetilde u_j(a,\hat\theta,\hat\varphi)=\delta_{j,0}g(\hat\theta, \hat\varphi),\;\; \partial_{\hat r}\widetilde u_j(b,\hat\theta,\hat\varphi)+T_b[\widetilde u_j(b,\hat\theta,\hat\varphi)]=H_j(\hat\theta,\hat\varphi;\widetilde u_{j-1}),\label{TFEmodeeq2}
	\end{numcases}
\end{subequations}
where $\delta_{j,0}$ is the Kronecker delta. Noting that $f_j$ only involves $\widetilde u_{j-s}(s=1, 2, 3, 4)$ and $H_j$ only depends on $\widetilde u_{j-1}$, each $\widetilde u_j$ can be solved from \eqref{TFEmodeeq}  in a spherical shell $\Omega_{a,b}$.
Thanks to the special geometry of the computational domain $\Omega_{a,b}$, an efficient spectral-Galerkin method with SPH expansion can be used. Denote the SPH expansions by
\begin{subequations}
\begin{align}
&\{\widetilde u_j(\hat r,\hat\theta,\hat\varphi), f_j(\hat r,\hat\theta,\hat\varphi)\}=\sum\limits_{l=0}^{\infty}\sum\limits_{m=-l}^l\{\widehat U_{lm}^j(\hat r), \widehat F_{lm}^j(\hat r)\}Y_l^m(\hat\theta,\hat\varphi),\\
&\{g(\hat\theta,\hat\varphi), H_j(\hat\theta,\hat\varphi)\}=\sum\limits_{l=0}^{\infty}\sum\limits_{m=-l}^l\{\widehat G_{lm}, \widehat H_{lm}^j\}Y_l^m(\hat\theta,\hat\varphi). 
\end{align}
\end{subequations}
Then, \eqref{TFEmodeeq} can be decomposed into a sequence of one-dimensional, two-point boundary value problems:
\begin{subequations}\label{dimreducedprob}
\begin{numcases}{}
\frac{d}{d\hat r}\Big(\hat r^2\frac{d\widehat U_{lm}^j}{d\hat r}\Big)+(\hat r^2k^2-l(l+1))\widehat U_{lm}^j=\widehat F_{lm}^j,\quad \hat r\in (a, b),\\
\widehat U_{lm}^j(a)=\delta_{j,0}\widehat G_{lm},\quad \frac{d\widehat U_{lm}^j(b)}{d\hat r}- k\frac{\partial_zh_l^{(1)}(kb)}{h_l^{(1)}(kb)}\widehat U^j_{lm}(b)=\widehat H_{lm}^j,
\end{numcases}
\end{subequations}
for integers $|m|\leq  l$, $ l\geq 0$ and $j=0, 1, \cdots, J$. We use the proposed SPH expansion algorithm to compute the expansion coefficients $\{\widehat F_{lm}^j, \widehat G_{lm}, \widehat H^j_{lm}\}$ and then employ  spectral solver to solve \eqref{dimreducedprob}. In the calculation of SPH expansion coefficients, we first evaluate the values of given functions on spectral element grids and then apply our SPH expansion algorithm. Further, we can use the solver introduced in the last subsection to compute the far field of the scattering wave after having the approximation of $\widetilde u_J$ on the artificial boundary $r=b$.  

We remark that the VSH expansion algorithm together with the TFE method can also be applied to solve the electromagnetic scattering problem in a  similar setting. We refer to \cite{ma2017} for the detailed descripition of the TFE method along this line. 
%

Below, we present some numerical results. Assume that $\partial D$ has the parameterization: 
\begin{equation}
r=a+\varepsilon \varrho(\theta, \varphi),
\end{equation}
where
\begin{equation}
a=0.2, \quad \varepsilon=0.05, \quad \varrho(\theta,\varphi)=\frac{1}{8}(35\cos^4\theta-30\cos^2\theta+3).
\end{equation}
Take the incident wave $g$ to be the plane wave $e^{\ri k\hat{\bs k}\cdot\bs x}$, and the spherical wave $e^{\ri k|\bs x-\bs x_0|}$ with $\hat{\bs k}=(1,0,0)$, $k=30$, $\bs x_0=(1, 0, 1)$. The artificial boundary centered at origin with radius $b=0.4$ is adopted. It is set to be very close to the scatterer $D$ for small truncated domain and thus less computation time. On the other hand, the scattering field outside $B(b)$ is computed from the spectral element approximation on $\partial B(b)$. For the spectral element approximation on spherical surfaces, we always use a uniform $3\times 4$ mesh in the $\theta$-$\varphi$ plane and set $N=40$. Actually, this spectral element approximation setting is good enough to make the spectral element interpolation to machine accuracy. In the spectral solver of one dimensional problems \eqref{dimreducedprob}, we use polynomial of degree $p=30$. The field expansion \eqref{fieldexp} is truncated at $J=10$. The computed values $\widehat{U}_{N,l}$ for $35\leq l\leq 40$ are listed in Table \ref{tabcubic5}. These results show that it is enough to set the truncation mode $L=40$ for $k=30$. 
We depict the scattering waves in Figure \ref{singlescatter02-01} and Figure \ref{singlescatter02-02}, where the scattering waves outside $B(b)$ in the right columns are computed by the numerical method introduced in the Subsection \ref{singlesphscat}.
\begin{figure}[!ht]
	\begin{center}
		\subfigure[Real part inside $B(b)$ ]{\includegraphics[scale=0.40]{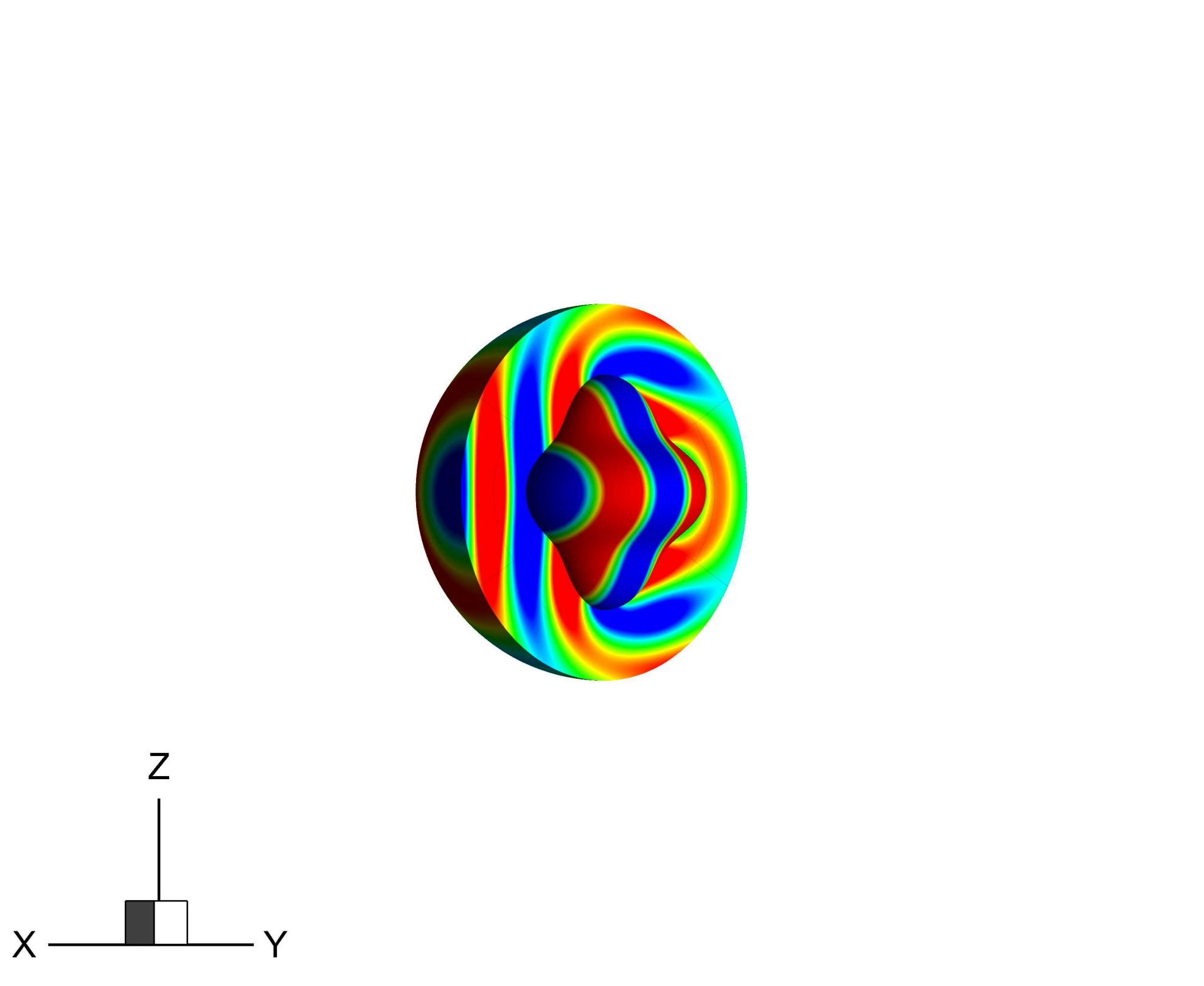}} 
		\subfigure[Real part inside and outside $B(b)$ ]{\includegraphics[scale=0.35]{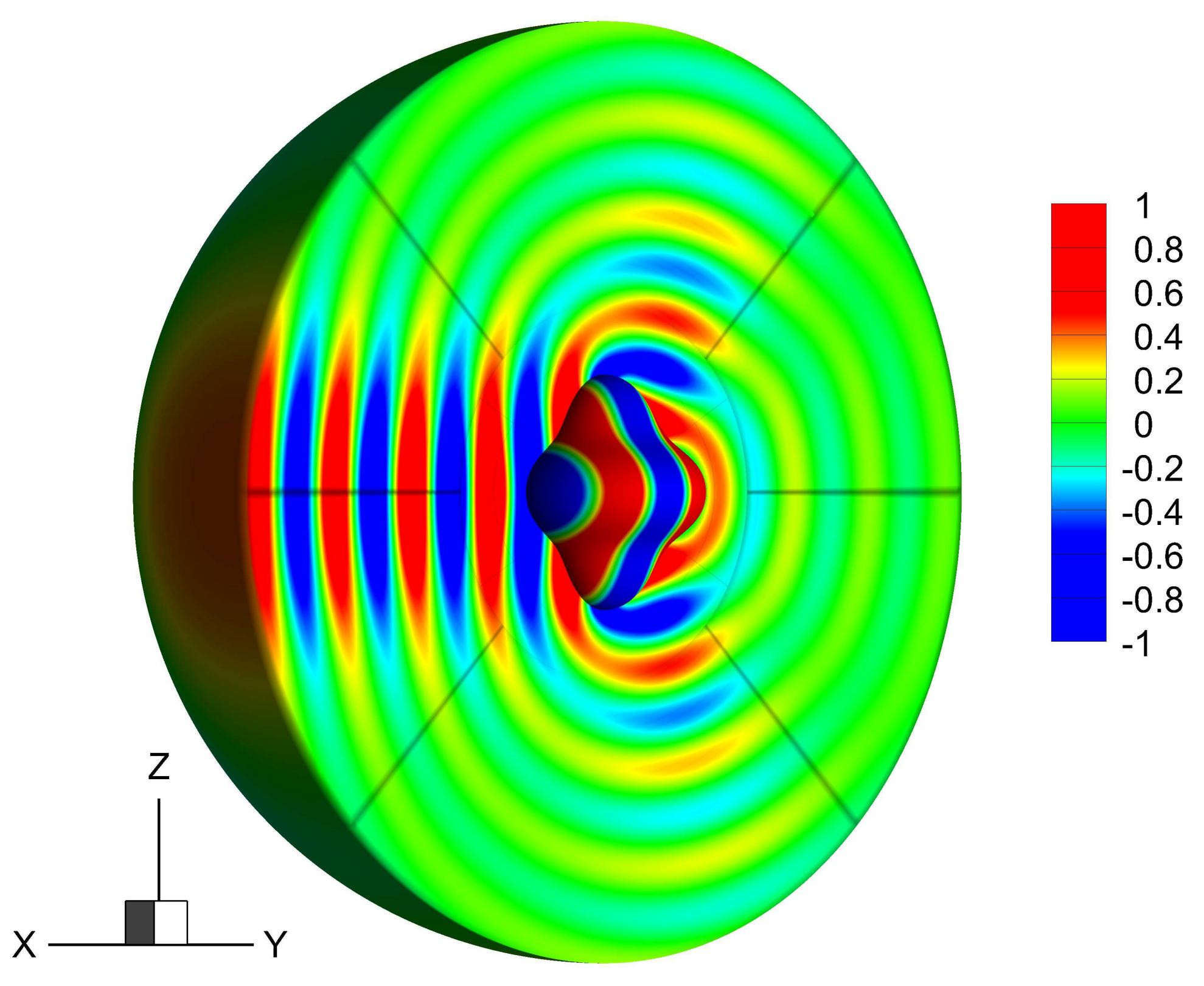}}\\
		\subfigure[Imaginary part inside $B(b)$]{\includegraphics[scale=0.40]{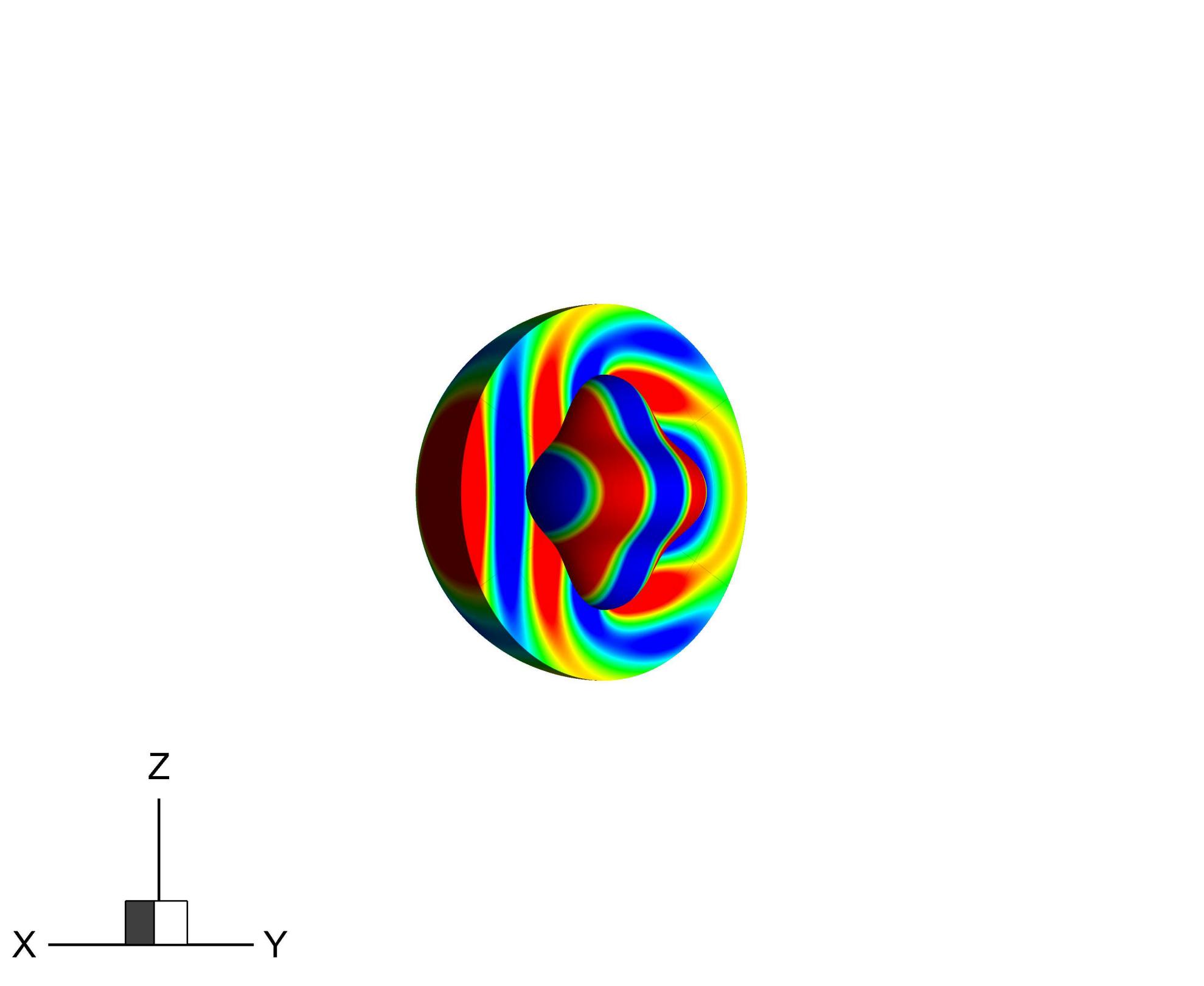}} 
		\subfigure[Imaginary part inside and outside $B(b)$]{\includegraphics[scale=0.35]{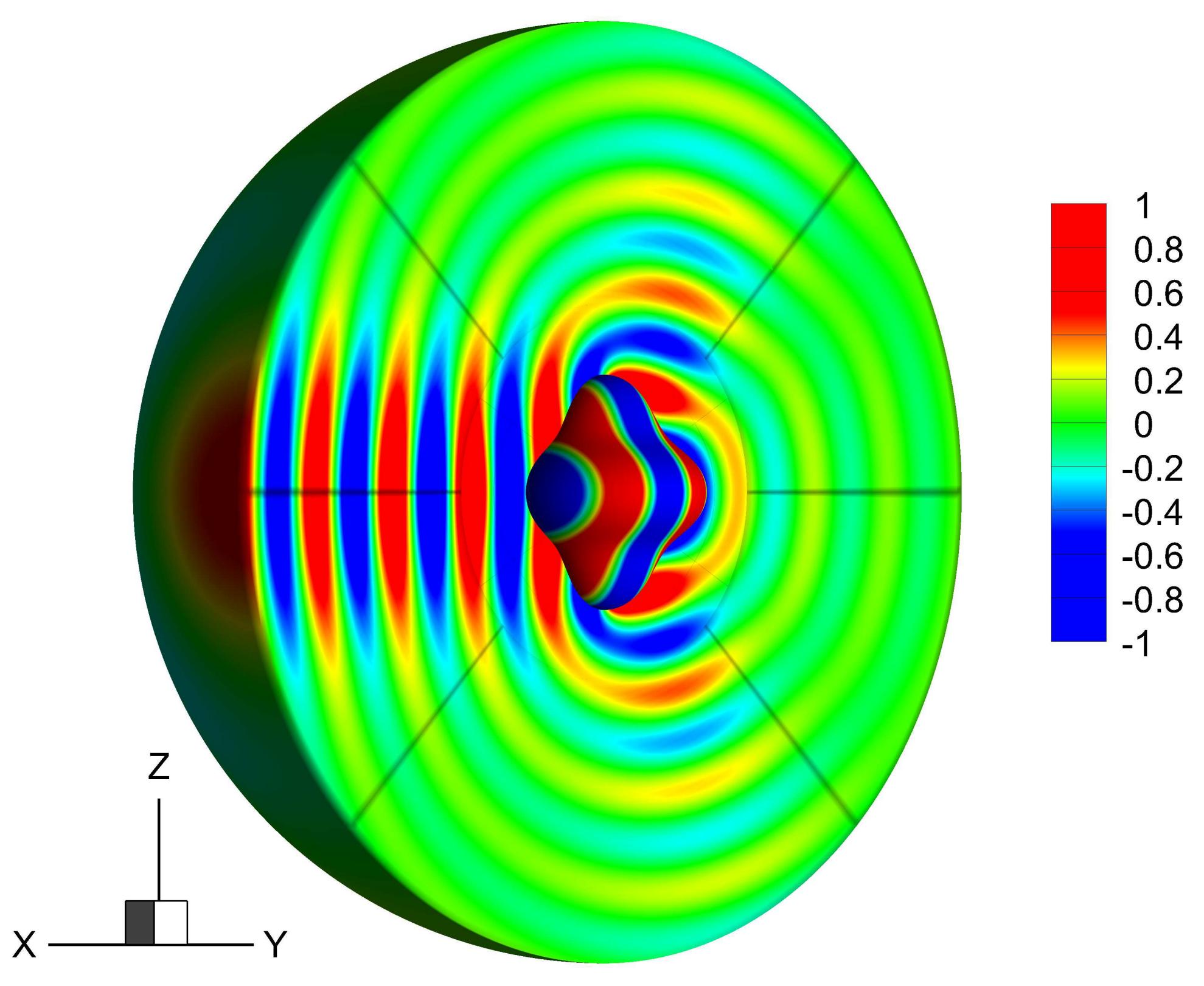}}
		\caption{Scattering wave of plane incident wave.}
		\label{singlescatter02-01}
	\end{center}
\end{figure}
\begin{figure}[!ht]
	\begin{center}
		\subfigure[Real part inside $B(b)$ ]{\includegraphics[scale=0.40]{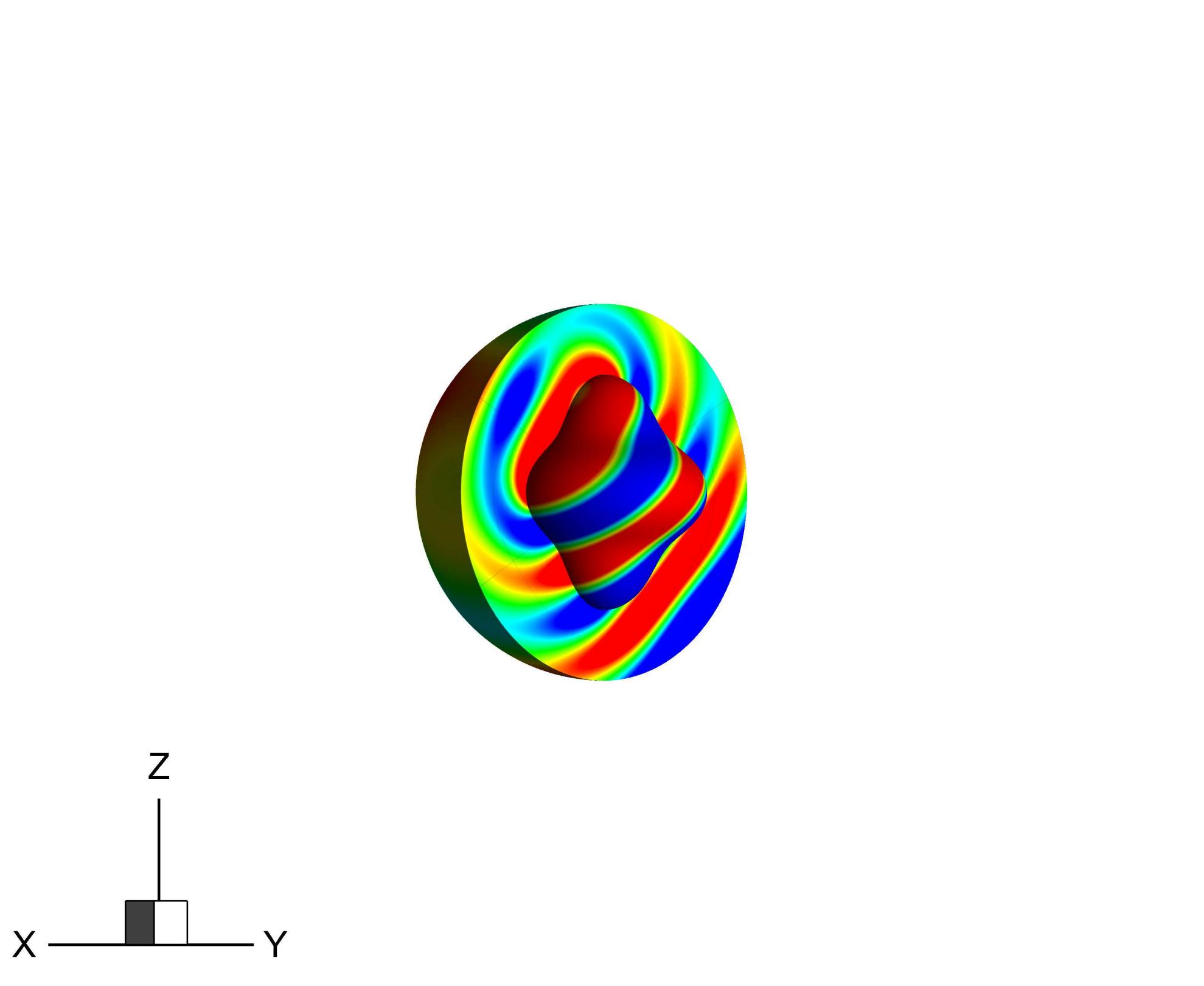}} 
		\subfigure[Real part inside and outside $B(b)$ ]{\includegraphics[scale=0.35]{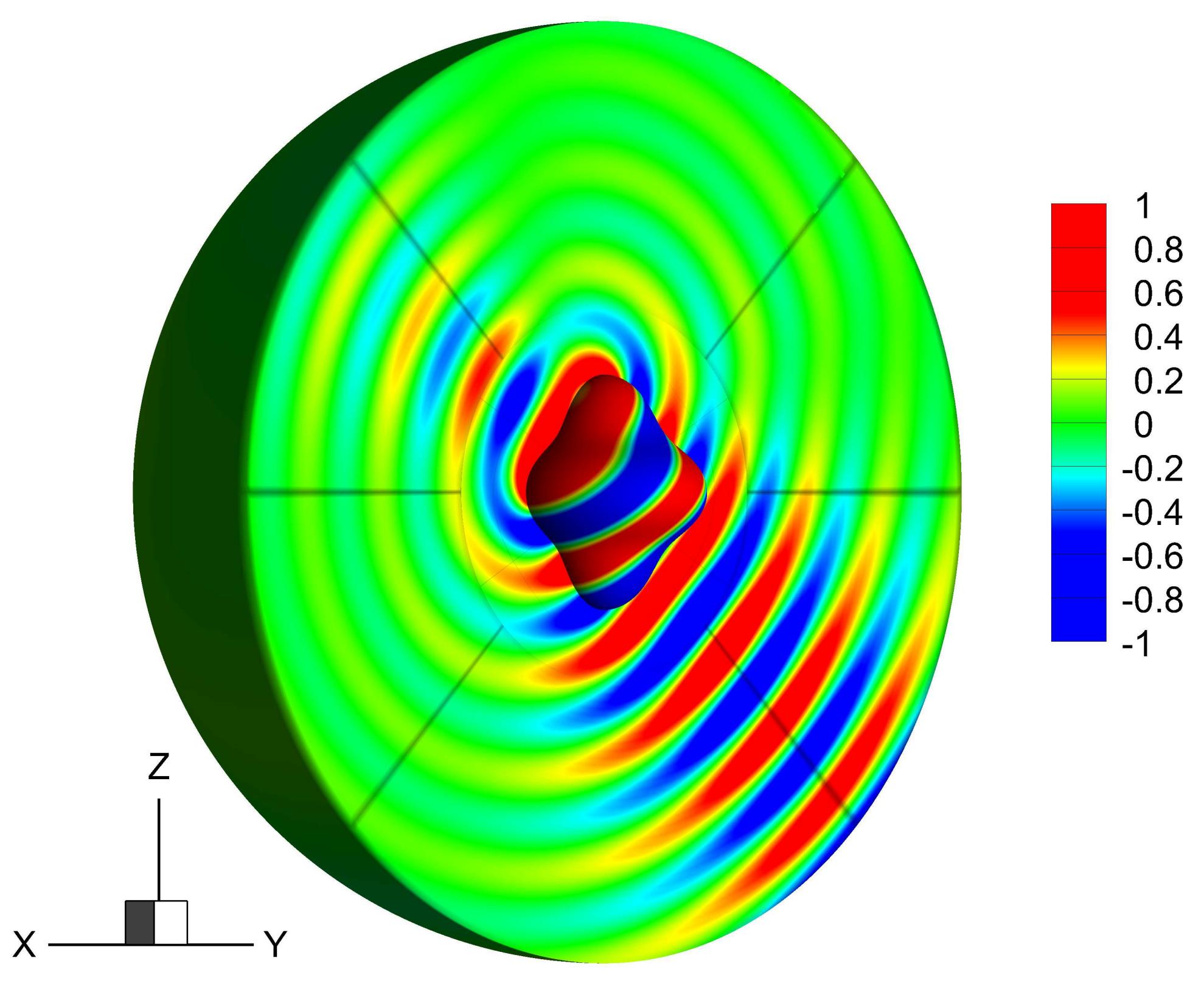}}\\
		\subfigure[Imaginary part inside $B(b)$]{\includegraphics[scale=0.40]{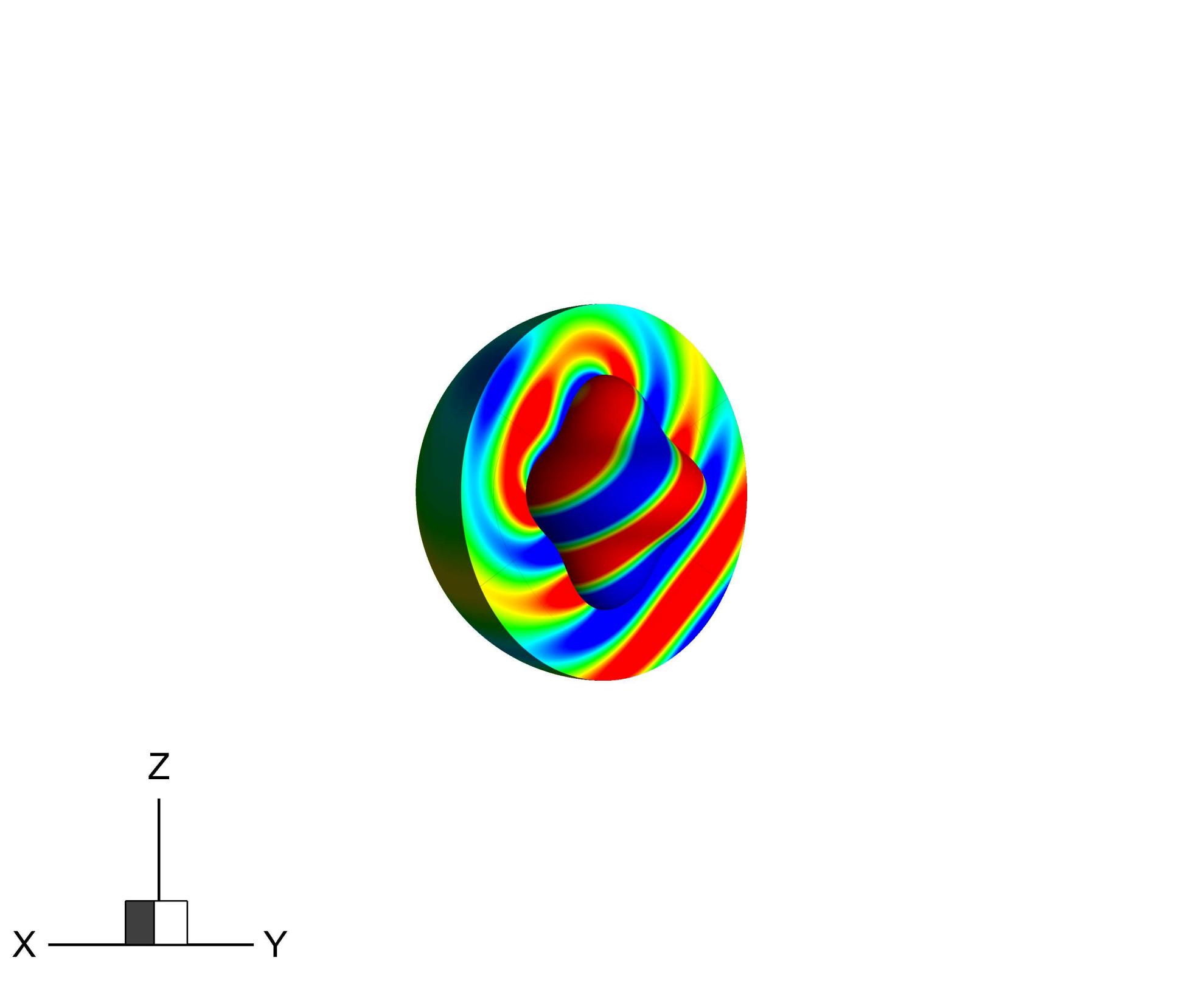}} 
		\subfigure[Imaginary part inside and outside $B(b)$]{\includegraphics[scale=0.35]{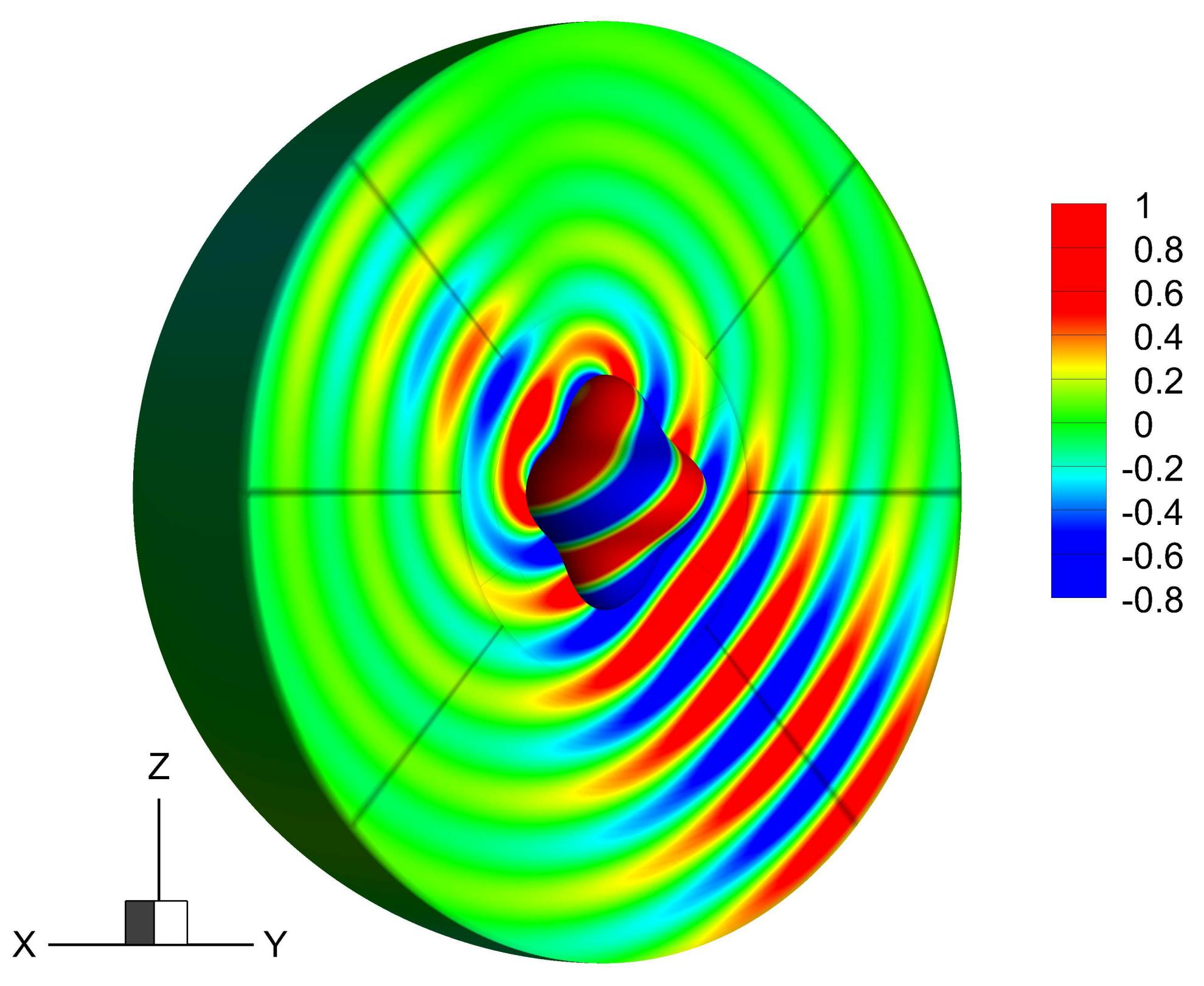}}
		\caption{Scattering wave of spherical incident wave.}
		\label{singlescatter02-02}
	\end{center}
\end{figure}
\begin{table}[h]
	\caption{Magnitude of spherical harmonic coefficients (single irregular scatterer).}
	\begin{center}
		\begin{tabular}{|c|c|c|}
			\hline
            \multirow{2}{*}{$l$} & \multicolumn{2}{|c|}{$\widehat{U}_{N,l}$}\\
			\cline{2-3}
			& plane incident wave   & spherical incident wave \\
			\hline
			35 &  5.5527e-12  &  3.8383e-12      \\
			\hline
			36 &  2.5861e-12  &  1.5264e-12    \\
			\hline
			37 &  9.9869e-13  &  7.6434e-13       \\
			\hline
			38 &  4.5069e-13  &  2.9476e-13     \\
			\hline
			39 &  1.6978e-13  &  1.3971e-13      \\
			\hline
			40 &  7.3989e-14  &  5.3222e-14    \\
			\hline		
		\end{tabular}
	\end{center}
	\label{tabcubic5}
\end{table}

\subsection{Computing scattering waves by multiple spherical scatterers}
Consider the multiple scattering problem
\begin{subequations}\label{2dmultiscatteringprob}
	\begin{numcases}{}
	\Delta u+k^2u=0, \quad\quad\quad\hbox{in}\;\; \mathbb{R}^3\backslash \bar D,\\
	u=g_N,\qquad\qquad\quad\;\;\;   \hbox{on}\;\;\partial D,\\
	\frac{\partial u}{\partial r}-{\rm i}ku=O\Big(\frac{1}{r}\Big),\quad r\rightarrow\infty,
	\end{numcases}
\end{subequations}
where $k$ is the wave number and $g_N$ is induced by the incident wave.  Without loss of generality,  we assume that the scatterer $D=D_1\cup D_2\cup\cdots D_M$ consists of $M$ well-separated spherical scatterers centered at  points $O_1,O_2,\cdots, O_M$ of radii $a_1, a_2,\cdots,a_M$, respectively, see Figure \ref{twoscater3d} for an example of two scatterers.  As with the previous case, $g_N$ is the high-order spectral element approximation of a given incident wave.

\begin{figure}[!ht]
	\includegraphics[width=0.6\textwidth]{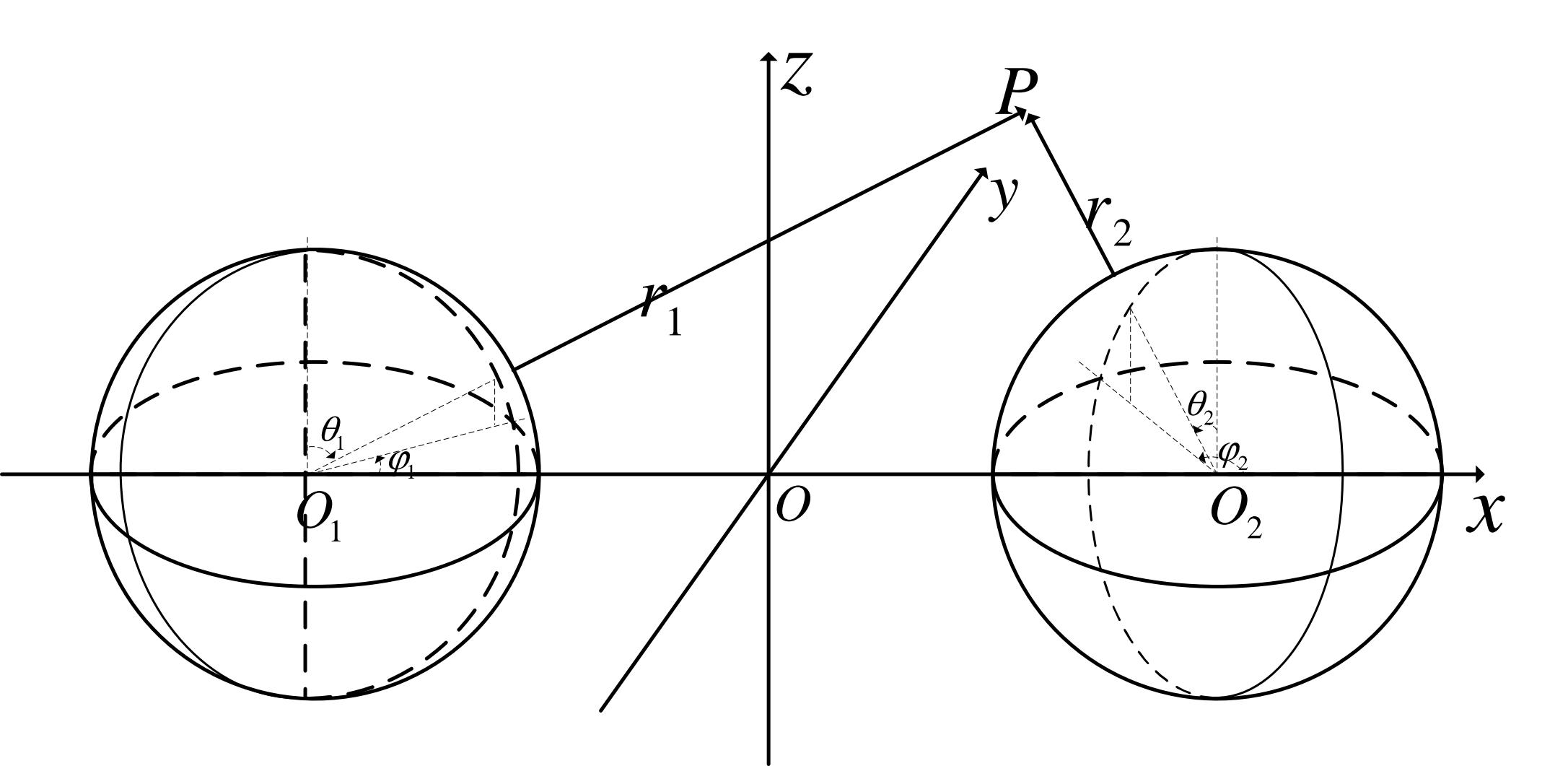}
	\caption{\small Geometric configuration for two spherical scatters.} \label{twoscater3d}
\end{figure}

It is known that the scattering field $u$ has a unique purely outgoing wave decomposition $u=u_1+u_2+\cdots+u_M,$ where $u_1, u_2,\cdots, u_M$ are purely outgoing waves satisfying the following scattering problems (cf. \cite{grote2004dirichlet,jiang2012adaptive}): 
\begin{subequations}\label{2dpurelyoutgoingwave}
	\begin{numcases}{}
	\Delta u_i+k^2u_i=0, \qquad\quad  \hbox{in}\;\; \mathbb{R}^2\backslash \bar D_i,\\
	u_i=g_N-\sum\limits_{j=1,j\neq i}^Mu_j, 
	\quad  \hbox{on}\;\; \partial D_i,\label{outgoingbc}\\
	\frac{\partial u_i}{\partial r}-{\rm i}ku_i=O\Big(\frac{1}{r}\Big),\quad r\rightarrow\infty.
	\end{numcases}
\end{subequations} 
The position vectors of any point $P$ with respect to the local spherical coordinate systems are given by
$${\bs r}_j=(r_j\sin\theta_j\cos\varphi_j,r_j\sin\theta_j\sin\varphi_j, r_j\cos\theta_j),\quad j=1,2,\cdots, M.$$
Let $\bs b_{ij}$ be the position vector of $O_j$ with respect to $O_i$, so that $\bs r_i=\bs r_j+\bs b_{ij}$. We further  introduce  the  unit vectors
$$\hat{\bs b}_{ij}=\frac{\bs b_{ij}}{|\bs b_{ij}|}=\frac{\bs b_{ij}}{b_{ij}},\quad \hat{\bs r}_j=\frac{\bs r_j}{|\bs r_j|}=\frac{\bs r_j}{r_j}, \;\;\;  i,j=1, 2,\cdots, M.$$
Recall  the following useful formulas (cf. \cite{danos1965multipole, martin2006multiple}).
\begin{lemma}
	\label{thmhlynm}
	For $m=0,\pm 1, \cdots, \pm l$ and $l=0, 1, \cdots, $ there hold
	\begin{equation}\label{2formuA}
	\begin{split}
	& h^{(1)}_l(kr_i)Y_l^m(\hat{\bs r}_i)=\sum\limits_{n=0}^{\infty}\sum\limits_{s=-n}^{n}{S}_{nl}^{sm}(\bs b_{ij})j_{n}(kr_j)Y_{n}^{s}(\hat{\bs r}_j), \quad r_j<b_{ij},\\
	& h^{(1)}_l(kr_i)Y_l^m(\hat{\bs r}_i)=\sum\limits_{n=0}^{\infty}\sum\limits_{s=-n}^{n}\widehat{S}_{nl}^{sm}(\bs b_{ij})h^{(1)}_{n}(kr_j)Y_{n}^{s}(\hat{\bs r}_i), \quad r_j>b_{ij},
	\end{split}
	\end{equation}
	where
	\begin{equation}
	\label{separationmat}
	\begin{split}
	{S}_{nl}^{sm}(\bs b_{ij})&=4\pi\ri^{n-l}\sum\limits_{q=0}^{\infty}\ri^qh^{(1)}_q(k b_{ij})\overline{Y_q^{s-m}}(\hat{\bs b}_{ij})\mathcal{G}(l,m;q,s-m;n),\\
	\widehat{S}_{nl}^{sm}(\bs b_{ji})&=4\pi\ri^{n-l}\sum\limits_{q=0}^{\infty}\ri^q(-1)^mj_q(k b_{ij})\overline{Y_q^{s-m}}(\hat{\bs b}_{ij})\mathcal{G}(l,m;n,-s;q),
	\end{split}
	\end{equation}
	and $\mathcal{G}(l,m;n,-s;q)$ is the Gaunt coefficient.
\end{lemma}
Like \eqref{approxsolution},  \eqref{2dpurelyoutgoingwave} admits the solution
\begin{equation*}
u_i(\bs r_i)=\sum\limits_{l=0}^{\infty}\sum\limits_{|m|=0}^lA_{lm}^i\psi_l^m(\bs r_i),\quad
i=1, 2, \cdots,M,
\end{equation*}
where 
$$\psi_l^m(\bs r_i):=\frac{h_l^{(1)}(kr_i)}{h_l^{(1)}(ka_i)}Y_l^m(\theta_i,\varphi_i),\quad i=1, 2, \cdots, M.$$
We truncate the series and obtain numerical approximations:
\begin{equation}\label{U1U2N}
U^i_L(\bs r_i)=\sum\limits_{l=0}^{L}\sum\limits_{|m|=0}^lA_{lm}^i\psi_l^m(\bs r_i),\quad
i=1, 2, \cdots,M.
\end{equation}
Then, the problem is to determine the coefficients $\bs A:=[\bs A^1, \bs A^2, \cdots, \bs A^M]$ where $\bs A^i=(A_{lm}^i) $.

Similar to the single scatterer case, these coefficients can be determined by matching the data on the spherical surfaces $\partial D_1,\partial D_2,\cdots,\partial D_M$  using the boundary conditions \eqref{outgoingbc}. For this purpose, we first use the developed  algorithm  to compute the spherical harmonic expansions of the boundary data $g_N$ on the spherical scatterers $\partial D_1,\partial D_2,\cdots,\partial D_M$:
\begin{equation*}
g_N|_{\partial D_j}\approx \sum\limits_{l=0}^L\sum\limits_{|m|=0}^l\widehat{g}_{lm}^{j} Y_l^m(\theta_j,\varphi_j).
\end{equation*}
According to Lemma  \ref{thmhlynm},  we have the  approximations
\begin{equation*}
\psi_l^m(\bs r_i)|_{\partial D_j}\approx\sum\limits_{n=0}^L\sum\limits_{|s|=0}^{n} \Psi_{n l}^{s m}(\bs b_{ij})Y_{n}^{s}(\theta_j,\varphi_j),
\end{equation*}
where
\begin{equation}
\label{sphcoefwavefunc}
\Psi_{n l}^{s m}(\bs b_{ij})=\frac{{S}_{nl}^{sm}(\bs b_{ij})j_{n}(ka_j)}{h^{(1)}_l(ka_i)}.
\end{equation}
Note that $\Psi_{n l}^{s m}(\bs b_{ij})$ also depends on $a_i, a_j$ and $k$, but  
we do not specify the dependence  for simplicity. Then the numerical approximations \eqref{U1U2N} on specific scatterer $\bs D_j$ can be further approximated by
\begin{equation}
U^i_L|_{\partial D_j}=\sum\limits_{n=0}^L\sum\limits_{|s|=0}^{n}\Big(\sum\limits_{l=0}^{L}\sum\limits_{|m|=0}^lA_{lm}^i\Psi_{n l}^{s m}(\bs b_{ij})\Big)Y_{n}^{s}(\widehat{\bs r}_j),\quad
j=1, 2, \cdots,i-1, i+1, \cdots, M.
\end{equation}
Although we have analytic formulas \eqref{separationmat} and \eqref{sphcoefwavefunc} for the matrices $\mathbb{S}^{ij}:=(\Psi_{n l}^{s m}(\bs b_{ij}))$, they are too complicated to be directly used in the practical computation. For example, consider the analytic formula for $\Psi_{n l}^{s m}(\bs b_{ij}),$ which has terms of the form
\begin{equation}
\label{wavefunsphcoefanalyform}
\frac{4\pi\ri^{n-l+q}h^{(1)}_q(kb_{ij})j_{n}(ka_j)\mathcal{G}(l,m;q,s-m;n)}{h^{(1)}_l(ka_i)}
\overline{Y_q^{s-m}}(\hat{\bs b}_{ij}),\quad q=0, 1, \cdots.
\end{equation}
We see that the  first kind spherical Hankel functions of different orders and arguments and quite complicated Gaunt coefficients are involved in \eqref{wavefunsphcoefanalyform}. According to the asymptotic formula \eqref{besselasymptotic}, direct computation of the first kind spherical Hankel function may lead to underflow/overflow for large order. Moreover, the Gaunt coefficients $\mathcal{G}(l,m;q,s-m;n)$ are very complicated and require sophisticated algorithm (cf. \cite{xu1996fast}) to compute. Therefore, a robust and accurate numerical algorithm for computing  \eqref{sphcoefwavefunc} is indispensable for highly accurate simulation.

A recurrence formula for the computation of $S_{nl}^{sm}(\bs b_{ij})$ was presented in \cite{chew1992recurrence,gumerov2004recursions}.  However, the calculation  of $S_{nl}^{sm}(\bs b_{ij})$ may cause underflow/overflow for large $n, l$ (see Table \ref{tabcubic4} below).  Thus, the procedure  that first calculates $S_{nl}^{sm}(\bs b_{ij})$ and then multiplies  it by the T-matrix, could result in very large errors for high modes. It is important to  note that we actually only need to compute the coefficients $\{\Psi_{n l}^{s m}(\bs b_{ij})\}$ in the multiscattering problems. Therefore, based on the recurrence formula for $S_{nl}^{sm}(\bs b_{ij})$, we derive the recurrence formulas for $\Psi_{n l}^{s m}(\bs b_{ij})$ as follows. First, we use 
\begin{equation}\label{sectorialvalformnormalize}
\begin{split}
b_{m+1}^{-m-1}&\Psi_{n,m+1}^{s,m+1}(\bs b_{ij})=b_n^{-s}\frac{\beta_m^i}{\alpha_{n-1}^j}\Psi_{n-1,m}^{s-1,m}(\bs b_{ij})-b_{n+1}^{s-1}\alpha_n^j\beta_m^i\Psi_{n+1,m}^{s-1,m}(\bs b_{ij}),\\
b_{m+1}^{-m-1}&\Psi_{n,m+1}^{s,-m-1}(\bs b_{ij})=b_n^{s}\frac{\beta_m^i}{\alpha_{n-1}^j}\Psi_{n-1,m}^{s+1,-m}(\bs b_{ij})-b_{n+1}^{-s-1}\alpha_n^j\beta_m^i\Psi_{n+1,m}^{s+1,-m}(\bs b_{ij}),
\end{split}
\end{equation}
to compute the so-called sectorial coefficients where
\begin{equation}\label{coefbnm}
\alpha_n^j=\frac{j_n(ka_j)}{j_{n+1}(ka_j)},\quad \beta_m^i=\frac{h_m(ka_i)}{h_{m+1}(ka_i)},\quad b_n^m=\begin{cases}
\sqrt{\frac{(n-m-1)(n-m)}{(2n-1)(2n+1)}}, & 0\leq m\leq n,\\[4pt]
-\sqrt{\frac{(n-m-1)(n-m)}{(2n-1)(2n+1)}}, & -n\leq m\leq 0,\\[4pt]
0, & |m|>n.
\end{cases}
\end{equation}
The recurrence process start with 
\begin{equation}\label{initialvalnormalize}
\begin{split}
\Psi_{n 0}^{s 0}(\bs b_{ij})=&(-1)^{n}\sqrt{4\pi}\frac{j_n(ka_j)h_n(kb_{ij})}{h_{0}(ka_i)}Y_n^{-s}(\widehat{\bs b}_{ij}),\\
\Psi_{0 l}^{0 m}(\bs b_{ij})=&\sqrt{4\pi}\frac{j_0(ka_j)h_l(kb_{ij})}{h_{l}(ka_i)}Y_l^{m}(\widehat{\bs b}_{ij}).
\end{split}
\end{equation}
Then all other coefficients are computed by 
\begin{equation}\label{recurenceformula2normalize}
\begin{split}
a_l^m\Psi_{n,l+1}^{sm}(\bs b_{ij})=&a_{l-1}^m\beta_{l-1}^i\beta_l^i\Psi_{n,l-1}^{sm}(\bs b_{ij})-a_n^s\alpha_n^j\beta_l^i\Psi_{n+1,l}^{sm}(\bs b_{ij})+a_{n-1}^s\frac{\beta_l^i}{\alpha_{n-1}^j}\Psi_{n-1,l}^{sm}(\bs b_{ij}),
\end{split}
\end{equation}
where
\begin{equation}\label{coefanm}
a_{n}^m=\begin{cases}
\sqrt{\frac{(n+1+|m|)(n+1-|m|)}{(2n+1)(2n+3)}},\;\; & n\geq |m|,\\[4pt]
0, & |m|>n.
\end{cases}
\end{equation}

We compare the behaviour of  $S_{nl}^{sm}(\bs b_{ij})$ and $\Psi_{nl}^{sm}(\bs b_{ij})$, and present 
their values obtained by recurrence formula in Table \ref{tabcubic4}. It  shows that it is infeasible  to use $S_{nl}^{sm}(\bs b_{ij})$ in the  numerical simulation.
\begin{table}[h]
	\caption{A comparison of $S_{nl}^{sm}(\bs b_{ij})$ and $\Psi_{nl}^{sm}(\bs b_{ij})$ for $k=90$, $\bs b_{ij}=(0.5, 0, 0)$, $a_i=a_j=0.15$.}
	\begin{center}
		\begin{tabular}{|r|r|r|}
			\hline
			$s$ & $S_{90,90}^{s,0}(\bs b_{ij})\qquad\qquad$ & $\Psi_{90,90}^{s,0}(\bs b_{ij})\qquad\qquad$ \\
			\hline
			0 & 0.02655 - 1.18586e+84i & 7.78972e-43 - 1.39298e-58i\\
			\hline
			1 & -9.99986e-18 + 1.30687e+70i & -8.58424e-57 + 8.75699e-73i\\
			\hline
			2 & 0.00052 + 1.17247e+84i & -7.70172e-43 + 1.49713e-58i\\
			\hline
			3 & -2.75346e-17 - 1.27731e+70i & 8.39038e-57 - 1.47936e-72i\\
			\hline
			4 & 0.02106 - 1.13316e+84i & 7.44350e-43 - 4.30398e-58i\\
			\hline
		\end{tabular}
	\end{center}
	\label{tabcubic4}
\end{table}

By matching the boundary data on $\partial D_i$, we obtain linear system
\begin{equation}\label{ABsystem}
\begin{bmatrix}
\mathbb I & \mathbb S^{21} & \mathbb S^{31} & \cdots & \mathbb S^{M1}\\
\mathbb S^{12} & \mathbb I & \mathbb S^{32} & \cdots & \mathbb S^{M2}\\
\vdots & \vdots & \vdots & \ddots & \vdots\\
\mathbb S^{1M} & \mathbb S^{2M} & \mathbb S^{3M} & \cdots & \mathbb{I}
\end{bmatrix}\begin{bmatrix}
\bs A^1\\
\bs A^2\\
\vdots\\
\bs A^M
\end{bmatrix}=\begin{bmatrix}
\bs G^1\\
\bs G^2\\
\vdots\\
\bs G^M
\end{bmatrix}
\end{equation}
where $\bs G^{i}=(\widehat{g}_{lm}^{i})$ are vectors consisting  of spherical harmonic expansion coefficients of $g_N$ on  $\partial D_i$, respectively. By using the Lapack to directly solve this linear system, we are able to solve a multiple scattering problem with a  limited  number of scatterers. We  report the simulation result  
of  $25$ scatterers at the end of this subsection. It is worthwhile to point out that  our expansion approach can be adopted not only to other fast numerical methods (e.g., fast multipole method \cite{koc1998calculation,gumerov2005computation}, and the iterative method \cite{hamid1991iterative}) to deal with many more spherical scatterers but also to other appropriate solvers (e.g., finite difference method in curvilinear coordinates \cite{acosta2010coupling}) to simulate multiple scattering with irregular scatterers.

%
%
%
%
%
%
%
%

 Next, we present some numerical results. We set two spherical scatterers located at $(-1, 0, 0)$ and $(1, 0, 0)$ with radius $a_1=a_2=0.25$, respectively. As in the last example 1, the incident wave is chosen to be the spectral element interpolations of a plane wave $e^{\ri k\hat{\bs k}\cdot\bs x}$ and a spherical wave $e^{\ri k|\bs x-\bs x_0|}$, respectively. We set $\hat{\bs k}=(0, 0, 1)$, ${\bs x}_0=(0, 0, 1)$, $k=50$ and use a uniform $3\times 4$ mesh in the $\theta-\varphi$ plane for the spherical surface and polynomials of degree $N=50$ for the interpolation. According to our computation, the interpolation errors for $e^{\ri k\hat{\bs k}\cdot\bs x}$ and $e^{\ri k|\bs x-\bs x_0|}$ in a discrete maximum norm are $9.4565e-15$ and $4.5760e-14$, respectively. 
Define
$$\widehat{g}_{l}=\max\limits_{0\leq |m|\leq l}\big\{|\widehat{g}_{lm}^{1}|, |\widehat{g}_{lm}^{2}|\big\},\quad C_{l}=\max\limits_{0\leq |m|\leq l}\big\{|{A}_{lm}^1|,|{B}_{lm}^2|\big\},$$
and list the computed values for $35\leq l\leq 40$ in Table \ref{tabcubic2}. These results show that it is enough to set the truncation mode $L=40$ for $k=50$. We depict the scattering waves in Figure \ref{multiscatter01-01}. 
\begin{table}[h]
	\caption{Magnitude of spherical harmonic coefficients (two spherical scatterers).}
	\begin{center}
		\begin{tabular}{|c|c|c|c|c|}
			\hline
			\multirow{2}{*}{$l$} & \multicolumn{2}{|c|}{plane incident wave} & \multicolumn{2}{|c|}{spherical incident wave}\\
			\cline{2-5}
			&  $\widehat{g}_{l}$  & $C_l$ & $\widehat{g}_{l}$ & $C_l$\\
			\hline
			35 &  1.0371e-12  &  1.0366e-12    & 2.6723e-13  &  2.6756e-13  \\
			\hline
			36 &  1.8550e-13  &  1.8483e-13    & 4.7839e-14  &  4.7969e-14  \\
			\hline
			37 &  3.2006e-14  &  3.1935e-14    & 8.1178e-15  &  8.0802e-15    \\
			\hline
			38 &  5.5439e-15  &  5.5707e-15    & 1.3823e-15  &  1.3574e-15   \\
			\hline
			39 &  9.3425e-16  &  9.3619e-16    & 4.6158e-16  &  4.6883e-16   \\
			\hline
			40 &  4.6226e-16  &  4.7945e-16    & 5.4598e-16  &  5.5918e-16    \\
			\hline		
		\end{tabular}
	\end{center}
	\label{tabcubic2}
\end{table}

\begin{figure}[!ht]
	\begin{center}
		\subfigure[Real part (with plane incident wave)]{\includegraphics[scale=0.31]{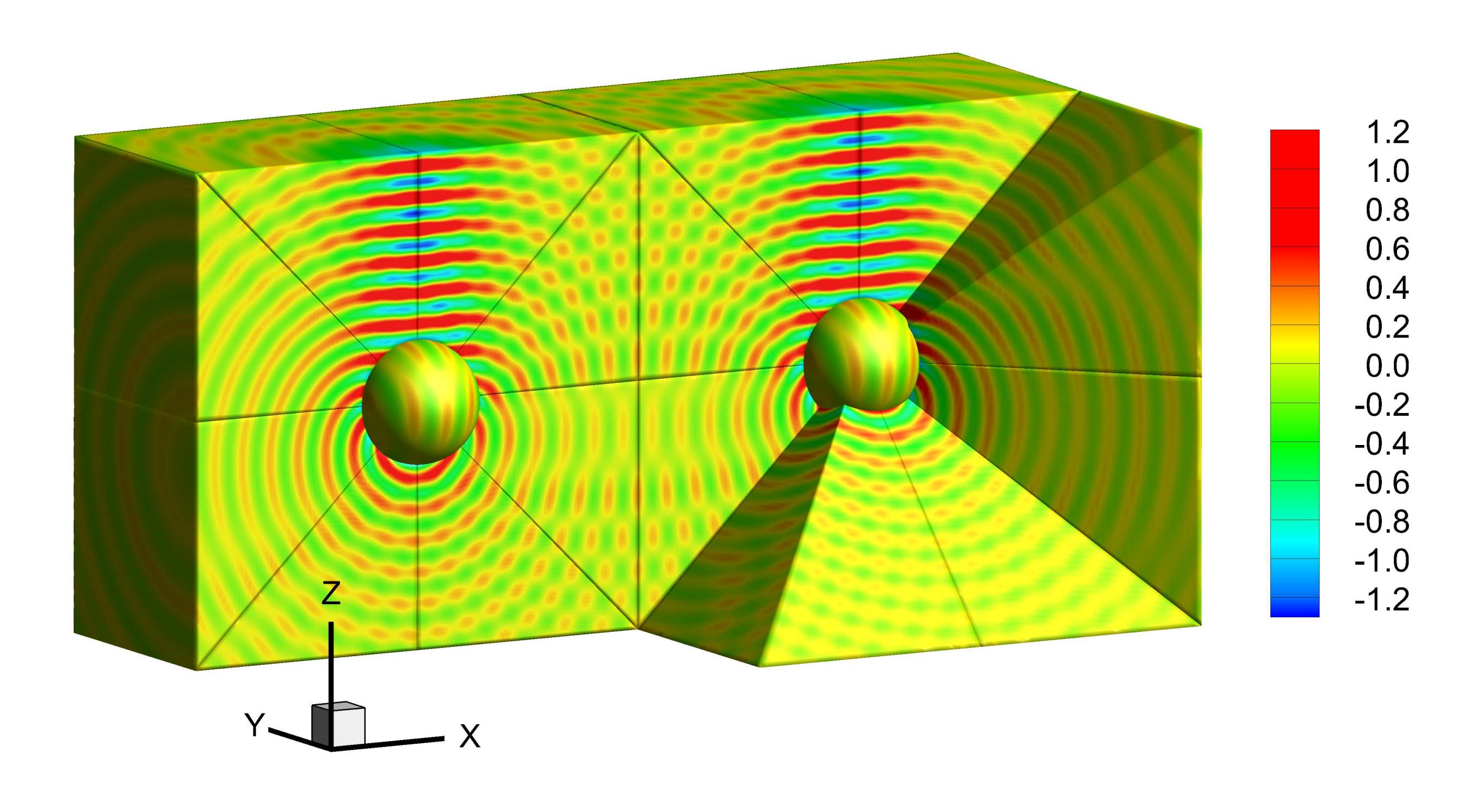}}
		\subfigure[Imaginary part (with plane incident wave)]{\includegraphics[scale=0.31]{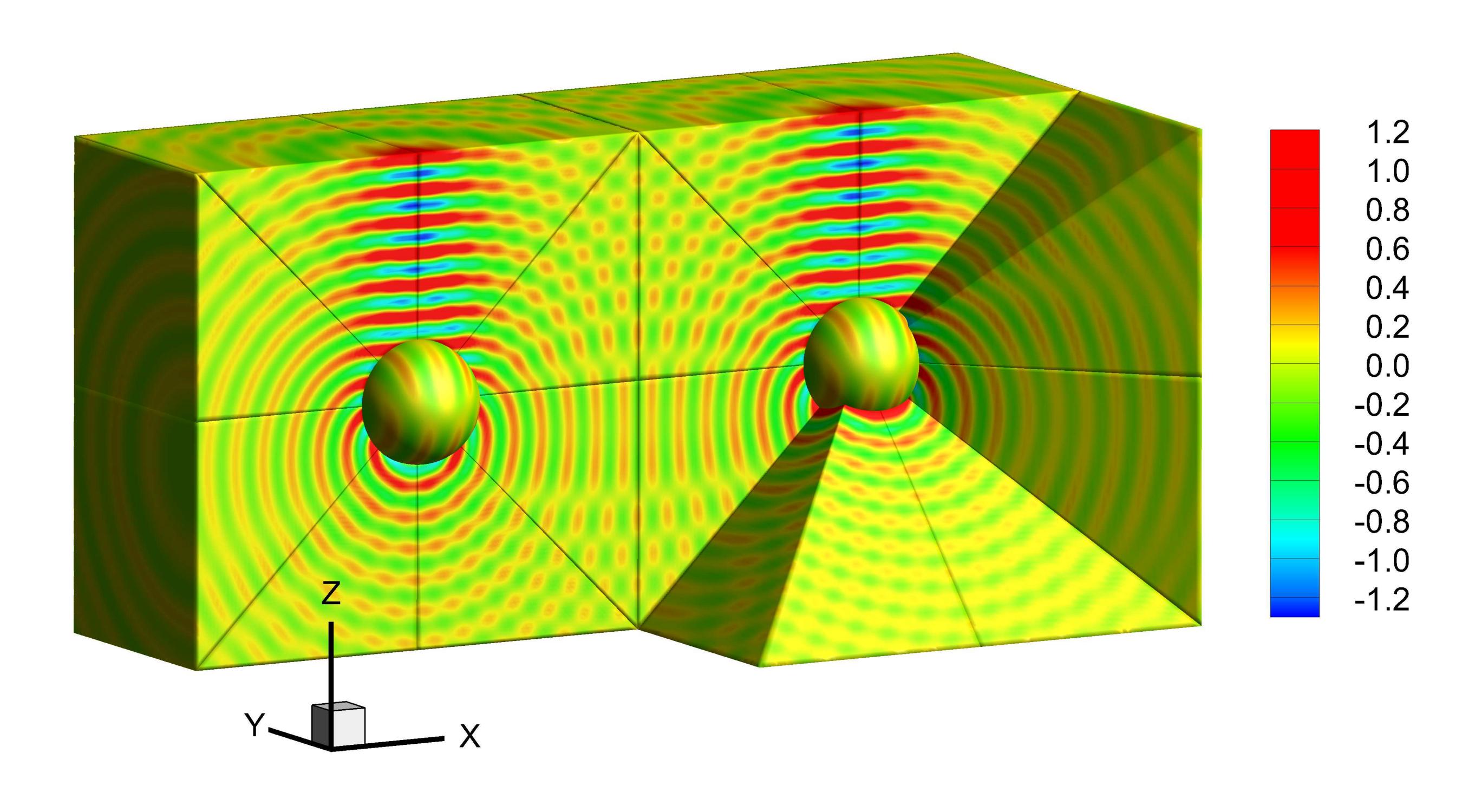}}
		\subfigure[Real part (with spherical incident wave)]{\includegraphics[scale=0.31]{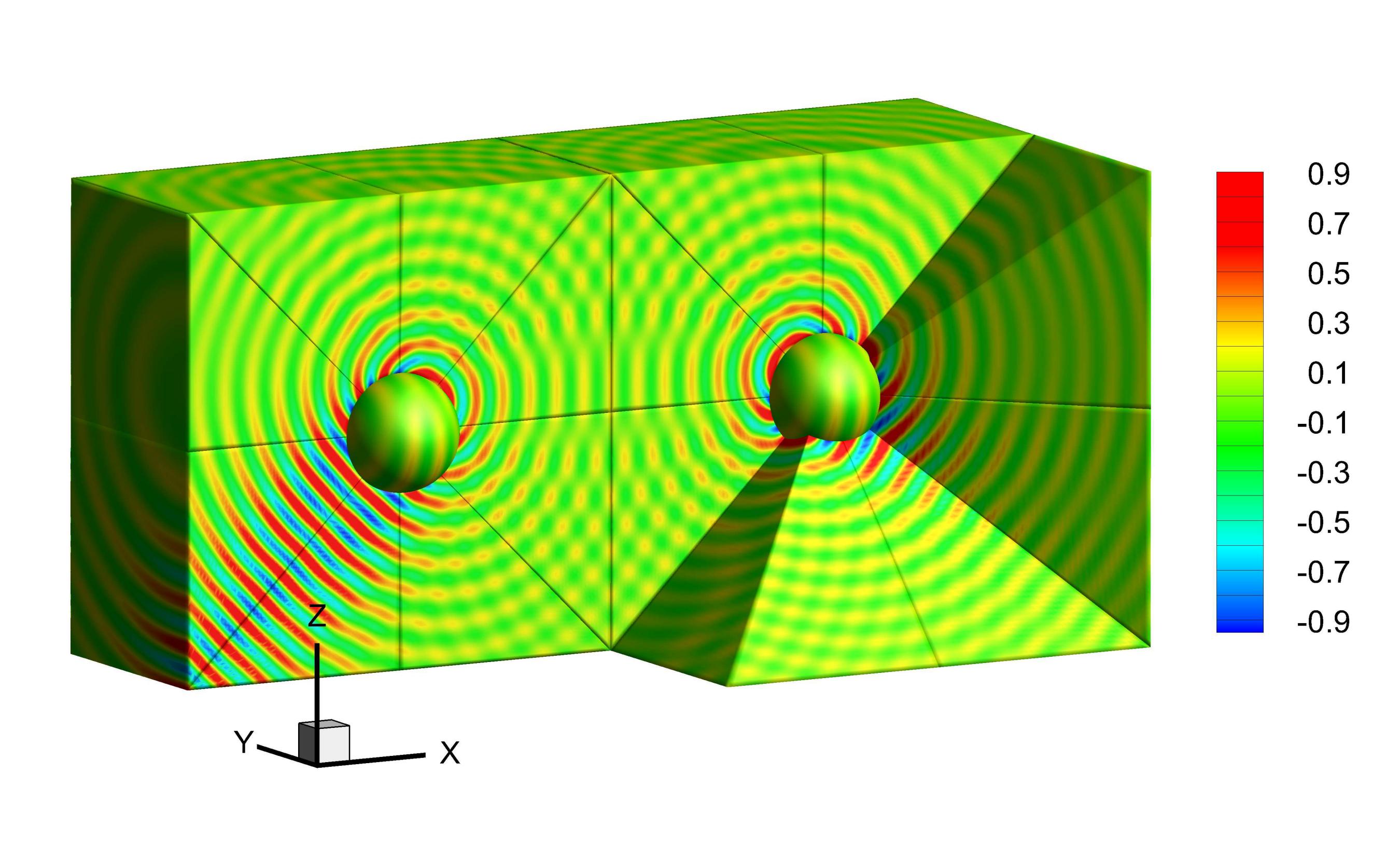}}
		\subfigure[Imaginary part (with spherical incident wave)]{\includegraphics[scale=0.31]{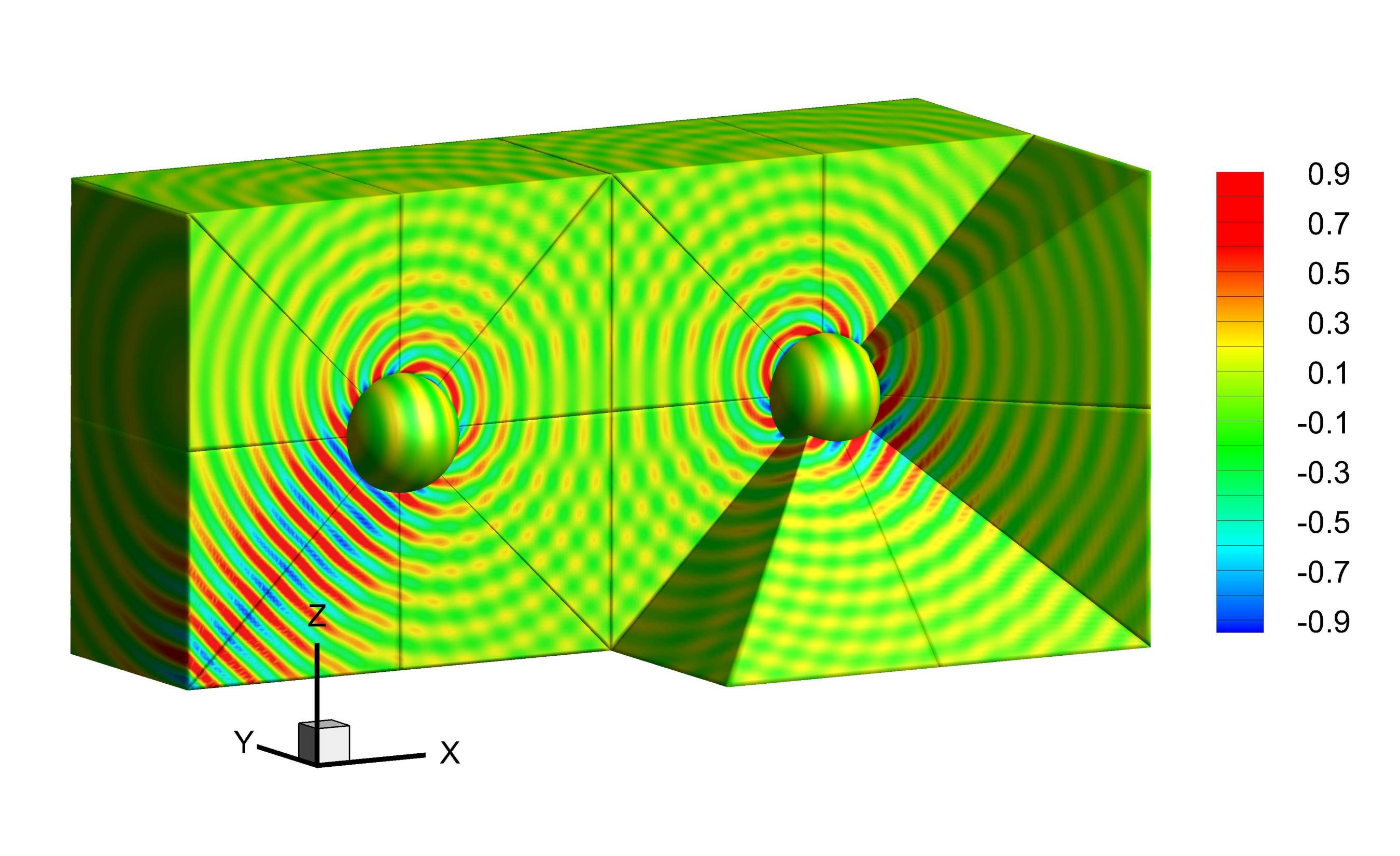}}
		\caption{Scattering waves from two spherical scatterers.}
		\label{multiscatter01-01}
	\end{center}
\end{figure}
\begin{table}[h]
	\caption{Magnitude of spherical harmonic coefficients (25 spherical scatterers).}
	\begin{center}
		\begin{tabular}{|c|c|c|c|c|c|}
			\hline
			$l$ & $\widehat{g}_{l}$  & $C_l$  & $l$ & $\widehat{g}_{l}$  & $C_l$ \\
			\hline
			25 &  1.4501e-9  &  1.8695e-9   & 28 &  6.6724e-12  &  6.9720e-12     \\
			\hline
			26 &  2.5081e-10  &  2.8347e-10    & 29 &  1.0286e-12  &  8.7441e-13      \\
			\hline
			27 & 4.1688e-11  &  4.2603e-11  & 30 &  1.5652e-13  &  1.3781e-13        \\
			\hline		
		\end{tabular}
	\end{center}
	\label{tabcubic3}
\end{table}
We then test the algorithm for many scatterers. We set $25$ spherical scatterers located at a uniform grid points in square area in $x$-$z$ plane delimited by points $(-2, 0, -2)$ and $(2, 0, 2)$. All scatterers are identical and of radius $a=0.25$. The incident wave is chosen to be the spectral element interpolations of a plane wave $e^{\ri k\hat{\bs k}\cdot\bs x}$. We set $\hat{\bs k}=(0, 0, 1)$, $k=35$ and use a uniform $3\times 4$ mesh in the $\theta$-$\varphi$ plane for the spherical surface and polynomials of degree $N=30$ for the interpolation. According to our computation, the interpolation errors for $e^{\ri k\hat{\bs k}\cdot\bs x}$ in a discrete maximum norm are $2.5367e-14$ . 
Define
$$\widehat{g}_{l}=\max\limits_{0\leq |m|\leq l}\big\{|\widehat{g}_{lm}^{1}|, \cdots, |\widehat{g}_{lm}^{25}|\big\},\quad C_{l}=\max\limits_{0\leq |m|\leq l}\big\{|{A}_{lm}^1|,\cdots, |{A}_{lm}^{25}|\big\},$$
and list the computed values from $l=25$ to $l=30$ in Table \ref{tabcubic3}. The results show that very accurate solution can be obtained by setting the truncation mode $L=30$ for $k=35$ in this example. We depict the scattering waves in Figure \ref{multiscatter01-02}. 

\begin{figure}[!ht]
	\begin{center}
		\subfigure[Real part (with plane incident wave)]{\includegraphics[scale=0.36]{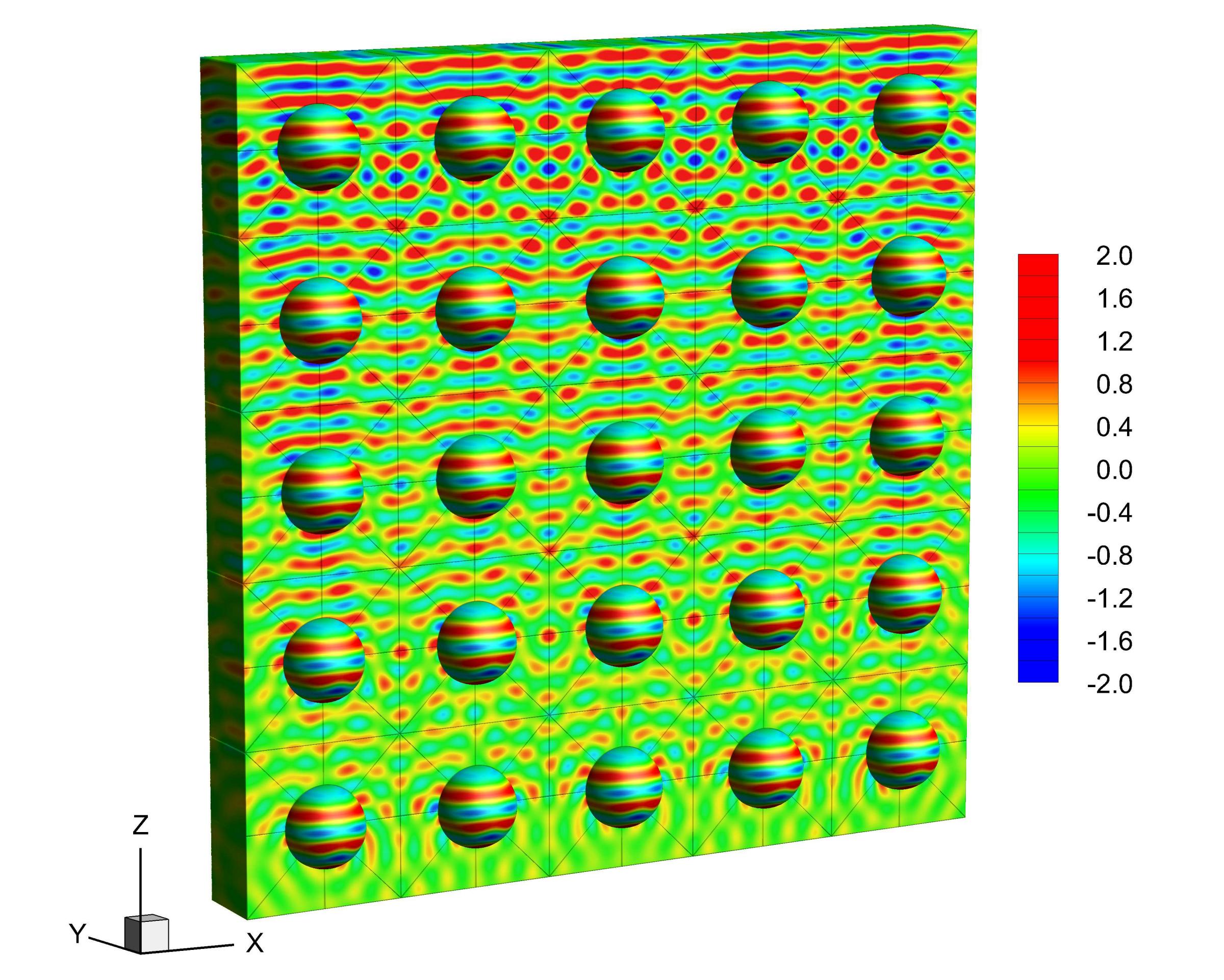}}
		\subfigure[Imaginary part (with plane incident wave)]{\includegraphics[scale=0.35]{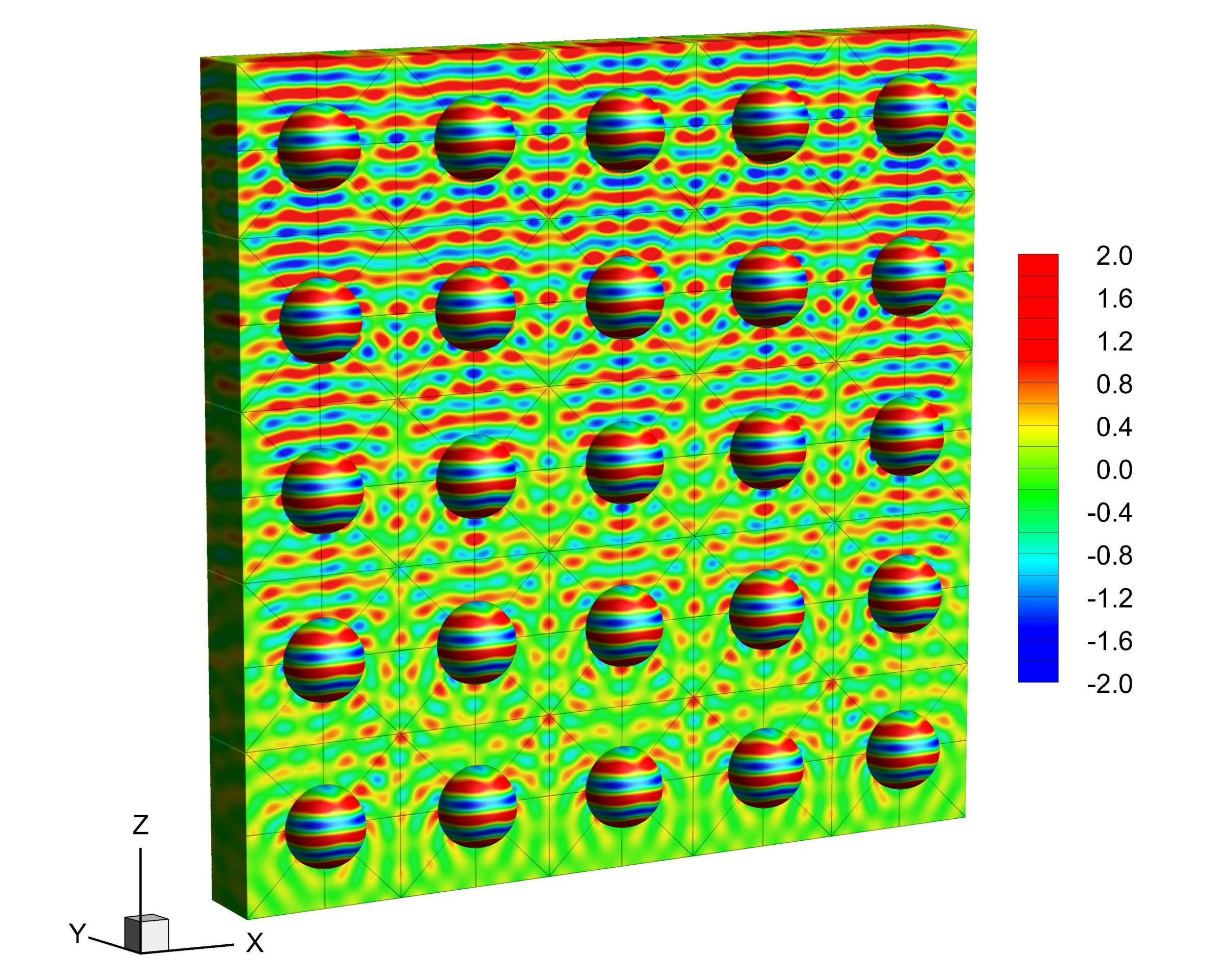}}	
		\caption{Scattering waves from many spherical scatterers.}
		\label{multiscatter01-02}
	\end{center}
\end{figure}

\subsection{Concluding remarks}
In this paper, we presented accurate  numerical tools for computing  SPH/VSH expansions of scalars/vector fields
with given nodal values of spectral element approximations. With the aid of some analytic  formulas,
we could accurately calculate the underlying highly oscillatory integrals so it is particularly robust for high mode expansions. As direct applications of the algorithms, we considered several 3D scattering problems, and demonstrated that stable and accurate simulations can be achieved with afford computational cost.
It is expected that the development in this paper paves the way for the simulation of scattering problems with complex scatterers using the exact DtN boundary conditions. We leave the integration of the tool herein with the interior spectral-element solver for challenging 3D simulations to a future work.


\end{document}